\documentclass[10pt,english]{amsart}
\usepackage{xcolor}
\usepackage{amsfonts}
\usepackage{amsmath, amsfonts, amsthm, amssymb, amscd}
\usepackage[mathcal]{euscript}
\usepackage{mathrsfs}
\usepackage{a4wide}

\usepackage{dsfont}
\usepackage{extarrows}

\newtheorem{theorem}{Theorem}[section]
\newtheorem{proposition}[theorem]{Proposition}
\newtheorem{lemma}[theorem]{Lemma}

\newtheorem{definition}[theorem]{Definition}

\numberwithin{equation}{section}

\newcommand \del \partial

\newcommand \RR {\mathbb{R}}

\newcommand \vep {\varepsilon}

\newcommand \Hcal{\mathcal{H}}
\newcommand \Hb{\bar{\Hcal}}
\newcommand \Tcal{\mathcal{T}}
\newcommand \Pcal{\mathcal{P}}
\newcommand \Fcal{\mathcal{F}}
\newcommand \Fcale{\mathcal \Fcal^E}
\newcommand \Fcali{\mathcal \Fcal^I}

\newcommand \Kcal{\mathcal{K}}
\newcommand \Ecal{\mathcal{E}}

\newcommand \delb{\bar{\del}}

\newcommand \Phib{\bar{\Phi}}
\newcommand \Psib{\bar{\Psi}}
\newcommand \Tb{\overline{T}}
\newcommand \Qb{\overline{Q}}

\newcommand \delu{\underline{\del}}
\newcommand \Phiu{\underline{\Phi}}
\newcommand \Psiu{\underline{\Psi}}

\newcommand \Nu{\underline{N}}

\newcommand \Tu{\underline{T}}
\newcommand \Qu{\underline{Q}}
\newcommand \miu{\underline{m}}

\newcommand \delt{\tilde{\del}}
\newcommand \Phit{\tilde{\Phi}}
\newcommand \Psit{\tilde{\Psi}}

\newcommand \Tt{\widetilde{T}}
\newcommand \Qt{\widetilde{Q}}
\newcommand \mt{\tilde{m}}
\newcommand \Nt{\widetilde{N}}

\newcommand \Zscr {\mathscr{Z}}
\newcommand \Zint{\Zscr^{\text{int}}}
\newcommand \Zext{\Zscr^{\text{ext}}}

\newcommand \EE {\Ecal^{\text{E}}}
\newcommand \EH {\Ecal^{\text{H}}}

\title{Global solutions of non-linear wave-Klein-Gordon system in one space dimension}

\author{Yue MA}
\begin{document}

\maketitle
\section{Introduction}

In the present work we will make a generalization in $\RR^{1+1}$ of the {\sl hyperboloidal foliation method} in order to remove the restriction on the support of the initial data. Then we will make a first application on the following model problem:
\begin{equation}\label{eq main}
\left\{
\aligned
&\Box u  = v^3
\\
&\Box v + c^2v = N^{\alpha\beta}\del_{\alpha}u\del_{\beta}u
\endaligned
\right.
\end{equation}
in $\RR^{1+1}$ with initial data
\begin{equation}\label{eq main-initial}
u(1,x) = u_0(x),\quad \del_t u(1,x) = u_1(x),\quad v(1,x) = v_0(x),\quad \del_t v(1,x) = v_1(x).
\end{equation}

The problem on the global behavior of the small regular solutions to hyperbolic equations or systems has attracted lots of attention in the past. After the pioneer works on nonlinear wave equations$\backslash$systems (e.g. \cite{Kl1}) and Klein-Gordon equations$\backslash$systems (e.g. \cite{Kl2}, \cite{Shatah85} etc.), people begin to be curious on the systems composed by wave and Klein-Gordon equations, which come naturally form the Einstein-massive scalar field system, $f(R)-$ theory of gravity, Maxwell-Klein-Gordon system, etc. The main difficulty on this type of system, compared with the pure wave or Klein-Gordon systems, is the lack of symmetry. In general, the conformal Killing vector filed $S = x^{\alpha}\del_{\alpha}$ of the linear wave operator is no longer conformal Killing with respect to the linear Klein-Gordon operator. This prevents any possibility of na\"ive combination of the methods for wave equations with those for Klein-Gordon equations. However, the pioneer work of S. Katayama \cite{Ka} point out that the wave-Klein-Gordon system is also globally stable with some reasonable restrictions on the nonlinearities.

The application of the hyperbolic hyper-surface in the analysis of nonlinear hyperbolic equations was introduced by S. Klainerman in \cite{Kl2} on Klein-Gordon equations in $\RR^{3+1}$ (see also \cite{Ho1}, section 7.6 and 7.7) and then developed in many context (e.g. \cite{Dfx} in $\RR^{2+1}$). It is then introduced in the analysis of the wave-Klein-Gordon system in \cite{M0} and developed in \cite{LM1}, \cite{LM2}, \cite{LM3}, see also \cite{Q.Wang2016}.

The essential point of this method is to foliate the inner part of the light-cone $\{t>r+1\}$ by $\Hcal_s = \{t = \sqrt{s^2+r^2}\}$. These hyperbolic hyper-surfaces will take the role of the time-constant hyperplanes on which people do energy estimates, global Sobolev's inequalities, etc. The restriction of this method is obvious, it can only handle the inner part of the light-cone, thus can only treat the initial data with compact support (by finite speed of propagation, the associated local solution is supported in $\{t>r+1\}$ after a time translation).

In order to remove this restriction on support, there is \cite{Ionescu-Pausader} in which a method based on Fourier analysis was introduced.  We also remark that the original method of Katayama does not demand the condition on support of the initial data.

In the present work we give an alternative approach, which is a generalization of the hyperboloidal foliation. The main new observation is to foliate the half space-time $\RR^+_t \times\RR_x$ by a family of specially constructed space-like curves ( called {\sl combined curves}) and establish energy estimate on these curves. A combined curve, just as its name implies, is a combination of an arc of a hyperbola together two half-lines. More precisely, given the canonical cartesian coordinates $\{t,x\}$, we consider the translated light-cone $\{t>|x|+1\}$. For a given hyperbola $\Hcal_s := \{t = \sqrt{|x|^2+s^2}\}$, we consider its inner part $\Hcal_s^*\cap \{t>|x|+1\}$ smoothly joined with two time-constant half-lines towards the spatial infinity out side of this cone. This construction is a combination of the hyperboloidal foliation of the cone $\{t>|x|+1\}$ with the standard time-constant foliation outside of this cone.

We will observe that, as in the classical case where we make energy estimates on the lines $\{t=\text{constant}\}$ or hyperbolae, the energy on combined curves also controls sufficient $L^2$ norms on the solution and/or  its gradient. We then proceed as before by establishing global Sobolev inequalities and other parallel analytical tools. With this {\sl combined foliation}, we can almost entirely remove the restriction on the support of the initial data (i.e., one only demands certain decreasing rate on the initial data at spatial infinity, according to different nonlinearities coupled in the system).

This construction can be generalized in higher dimensional case in $\RR^{1+3}$ with some additional non-trivial observations concerning rotation-invariance (see in \cite{LM2017}). There is also other constructions to generalize the hyperboloidal foliation method, see in detail \cite{K-W-Y-2018}.


The present work is composed of two parts. In part 1 (section \ref{sec foliation-construction} to section \ref{sec decay-ext}), we make a detailed description on the combined foliation method, especially on the construction of the foliation (section \ref{sec foliation-construction}), the global Sobolev inequalities (section \ref{sec sobolev}) and the decay estimates in transition-exterior region (section \ref{sec decay-ext}).

In part 2
(section \ref{sec bootstrap} to section \ref{sec conclusion}), we analyse the model problem \eqref{eq main} and establish the following main result:
\begin{theorem}\label{thm main}
Let $N\geq 9$  be an integer and $\vep>0$. Suppose that the following smallness conditions hold for the initial data:
\begin{equation}\label{eq main-smallness}
\|(1+r)^{\gamma}\del_x\del^IL^j u_0(\cdot,1)\|_{L^2(\RR)} + \|(1+r)^{\gamma}\del^IL^j u_1(\cdot,1)\|_{L^2(\RR)}\leq \vep,\quad |I|+j\leq N.\footnote{$L = t\del_x + x\del_t$. See in detail in the following sections.}
\end{equation}
Then when $\vep$ sufficiently small, the corresponding local solution extends to time infinity.
\end{theorem}
In one space dimension, there are already plenty of results on the global behavior of wave equations and Klein-Gordon equations (e.g. \cite{Yang-Yu-2017}, \cite{Delort-2014}, \cite{Stingo-2015} etc). Through this result, we observe that the global stability is also expectable for certain nonlinearities for wave-Klein-Gordon system. 

\part{The Combined foliation framework}
\section{Construction of the combined foliation}\label{sec foliation-construction}
\subsection{Basic notation}
We are working in $\RR^+_t\times \RR_x$. For $x\in\RR_x$ we denote by $r = |x|$. We define the $C^{\infty}_c$ function on $\RR$:
$$
\rho(x) := \left\{
\aligned
&0,\quad r\geq 1/2,
\\
&e^{\frac{4}{4x^2-1}},\quad r<1/2.
\endaligned
\right.
$$
Remark that $\rho$ is in $C_c^{\infty}(\RR)$ thus
$$
0<\int_{\RR}\rho(x)dx <+\infty.
$$
We define
\begin{equation}\label{eq 1 cu-off}
\chi(x) := \frac{\int_{-\infty}^{x+1/2}\rho(y)dy}{\int_\RR \rho(y)dy}
\end{equation}
which is also a $C^{\infty}$ function and satisfies the following condition:
$$
\chi(x) =0 ,\quad  x\leq0, \quad \chi(x) = 1, \quad x\geq 1,\quad \text{and}\quad \chi'(x)>0, \quad 0< x< 1.
$$
We define
$$
\xi_s(r) := 1 - \chi(r-(s^2-1)/2), \quad  r\geq 0
$$
which is a $C^{\infty}$ function defined on $\RR^+$ satisfies the following properties:
$$
\xi_s(r) = 1, \quad r\leq \frac{s^2-1}{2}, \quad \xi_s(r) = 0,\quad r\geq \frac{s^2+1}{2}
$$
and
\begin{equation}\label{eq 1 xi'}
\xi_s'(x)<0,\quad \frac{s^2-1}{2}<x<\frac{s^2+1}{2}.
\end{equation}
Also,
$$
\del_s(\xi_s)(x) = s\chi'(x - (s^2-1)/2) = -s\xi_s'(x).
$$

For $s\geq 2$, we define a family of curves via the following ODE:
$$
T(s,0) := s,\quad \del_xT(s,x) = \frac{\xi_s(r)x}{\sqrt{x^2+s^2}}.
$$
and we denote by
$$
\Fcal_s : = \{(T(s,x),x)\in \RR^2|x\in\RR\}.
$$
We remark that $\Fcal_s$ are $C^{\infty}$ curves and symmetric with respect to the axis $\{x=0\}$. Furthermore, for $x\geq 0$
\begin{equation}\label{eq 1 feuille}
T(s,x) = \left\{
\aligned
&\sqrt{x^2+s^2},&&0\leq x\leq \frac{s^2-1}{2},
\\
& s + \int_0^{x}\frac{\xi_s(r)y}{\sqrt{s^2+y^2}}dy, &&\frac{s^2-1}{2}<x< \frac{s^2+1}{2},
\\
&T(s) = s + \int_0^{\frac{s^2+1}{2}}\frac{\xi_s(r)y}{\sqrt{s^2+y^2}}dy, &&\frac{s^2+1}{2}\leq x <\infty.
\endaligned
\right.
\end{equation}

For the convenience of discussion, we introduce the following notation:
$$
\Fcal_{[s_0,s_1]} = \{(t,x)|T(s_0,x)\leq t\leq T(s_1,x)\}, \quad \Fcal_{[s_0,+\infty)} = \{(t,x)|T(s_0,x)\leq t\}
$$
the region limited by one or two such curves. We write
$$
\Fcal_{[2,+\infty)} = \bigcup_{s\geq 2}\Fcal_s
$$
which is an one-parameter foliation of the region $\Fcal_{[2,+\infty)}$,  called the {\sl combined foliation}. Then we state the following result:
\begin{proposition}\label{prop 1 feuille}
$\Fcal_s$ are $C^{\infty}$ curves. For $|x|\leq \frac{s^2-1}{2}$,
$$
T(s,x) = \sqrt{x^2+s^2}.
$$
For $|x|\geq \frac{s^2+1}{2}$, $T(s,x)$ is constant with respect to $x$ and
\begin{equation}\label{eq 1 prop 1 feuille}
\frac{s^2+1}{2}\leq T(s,x)\leq \frac{\sqrt{s^4+6s^2+1}}{2}.
\end{equation}
\end{proposition}
\begin{proof}
We only need to prove \eqref{eq 1 prop 1 feuille}, and this is direct by \eqref{eq 1 feuille}.
\end{proof}

We list out some geometric facts on the curve $\Fcal_s$. The normal vector (with respect to Euclidian metric) of $\Fcal_s$ is
\begin{equation}\label{eq normal}
\vec{n} =
\left\{
\aligned
&\frac{1}{\sqrt{t^2+x^2}}(t,-x) = \frac{\sqrt{s^2+x^2}}{\sqrt{s^2+2x^2}}\Big(1,\frac{-x}{\sqrt{s^2+x^2}}\Big),\quad |x|\leq \frac{s^2-1}{2},
\\
&\frac{\sqrt{s^2+x^2}}{\sqrt{s^2+(1+\xi_s(r))x^2}}\Big(1,\frac{-\xi_s(|x|)x}{\sqrt{s^2+x^2}}\Big),\quad \frac{s^2-1}{2}\leq |x|\leq \frac{s^2+1}{2},
\\
&(1,0),\quad  \frac{s^2+1}{2}\leq |x|.
\endaligned
\right.
\end{equation}
The volume element of $\Fcal_s$ (viewed as a surface) is
\begin{equation}\label{eq volume}
d\sigma = \sqrt{1+|\del_xT|^2} =
\left\{
\aligned
& \frac{\sqrt{x^2+t^2}}{t} =\frac{\sqrt{s^2+2x^2}}{\sqrt{s^2+x^2}},\quad |x|\leq \frac{s^2-1}{2},
\\
& \frac{\sqrt{s^2+(1+\xi_s^2(|x|))x^2}}{\sqrt{s^2+x^2}}, \quad \frac{s^2-1}{2}\leq |x|\leq \frac{x^2+1}{2},
\\
& 1,\quad  \frac{s^2+1}{2}\leq |x|.
\endaligned
\right.
\end{equation}

For the convenience of discussion, the curve $\Fcal_s$ is divided into three pieces:
$$
\Hcal^*_s := \left\{(T(s,x),x)\Big|0\leq |x|\leq (s^2-1)/2\right\}, \quad \Tcal_s: = \left\{(T(s,r),x)\Big|(s^2-1)/2\leq r\leq (s^2+1)/2\right\}
$$
and
$$
\Pcal_s:=\left\{(T(s,r),x)\Big||x|\geq (s^2+1)/2\right\}.
$$
Remark that the part $\Hcal^*_s$ is a part of the hyperbola with radius $s$, and $\Pcal_s$ is part of the line $\{t=T(s,(s^2+1)/2)\}$. The part $\Tcal_s$ joints the above two pars together in a smooth manner.

We also introduce the following notation
$$
\Hcal^*_{[s_0,s_1]} : = \left\{(t,x)\Big|T(s_0,x)\leq t\leq T(s_1,x), |x|\leq (s^2-1)/2\right\},
$$
$$
\Tcal_{[s_0,s_1]} : = \left\{(t,x)\Big|T(s_0,x)\leq t\leq T(s_1,x), (s^2-1)/2\leq|x|\leq (s^2+1)/2\right\}
$$
and
$$
\Pcal_{[s_0,s_1]} : = \left\{ T(s_0,x)\leq t\leq T(s_1,x), |x|\geq (s^2+1)/2 \right\}.
$$
We also denote by
$$
\Hcal^*_{[s_0,\infty)} : = \left\{(t,x)\Big|T(s_0,x)\leq t, |x|\leq (s^2-1)/2\right\},
$$
$$
\Tcal_{[s_0,\infty)} : = \left\{(t,x)\Big|T(s_0,x)\leq t , (s^2-1)/2\leq|x|\leq (s^2+1)/2\right\}
$$
and
$$
\Pcal_{[s_0,\infty)} : = \left\{ T(s_0,x)\leq t , |x|\geq (s^2+1)/2 \right\}.
$$
These regions are called ``hyperbolic region'', ``transition region'' and ``flat region''.

For the convenience of discussion, we also denote by
$$
\Hb_s := \Tcal_s\cup \Pcal_s, \quad \Hb_{[s_0,s_1]} := \Tcal_{[s_0,s_1]}\cup\Pcal_{[s_0,s_1]}.
$$
which are called ``exterior region''. The region $\Hcal^*_{[s_0,s_1]}$ or $\Hcal^*_{[s_0,\infty]}$ is also called ``interior region''. The frontier between exterior region and interior region
$$
\del\Kcal_{[s_0,s_1]} := \{(s,x)||x| = (s^2-1)/2,\ s_0\leq s\leq s_1\} = \{(t,x)| t = |x|+1,\ t=T(s), s_0\leq s\leq s_1\}.
$$

We remark that $\Hb_s$ has two components of connection, which are $\Hb_s^+ := \{x>0\}\cap \Hb_s$ and $\Hb^-:=\{x<0\}\cap \Hb_s$. Also, $\Hb_{[s_0,s_1]}$ has two components of connection, which are denoted by
$$
\Hb^+_{[s_0,s_1]} = \{r>0\}\cap \Hb_{[s_0,s_1]}, \quad
\Hb^-_{[s_0,s_1]} = \{r<0\}\cap \Hb_{[s_0,s_1]} .
$$
The the above regions, the following bounds hold:
\begin{lemma}\label{lem 2 position}
Let $(t,x)\in \Fcal_{[s_0,s_1]}$ and $s_0\geq 2$, then
\begin{equation}\label{eq 1 lem 2 position}
r \in \left\{
\aligned
\leq& t-1, \quad(t,x)\in \Hcal^*_{[s_0,s_1]},
\\
\in& [t-1,t - c(s)],\quad (t,x)\in\Tcal_{[s_0,s_1]},
\\
\geq& t - c(s),\quad (t,x)\in \Pcal_{[s_0,s_1]}.
\endaligned
\right.
\end{equation}
where $1>c(s)>0$ is determined by the function $\chi$.
\end{lemma}
\begin{proof}
We remark the following calculation: denote by
$$
\lambda(s,r) := T(s,r) - r.
$$
Then
$$
\del_r\lambda(s,r) = \frac{\xi_s(r)r}{\sqrt{s^2+r^2}} - 1<0.
$$
We remark that
$$
T\big(s,(s^2-1)/2\big) = s + \int_0^{(s^2-1)/2}\!\!\!\!\frac{ydy}{\sqrt{s^2+y^2}} =  \frac{s^2+1}{2}.
$$
Now consider a point on $\Fcal_s$. Then $\lambda(s,r)$ is strictly decreasing with respect to $r$. Thus on $\Hcal^*_s$,
$$
\lambda(s,r)\geq \lambda\big(s,(s^2-1)/2\big) = 1.
$$
On $\Tcal_s$,
$$
1\geq \lambda(s,r)\geq   \lambda \big(s,(s^2+1)/2\big).
$$
Remark that
$$
\aligned
\lambda \big(s,(s^2+1)/2\big) =& s + \int_0^{(s^2+1)/2}\frac{\xi_s(r)ydy}{\sqrt{s^2+y^2}} - \frac{s^2+1}{2}
\\
=& s + \int_0^{(s^2+1)/2}\frac{\xi_s(r)ydy}{\sqrt{s^2+y^2}} - s - \int_0^{(s^2-1)/2}\frac{\xi_s(r)ydy}{\sqrt{s^2+y^2}}
\\
=&\int_{(s^2-1)/2}^{(s^2+1)/2}\frac{\xi_s(r)ydy}{\sqrt{s^2+y^2}}.
\endaligned
$$
We also remark that
$$
0< \int_{(s^2-1)/2}^{(s^2+1)/2}\frac{\xi_s(r)ydy}{\sqrt{s^2+y^2}}
< \int_{(s^2-1)/2}^{(s^2+1)/2}\frac{ydy}{\sqrt{s^2+y^2}}
 = \sqrt{s^2+y^2}\bigg|_{(s^2-1)/2}^{(s^2+1)/2}<1.
$$

So we conclude that on $\Tcal_s$,
\begin{equation}\label{eq 1 pr lem 2 proposition}
-1<\lambda \big(s,(s^2+1)/2\big)<0
\end{equation}

And, on $\Pcal_s$, we remark that
$$
\lambda(s,r)\leq \lambda \big(s,(s^2+1)/2\big).
$$

Consider together the above three cases, we conclude \eqref{eq 1 lem 2 position}.
\end{proof}

\subsection{Frames, vector fields}\label{subsec frame}
In $\Fcal_{[2,+\infty)}$, we denote by
$$
\del_0 = \del_t,\quad \del_1 = \del_x.
$$
We introduce the following vector filed:
$$
L := x\del_t+t\del_a
$$
which is called the {\sl Lorentzian boost}.
We also denote by
$$
\delu_x := \frac{x}{t}\del_t + \del_x.
$$

In $\Fcal_{[s_0,\infty)}$ we introduce the following semi-hyperboloidal frame ({\sl SHF} for short) :
$$
\delu_0 := \del_t,\quad \delu_1 = \delu_x.
$$
The transition matrices between SHF and the canonical frame $\{\del_t,\del_x\}$ are:
$$
\Phiu_{\alpha}^{\beta} = \left(
\begin{array}{cc}
1 &0
\\
x/t &1
\end{array}
\right),
\quad
\Psiu_{\alpha}^{\beta} = \left(
\begin{array}{cc}
1 & 0
\\
-x/t &1
\end{array}
\right)
$$
with
$$
\delu_\alpha = \Phiu_{\alpha}^{\beta}\del_\beta,\quad \del_{\alpha} = \Psiu_{\alpha}^{\beta}\delu_{\beta}.
$$

We also introduce the following {\sl null} frame (NF for short) in the region $\Fcal_{[s_0,\infty)}\cap \{r > t/2\}$, defined as following:
$$
\delt_0:= \del_t,\quad \delt_x = \frac{x}{|x|}\del_t + \del_x.
$$
The transition matrices between NF and the canonical frame are:
$$
\Phit_{\alpha}^{\beta} = \left(
\begin{array}{cc}
1 &0
\\
x/r &1
\end{array}
\right),
\quad
\Psiu_{\alpha}^{\beta} = \left(
\begin{array}{cc}
1 & 0
\\
-x/r &1
\end{array}
\right)
$$
with
$$
\delt_\alpha = \Phit_{\alpha}^{\beta}\del_\beta,\quad \del_{\alpha} = \Psit_{\alpha}^{\beta}\delt_{\beta}.
$$

In $\Fcal_{[2,+\infty)}$, we define the following frame (called the {\bf tangent frame}, or {\bf TF} for short):
$$
\delb_0 := \del_t ,\quad \delb_1 := \delb_x = \frac{\xi_s(r)x}{\sqrt{s^2+x^2}}\del_t + \del_x.
$$
The transition matrices between TF and the canonical frame are
$$
\Phib_{\alpha}^{\beta} = \left(
\begin{array}{cc}
1 &0
\\
\frac{\xi_s(r)x}{\sqrt{s^2+x^2}} &1
\end{array}
\right),
\quad
\Psib_{\alpha}^{\beta} = \left(
\begin{array}{cc}
1 &0
\\
\frac{-\xi_s(r)x}{\sqrt{s^2+x^2}} &1
\end{array}
\right),
$$
with
$$
\delb_{\alpha} = \Phib_{\alpha}^{\beta}\del_{\beta}u, \quad \del_{\alpha} = \Psib_{\alpha}^{\beta}\delb_{\beta}.
$$

Let $T$ be a two tensor defined in $\Fcal_{[2,+\infty)}$. Then it can be written in different frame as following:
$$
T = T^{\alpha\beta}\del_{\alpha}\otimes\del_{\beta} = \Tu^{\alpha\beta}\delu_{\alpha}\otimes\delu_{\beta} = \Tb^{\alpha\beta}\delb_{\alpha}\otimes\delb_{\beta} = \Tt^{\alpha\beta}\delt_{\alpha}\otimes\delt_{\beta}.
$$
For a three-tensor $Q$, the same relation holds:
$$
Q = Q^{\alpha\beta\gamma}\del_{\alpha}\otimes\del_{\beta}\otimes\del_{\gamma} =
\Qu^{\alpha\beta\gamma}\delu_{\alpha}\otimes\delu_{\beta}\otimes\delu_{\gamma} =
\Qb^{\alpha\beta\gamma}\delb_{\alpha}\otimes\delb_{\beta}\otimes\delb_{\gamma} = \Qt^{\alpha\beta\gamma}\delt_{\alpha}\otimes\delt_{\beta}\otimes\delt_{\gamma}.
$$
We remark the following relation:
$$
\aligned
\Tu^{\alpha\beta}& = \Psiu_{\alpha'}^{\alpha}\Psiu_{\beta'}^{\beta}T^{\alpha'\beta'},\quad \Qu^{\alpha\beta\gamma} = \Psiu_{\alpha'}^{\alpha}\Psiu_{\beta'}^{\beta}\Psiu_{\gamma'}^{\gamma}Q^{\alpha'\beta'\gamma'},
\\
\Tb^{\alpha\beta}& = \Psib_{\alpha'}^{\alpha}\Psib_{\beta'}^{\beta}T^{\alpha'\beta'},\quad \Qb^{\alpha\beta\gamma} = \Psib_{\alpha'}^{\alpha}\Psib_{\beta'}^{\beta}\Psib_{\gamma'}^{\gamma}Q^{\alpha'\beta'\gamma'}
\endaligned
$$
and
$$
\Tt^{\alpha\beta} = \Psit_{\alpha'}^{\alpha}\Psit_{\beta'}^{\beta}T^{\alpha'\beta'},\quad
\Qt^{\alpha\beta\gamma} = \Psit_{\alpha'}^{\alpha}\Psit_{\beta'}^{\beta}\Psit_{\gamma'}^{\gamma}Q^{\alpha'\beta'\gamma'}.
$$

\subsection{Functional spaces}
In this subsection we introduce some norms and functional spaces for the following discussion. Firstly, for a positive constant $\gamma$, we denote by $|x|=r$ and introduce the following function:
$$
w_{\gamma}(t,x)  = \chi(r-t)(r-t+1)^{\gamma}.
$$
This is a smooth function vanishing when $r\leq t$.  We remark that in the region $\Pcal_s$, when $r\geq t$
\begin{equation}\label{eq 1 weight}
\aligned
\del_xw_{\gamma} =& \big(\chi'(r-t)(r-t+1)^{\gamma} + \gamma\chi(r-t) (r-t+1)^{\gamma-1}\big)\frac{x}{r}
\\
=& \Big(\bar{w}_{\gamma} + \frac{\gamma w_{\gamma}}{1+r-t}\Big)\frac{x}{r},
\\
\del_tw_{\gamma} =& -\big(\chi'(r-t)(r-t+1)^{\gamma} + \gamma\chi(r-t) (r-t+1)^{\gamma-1}\big)
\\
=& - \Big(\bar{w}_{\gamma} + \frac{\gamma w_{\gamma}}{1+r-t}\Big)
\endaligned
\end{equation}
with
$$
\bar{w}_{\gamma} = \chi'(|x|-t)(|x|-t+1)^{\gamma}\geq 0
$$
and vanishes for $r-t\leq 0$ or $r-t\geq 1$. This leads to the following bound:
\begin{equation}\label{eq 2 weight}
C\geq \chi'(|x|-t)(|x|-t+1)^{\gamma}\geq 0
\end{equation}
where $C$ is positive constant determined by $\chi$.

Now we denote by $\Fcal_{(s_0,s_1)}$ the open set
$$
\Fcal_{(s_0,s_1)} := \{(t,x)|T(s_0,x)<t<T(s_1,x)\}.
$$
Now introduce the class of functions
$$
\mathcal{S}_{[s_0,s_1]} := C_c^{\infty}(\Fcal_{(s_0-1,s_1+1)})
$$
which are smooth functions compactly supported in the open set $\Fcal_{(s_0-1,s_1+1)}$ which is a larger open set containing $\Fcal_{[s_0,s_1]}$.

Let $u\in\mathcal{S}_{[s_0,s_1]}$. We denote by $w_s(x) := u(T(s,x),x)$ the restriction of $u$ on $\Fcal_s$. Then $\forall s\in[s_0,s_1]$, $w_s(x)$ is smooth and compactly supported (i.e. in the class $C_c^{\infty}(\RR)$). Then we define the norm
$$
\|u\|_{L^{p}(\Fcal_s)}^p : = \|w_s\|_{L^p(\RR)}^p = \int_{\RR}|u(T(s,x),x)|^p dx ,\quad 1\leq p<\infty.
$$
We denote by
$$
\|u\|_{L_p^\infty([s_0,s_1])} := \sup_{s\in[s_0,s_1]}\big\{\|w_s\|_{L^p(\RR)}\big\},
$$
and
$$
\|u\|_{L_p^q([s_0,s_1])} := \bigg(\int_{[s_0,s_1]}\|w_s\|_{L^p(\RR)}^q\ ds\bigg)^{1/q},\quad 1\leq q<\infty.
$$
We denote by $L^q_p([s_0,s_1])$ the completion of $\mathcal{S}_{[s_0,s_1]}$ with respect to the norm $\|\cdot \|_{L_p^q([s_0,s_1])}$. In the following discussion, almost all functions under discussion are in $L^{\infty}_2([s_0,s_1])$.

\subsection{Energy estimates with combined foliation}
Remark that $\Fcal_{[s_0,\infty)}$ is also parameterized by $(s,x)$ with the relation
$$
t = T(s,x),\quad  x=x.
$$
Then we calculate the Jacobian between these two parameterizations:
\begin{equation}\label{eq 1 Jac}
J = \det\left(\frac{\del(t,x)}{\del(s,x)}\right) = \del_sT.
\end{equation}
In general the following bounds hold:
\begin{lemma}\label{lem 1 Jac}
With the above notation,
$$
0<(1-\xi_s(r))s + \frac{\xi_s(r)s}{\sqrt{s^2+x^2}}\leq
\del_sT \left\{
\aligned
=&\frac{s}{\sqrt{s^2+x^2}}, \quad |x|\leq \frac{s^2-1}{2},
\\
\leq& \frac{\xi_s(r)s}{\sqrt{s^2+r^2}} + 2(1-\xi_s(r))s, \quad \frac{s^2-1}{2}\leq |x|\leq \frac{s^2+1}{2},
\\
\leq& 2s, \quad |x|\geq \frac{s^2+1}{2}.
\endaligned
\right.
$$
\end{lemma}
\begin{proof}
{\bf Attention}, in the following proof the calculations are made with the parameterizations $(s,x)$ of $\Fcal_{[s_0,\infty)}$.

When $|x|<(s^2-1)/2$, we remark that
$$
T(s,x) = \sqrt{s^2+x^2},\quad \del_sT(s,x) = \frac{s}{\sqrt{s^2+x^2}}>0, \quad \del_sT(s,0) = 1.
$$

By the fact that $T$ is symmetric with respect to $t-$axis, we only consider the case where $x>0$. Then
$$
\aligned
\del_x\del_sT(s,x) = \del_s\del_xT(s,x) =& \frac{\del_s(\xi_s(x)x)}{\sqrt{s^2+x^2}} - \frac{s\xi_s(x)x}{(s^2+x^2)^{3/2}}
\\
=& - \frac{s\xi_s'(x)x}{\sqrt{s^2+x^2}} - \frac{s\xi_s(x)x}{(s^2+x^2)^{3/2}}
\endaligned
$$
so
$$
\aligned
\del_sT(s,x) = 1 + \int_0^x\del_s\del_xT(s,y)dy = 1 - \int_0^x\frac{s\xi_s'(y)y}{\sqrt{s^2+y^2}}dy - \int_0^x\frac{s\xi_s(y)y}{(s^2+y^2)^{3/2}}
\endaligned
$$
Remark that when $0\leq x\leq (s^2-1)/2$, $\xi_s'(x) = 0$, $\xi_s(x) = 1$, then when $x>(s^2-1)/2$,
\begin{equation}\label{eq 1 proof Jac}
\aligned
\del_sT(s,x) =& 1 - \int_0^{\frac{s^2-1}{2}}\frac{sy}{(s^2+y^2)^{3/2}}dy - \int_{\frac{s^2-1}{2}}^x\frac{s\xi_s'(y)y}{\sqrt{s^2+y^2}}dy - \int_{\frac{s^2-1}{2}}^x\frac{s\xi_s(y)y}{(s^2+y^2)^{3/2}}dy
\\
=&\frac{2s}{s^2+1} - s\int_{\frac{s^2-1}{2}}^x\frac{\xi_s'(y)y}{\sqrt{s^2+y^2}}dy + s\int_{\frac{s^2-1}{2}}^x\xi_s(y)\big((s^2+y^2)^{-1/2}\big)'dy
\\
=&\frac{2s}{s^2+1} - 2s\int_{\frac{s^2-1}{2}}^x\frac{\xi_s'(y)y}{\sqrt{s^2+y^2}}dy + \frac{s\xi_s(y)}{\sqrt{s^2+y^2}}\bigg|_{\frac{s^2-1}{2}}^x
\\
=&\frac{s\xi_s(x)}{\sqrt{s^2+x^2}} - 2s\int_{\frac{s^2-1}{2}}^x\frac{\xi_s'(y)y}{\sqrt{s^2+y^2}}dy.
\endaligned
\end{equation}

We denote by
$$
f_s(x) := \frac{x}{\sqrt{s^2+x^2}},
$$
then
\begin{equation}\label{eq 2 proof Jac}
f_s'(x) = \frac{s^2}{(s^2+x^2)^{3/2}}\geq 0.
\end{equation}
then on the right-hand-side of \eqref{eq 1 proof Jac},
$$
\aligned
&s\int_{\frac{s^2-1}{2}}^x\frac{\xi_s'(y)y}{\sqrt{s^2+y^2}}dy
\\
=&s\int_{\frac{s^2-1}{2}}^x\xi_s'(y)f_s(y)dy = sf_s(y)\xi_s(y)\bigg|_{\frac{s^2-1}{2}}^x - s\int_{\frac{s^2-1}{2}}^x\xi_s(y)f_s'(y)dy
\\
=&sf_s(x)\xi_s(x) - sf_s((s^2-1)/2) - s\int_{\frac{s^2-1}{2}}^x\xi_s(y)f_s'(y)dy
\\
=&sf_s(x)\big(\xi_s(x)-1\big) + s\big(f_s(x) - f_s((s^2-1)/2)\big) - s\int_{\frac{s^2-1}{2}}^x\xi_s(y)f_s'(y)dy.
\endaligned
$$
Remark in the above identity by \eqref{eq 2 proof Jac},
$$
\aligned
s\int_{\frac{s^2-1}{2}}^x\frac{\xi_s'(y)y}{\sqrt{s^2+y^2}}dy
\geq& sf_s(x)\big(\xi_s(x)-1\big) + s\big(f_s(x) - f_s((s^2-1)/2)\big) - s\int_{\frac{s^2-1}{2}}^xf_s'(y)dy
\\
= &sf_s(x)\big(\xi_s(x)-1\big)\geq s(\xi_s(x)-1)
\endaligned
$$
where  in the last inequality we have applied the fact that $\xi_s(x)-1\leq 0$ and $0\leq f_s(x)\leq 1$. Then recall \eqref{eq 1 proof Jac}, we see that for $(s^2-1)/2\leq x\leq (s^2+1)/2$,
\begin{equation}\label{eq 3 proof Jac}
\del_sT(s,x)\leq \frac{s\xi_s(x)}{\sqrt{s^2+r^2}} + 2s(1-\xi_s(x)).
\end{equation}
Then the upper bound is established.

On the other hand, remark that $-\xi_s'(x)\geq 0$ and $f$ increasing,
$$
\aligned
&- 2s\int_{\frac{s^2-1}{2}}^x\frac{\xi_s'(y)y}{\sqrt{s^2+y^2}}dy = 2s\int_{\frac{s^2-1}{2}}^x \big(-\xi_s'(y)\big)f(y)dy
\\
\geq& -2sf((s^2-1)/2)\int_{\frac{s^2-1}{2}}^x\xi_s'(y)dy = \frac{2(1-\xi_s(x))s(s^2-1)}{s^2+1}.
\endaligned
$$
Thus for $s\geq 2$,
\begin{equation}\label{eq 4 proof Jac}
- 2s\int_{\frac{s^2-1}{2}}^x\frac{\xi_s'(y)y}{\sqrt{s^2+y^2}}dy \geq (1-\xi_s(x))s.
\end{equation}

Then by \eqref{eq 1 proof Jac}, the lower bound is established for $(s^2-1)/2\leq x\leq (s^2+1)/2$.

For the case $x\geq (s^2+1)/2$, we remark that $\xi_s(x) = \xi_s'(x) = 0$ for $x\geq (s^2+1)/2$. So
$$
\aligned
\del_sT(s,x) =& \int_0^x- \frac{s\xi_s'(y)y}{\sqrt{s^2+y^2}} - \frac{s\xi_s(y)y}{(s^2+y^2)^{3/2}}\, dy
\\
=& \int_0^{(s^2+1)/2} - \frac{s\xi_s'(y)y}{\sqrt{s^2+y^2}} - \frac{s\xi_s(y)y}{(s^2+y^2)^{3/2}}\, dy
\\
=& \del_s T(s,(s^2+1)/2).
\endaligned
$$
Then by \eqref{eq 1 proof Jac}
$$
\del_sT(s,x) = \del_sT(s,(s^2+1)/2) = \frac{s\xi_s(x)}{\sqrt{s^2+x^2}} - 2s\int_{\frac{s^2-1}{2}}^x\frac{\xi_s'(y)y}{\sqrt{s^2+y^2}}dy.
$$
Then by \eqref{eq 3 proof Jac} and \eqref{eq 4 proof Jac}, the bounds for $x\geq (s^2+1)/2$ is established.
\end{proof}

\begin{lemma}\label{lem 1 ds}
Taking $s$ as a function of $(t,x)$,
\begin{equation}
\del_ts(t,x) = \frac{1}{\del_s T}, \quad \del_xs(t,x) = -\frac{\del_xT}{\del_s T}
\end{equation}
\end{lemma}
\begin{proof}
Recall that by \eqref{eq 1 Jac}
$$
J = \frac{\del(t,x)}{\del(s,x)}
 = \left(
\begin{array}{cc}
\del_s T &\del_x T
\\
\del_s x &\del_x x
\end{array}
\right)
=
\left(
\begin{array}{cc}
\del_s T &\del_x T
\\
0 &1
\end{array}
\right).
$$
So
$$
\frac{\del(s,x)}{\del(t,x)} = J^{-1} =
\left(
\begin{array}{cc}
\frac{1}{\del_s T} &-\frac{\del_x T}{\del_s T}
\\
0 &1
\end{array}
\right).
$$
\end{proof}

Now we introduce the following energy on $\Fcal_s$. For $u\in\mathcal{S}_{[s_0,s_1]}$, we define:
$$
\vec{V} := (1+w_{\gamma})^2\big(g^{00}|\del_tu|^2-g^{11}|\del_xu|^2 + c^2u^2,2g^{1\beta}\del_tu\del_{\beta}u\big)
$$
\begin{equation}\label{eq 1 energy}
\aligned
E_{g,\gamma,c}(s,u) :=&\int_{\Fcal_s}\vec{V}\cdot\vec{n}d\sigma
\\
=&\int_{\Fcal_s}(1+w_{\gamma})^2\Big(g^{00}|\del_tu|^2 - g^{11}|\del_xu|^2 - 2\del_x T\cdot g^{1\beta}\del_tu\del_{\beta}u + c^2u^2\Big)\ dx
\endaligned
\end{equation}
where $\vec{n}$ is the normal vector of $\Fcal_s$.
Recall that for $|x|\leq t$, $w_{\gamma}=0$. Also recall that by \eqref{eq 1 lem 2 position}, $r\leq t$ on $\Hcal_s^*\cup \Tcal_s$. Then
$$
\aligned
E_{g,\gamma,c}(s,u) =& \int_{\Hcal_s^*}\big(g^{00}|\del_tu|^2 - g^{11}|\del_xu|^2 - 2(x/t)g^{1\beta}\del_tu\del_{\beta}u + c^2u^2\big)\ dx
\\
&+\int_{\Tcal_s}\big(g^{00}|\del_tu|^2 - g^{11}|\del_xu|^2 - \frac{2\xi_s(r)x}{\sqrt{s^2+x^2}}g^{1\beta}\del_tu\del_{\beta}u + c^2u^2\big)\ dx
\\
&+\int_{\Pcal_s} (1+w_{\gamma})^2\big(g^{00}|\del_tu|^2 - g^{11}|\del_xu|^2 + c^2u^2 \big)\ dx
\\
=:& E^{\text{H}}_{g,c}(s,u) + E^{\text{T}}_{g,c}(s,u) + E^{\text{P}}_{g,\gamma,c}(s,u).
\endaligned
$$
For the convenience of discussion, we also introduce:
\begin{equation}\label{eq 2 energy}
E^{\text{E}}_{g,\gamma,c} := \int_{\Hb_s}\vec{V}\cdot \vec{n}d\sigma = E^{\text{T}}_{g,c}(s,u) + E^{\text{P}}_{g,\gamma,c}(s,u)
\end{equation}
and
\begin{equation}\label{eq 3 energy}
\aligned
E^{\text{K}}_{g,c}(s, u;s_0) :=& \int_{\del\Kcal[s_0,s]}\vec{V}\cdot \vec{n}d\sigma = \int_{\del\Kcal[s_0,s]}\big(g^{00}|\del_tu|^2 - g^{11}|\del_xu|^2 - 2(x/r)g^{1\beta}\del_tu\del_{\beta}u + c^2u^2\big)dt.
\endaligned
\end{equation}

We pay special attention to the case where $g = m$. In this case we denote by $E_{\gamma,c}(s,u) = E_{m,\gamma,c}(s,u)$, and when $c=0,g=m$, we denote by $E_\gamma(s,u) = E_{\gamma,c}(s,u)$. Then
$$
\aligned
E_{\gamma,c}(s,u) =& \int_{\Hcal_s^*}\big(|\del_tu|^2 + |\del_xu|^2 + 2(x/t)\del_tu\del_xu + c^2u^2\big)\ dx
\\
&+\int_{\Tcal_s} \big(|\del_tu|^2 + |\del_xu|^2 + \frac{2\xi_s(|x|)x}{\sqrt{s^2+x^2}}\del_tu\del_xu + c^2u^2\big)\ dx
\\
&+\int_{\Pcal_s}(1+w_{\gamma})^2 \big(|\del_tu|^2 + |\del_xu|^2 + c^2u^2\big)\ dx.
\endaligned
$$
So
\begin{equation}\label{eq int-energy-flat}
\aligned
E^{\text{H}}_{c}(s,u) =& \int_{\Hcal_s}\big(|\del_tu|^2 + |\del_xu|^2 + 2(x/t)\del_tu\del_xu + c^2u^2\big)\ dx
\\
=& \int_{\Hcal_s}\Big(\big|(x/t)\del_t u+ \del_x u\big|^2 + (s/t)^2|\del_tu|^2 + c^2u^2\Big)\ dx
\\
=& \int_{\Hcal_s}\Big(\big|(x/t)\del_xu + \del_t u\big|^2 + (s/t)^2|\del_xu|^2 + c^2u^2\Big)\ dx.
\endaligned
\end{equation}
\begin{equation}\label{eq trans-energy-flat}
\aligned
E^{\text{T}}_{c}(s,u) =& \int_{\Tcal_s}\Big(|\del_tu|^2\ + |\del_xu|^2 + \frac{2\xi_s(r)x}{\sqrt{s^2+x^2}}\del_xu\del_t u + c^2u^2\Big)\ dx
\\
=& \int_{\Tcal_s}\left(|\zeta(s,x)\del_tu|^2 + \Big(\frac{\xi_s(r)x}{\sqrt{s^2+x^2}}\del_tu + \del_x u\Big)^2 + c^2u^2\right)\ dx
\\
=& \int_{\Tcal_s}\left(|\zeta(s,x)\del_xu|^2 + \Big(\frac{\xi_s(r)x}{\sqrt{s^2+x^2}}\del_xu + \del_t u\Big)^2 + c^2u^2\right)\ dx
\endaligned
\end{equation}
where
$$
\zeta(s,x) = \sqrt{\frac{s^2+(1-\xi^2_s(r))x^2}{s^2+x^2}}.
$$
We remark that for $(t,x)\in\Tcal_s$,
\begin{equation}
0\leq \frac{t-r}{t}\leq \zeta
\end{equation}

\begin{equation}\label{eq flat-energy-flat}
E^{\text{P}}_{\gamma,c}(s,u):= \int_{\Pcal_s} (1+w_{\gamma})^2\big(|\del_tu|^2 + |\del_xu|^2 + c^2u^2 \big)\ dx.
\end{equation}

We remark especially that
\begin{equation}\label{eq cone-energy-flat}
\aligned
E^{\text{K}}_c(s,u;s_0) =& \int_{\del\Kcal_{[s_0,s]}} \big(((x/r)\del_t u + \del_x u)^2 + c^2u^2\big)\ dt
=  \int_{\del\Kcal_{[s_0,s]}} \big(|\delt_1u|^2 + c^2u^2\big)\ dt\geq 0.
\endaligned
\end{equation}

%


We pay special attention to the transition region, and establish the following relations:
\begin{lemma}\label{lem 1 trans-energy}
Let $u$ be a sufficiently regular function defined in $\Fcal_{[s_,s_1]}$. Then for $s_0\leq s\leq s_1$, the following quantities:
\begin{equation}\label{eq 1 lem 1 trans-energy}
\|\zeta(s,\cdot )\del_{\alpha}u\|_{L^2(\Tcal_s)},\quad \|(s^{-1}\xi_s(\cdot ) + (1-\xi_s(\cdot))^{1/2})\del_\alpha u\|_{L^2(\Tcal_s)},\quad \|\delt_1u\|_{L^2(\Tcal_s)}
\end{equation}
are bounded by $E^T(s,u)^{1/2}$
\end{lemma}
\begin{proof}
The first term is obvious.

For the second, we remark that in $\Tcal_{[s_0,s_1]}$,
$$
t\sim r\sim s^2.
$$
Then we only need to remark that
\begin{equation}\label{eq xi-zeta}
\sqrt{1-\xi_s(r)}\leq \zeta(s,x),\quad \frac{1}{s}\lesssim \frac{s}{\sqrt{s^2+r^2}}\leq \zeta(s,x).
\end{equation}

For the last, remark that the $L^2$ norm of the following terms
$$
 \frac{(1-\xi_s(r))r}{\sqrt{s^2+x^2}}\del_t u,\quad \frac{\xi_s(r)x}{\sqrt{s^2+x^2}}\del_tu + \del_x u, \quad \left(1-\frac{|x|}{\sqrt{s^2+x^2}}\right)\del_tu
$$
are bounded by $E^T(s,u)^{1/2}$. Here we remark that
$$
0\leq 1-\frac{|x|}{\sqrt{s^2+x^2}}\leq\frac{s}{\sqrt{s^2+x^2}}\leq \zeta.
$$
Then
$$
\aligned
\delt_1 u =& \frac{x}{|x|}\del_t + \del_x
\\
=& \left(\frac{\xi_s(r)x}{\sqrt{s^2+x^2}}\del_tu + \del_x u\right) + (x/|x|) \frac{(1-\xi_s(r))r}{\sqrt{s^2+x^2}}\del_t u + (x/|x|)\left(1-\frac{r}{\sqrt{s^2+x^2}}\right)\del_tu
\endaligned
$$
leads to the desired bound.
\end{proof}

Remark that
\begin{equation}\label{eq energy-stuc-ext}
\|(1+|r-t|)^{\gamma}\zeta(s,\cdot)\del_{\alpha}u\|_{L^2(\Hb_s)},\quad \|(1+|r-t|)^{\gamma}\delt_1u\|_{L^2(\Hb_s)}
\end{equation}
are controlled by $E^{\text{E}}(s,u)^{1/2}$.

The above energy can be defined for any $C^1$ function $u$ with $\del_{\alpha}u$ and $u$ being continuous in $\Fcal_{[s_0,s_1]}$ satisfying the following condition
\begin{equation}\label{eq energy-class}
(1+w_{\gamma})\del_{\alpha}u\in L^{\infty}_2([s_0,s_1]),\quad  c(1+w_{\gamma})u\in L^{\infty}_2([s_0,s_1]).
\end{equation}
We denote by
$$
\mathcal{E}_{\gamma,c}([s_0,s_1]) := \{u \text{ defined in } \Fcal_{[s_0,s_1]} \text{ satisfying }\eqref{eq energy-class}\}.
$$
Remark that $\mathcal{S}_{[s_,s_1]}$ is dense in $\mathcal{E}_{\gamma,c}([s_0,s_1])$ with respect to the norm
$$
\|(1+\omega_{\gamma})u\|_{L^2(\RR)} + \|c(1+\omega_{\gamma})u\|_{L^2(\RR)}.
$$

\begin{proposition}[Energy estimate in interior region]\label{prop 1 energy}
Let $u$ be a $C^2$ function defined in the region $\Fcal_{[s_0,s_1]}$ with $2\leq s_0<s_1$ and $u\in \mathcal{E}_{\gamma,c}([s_0,s_1])$. Suppose that
\begin{equation}\label{eq main-energy}
\Box u + c^2u = f
\end{equation}
with  $f$ continuous in $\Fcal_{[s_0,s_1]}$. Then
\begin{equation}\label{eq 1 prop 1 energy}
\aligned
E^{\text{H}}_c(s,u)\leq E^{\text{H}}_c(s_0,u) + E^{\text{K}}_c(s,u;s_0) + 2\int_{s_0}^sE^{\text{H}}_c(s',u)^{1/2}\cdot\|f\|_{L^2(\Hcal_{s'})}\ ds'.
\endaligned
\end{equation}
\end{proposition}

\begin{proof}
By applying the multiplier $\del_t$, \eqref{eq main-energy} leads to
$$
\del_t\big(|\del_tu|^2 + |\del_xu|^2 + c^2u^2\big) - 2\del_x(\del_tu\del_xu)  = 2\del_t u\cdot f.
$$
Integrate the above identity in the region $\Hcal^*_{[s_0,s]}$ and apply Stokes' formula, we obtain that
\begin{equation}\label{eq 1 pr prop 1 energy}
E_c^{\text{H}}(s,u) - E_c^{\text{H}}(s_0,u) - E_c^{\text{K}}(s,u;s_0) = \int_{\Hcal^*_{[s_0,s]}}2\del_tu\cdot f\ dxdt.
\end{equation}
Then we remark the relation
$$
dxdt = (s/t)dxds.
$$
thus the LHD of \eqref{eq 1 pr prop 1 energy} is bounded by
$$
2\int_{s_0}^sE^{\text{H}}(s',u)^{1/2}\cdot \|f\|_{L^2(\Hcal_{s'})}\ ds'
$$
which guarantees \eqref{eq 1 prop 1 energy}.
\end{proof}

\begin{proposition}[Energy estimate in exterior region]\label{prop 2 energy}
Let $u$ be a $C^2$ function defined in the region $\Fcal_{[s_0,s_1]}$ with $2\leq s_0<s_1$ and $u\in \mathcal{E}_{\gamma,c}([s_0,s_1])$. Suppose that
\begin{equation}\label{eq main-energy}
\Box u + c^2u = f
\end{equation}
with  $f$ continuous in $\Fcal_{[s_0,s_1]}$. Then
\begin{equation}\label{eq 1 prop 2 energy}
\aligned
E^{\text{E}}_{\gamma,c}(s,u) + E^{\text{K}}_c(s,u)\leq E^{\text{E}}_{\gamma,c}(s_0,u)
+ C\int_{s_0}^sE^{\text{E}}_{\gamma,c}(s',u)^{1/2}\cdot \|(1+(1-\xi_{s'})^{1/2}s')(1+\omega_{\gamma})f\|_{L^2(\Hb_{s'})}ds'.
\endaligned
\end{equation}
\end{proposition}
\begin{proof}
We remark the following identity:
\begin{equation}\label{eq 0 pr prop 2 energy}
\aligned
2(1+\omega_{\gamma})^2\del_t u \cdot (\Box u + c^2u)
&=\del_t\big((1+\omega_{\gamma})^2(|\del_tu|^2+|\del_xu|^2+c^2u^2)\big) - 2\del_x\big((1+\omega_{\gamma})^2\del_tu\del_xu\big)
\\
&+2(\bar{\omega}_{\gamma} + \gamma\omega_{\gamma}(1+|r-t|)^{-1})(1+\omega_{\gamma})\big(|\delt_1u|^2+c^2u^2\big).
\endaligned
\end{equation}
Then integrate the above identity in $\Hb_{[s_0,s]}$ and apply the Stokes' formula:
\begin{equation}\label{eq 1 pr prop 2 energy}
\aligned
E^{\text{E}}_{\gamma,c}(s,u) + E^{\text{K}}_{\gamma,c}(s,u;s_0) =& E^{\text{E}}_{\gamma,c}(s_0,u)
+ \int_{s_0}^s\int_{\Hb_{s'}}2(1+\omega_{\gamma})^2\del_t u\cdot f\cdot |\del_{s'}T|dxds'
\\
&- 2\int_{\Hb_{[s_0,s]}}(\bar{\omega}_{\gamma} + \gamma\omega_{\gamma}(1+|r-t|)^{-1})(1+\omega_{\gamma})\big(|\delt_1u|^2+c^2u^2\big)dxdt.
\endaligned
\end{equation}
For the second term in RHS of the above identity, by lemma \ref{lem 1 Jac},
$$
\aligned
&\int_{s_0}^s\int_{\Hb_{s'}}\big|2(1+\omega_{\gamma})^2\del_t u\cdot f\cdot |\del_{s'}T|\ \big|\ dxds'
\\
\leq&4\int_{s_0}^s\int_{\Hb_{s'}}\big|(1+\omega_{\gamma})\del_t u(1-\xi_{s'})^{1/2}\big|\cdot \big|s'(1+\omega_{\gamma})(1-\xi_{s'})^{1/2}f\big|\ dxds'
\\
&+ 2\int_{s_0}^s\int_{\Hb_{s'}} \frac{(1+\omega_{\gamma})|s'\del_tu|}{\sqrt{{s'}^2+r^2}}\cdot \big|(1+\omega_{\gamma})f\big| \ dxds'
\\
\leq& C\int_{s_0}^sE^{\text{E}}_{\gamma,c}(s',u)^{1/2}\cdot \|(1+(1-\xi_{s'})^{1/2}s')(1+\omega_{\gamma})f\|_{L^2(\Hb_{s'})}ds'.
\endaligned
$$

The last term in RHS of \eqref{eq 1 pr prop 2 energy} is positive. So \eqref{eq 1 prop 2 energy} is concluded.
\end{proof}

In the following discussion, we denote by
$$
\EE_{N,\gamma,c}(s,u) := \sum_{|I|+j\leq N}E^{\text{E}}_{\gamma,c}(s,\del^IL^j u), \quad
\EH_{N,c}(s,u) := \sum_{|I|+j\leq N}E^{\text{H}}_c(s,\del^IL^j u).
$$

\section{Global Sobolev inequality}\label{sec sobolev}
\subsection{Global Sobolev inequality in combined foliation context}
We first establish the following basic Sobolev type inequality:
\begin{lemma}\label{lem 1 sobolev}
Let $u$ be a function defined on $\RR$, sufficiently regular. Then
\begin{equation}\label{eq 1 lem 1 sobolev}
|u(x)|^2\leq C\int_x^{x+1}\big(u^2(y) + |u'(y)|^2\big)\,dy.
\end{equation}
\end{lemma}
\begin{proof}
Recalling the function $\chi$ defined in \eqref{eq 1 cu-off}, we define
$$
v_x(y) := u(x+y)\big(1-\chi(y)\big).
$$
Then $v_x(0) = u(x)$, $v_x(1) = 0$. We also remark that
\begin{equation}\label{eq 1 pr lem 1 sobolev}
v_x'(y) = u'(x+y)\big(1-\chi(y)\big) - \chi'(y)\cdot u(x).
\end{equation}
Recall that $\chi'$ is bonded by a constant.

Also, we remark that for a sufficiently regular function $v$ defined on $[0,1]$ with $v(1)=0$,
$$
-v^2(0) = \int_0^1\frac{d}{dy}(v^2(y))dy =  2\int_0^1v(y)v'(y)dy.
$$
Then
\begin{equation}\label{eq 2 pr lem 1 sobolev}
v^2(0)\leq C\|v\|_{L^2([0,1])}\cdot\|v'\|_{L^2([0,1])}\leq C\big(\|v\|_{L^2([0,1])} + \|v'\|_{L^2([0,1])}\big)^2.
\end{equation}

Now in\eqref{eq 2 pr lem 1 sobolev}, taking $v = v_x$ and take \eqref{eq 1 pr lem 1 sobolev} into consideration, the desired result is proved.
\end{proof}

Now we are ready to establish the following global Sobolev type estimate:
\begin{proposition}\label{prop 1 K-S}
Let $u$ be a function defined in $\Fcal_{[s_0,s_1]}$, sufficiently regular. Suppose that $(t,x)\in\Hcal^*_s$. Then the following estimate holds:
\begin{equation}\label{eq 1 prop 1 K-S}
tu^2(t,x)\leq Cc\big(\|u\|_{L^2(\Hcal_s)}^2 + \|Lu\|_{L^2(\Hcal_s)}^2\big).
\end{equation}
\end{proposition}
\begin{proof}
Remark that $(t,x)\in \Hcal^*_s$ leads to $t = \sqrt{s^2+x^2}$. We denote by
$$
v_s(x) := u\big(\sqrt{s^2+x^2},x\big), \quad |x|\leq (s^2-1)/2
$$
the restriction of $u$ on $\Hcal^*_s$. In the following proof we only consider $-(s^2-1)/2\leq x\leq 0$. For rest part one can take $\tilde {u}(x) = u(-x)$ and the argument is similar.

 Remark that
$$
v_s'(x) = \bigg(\frac{x}{\sqrt{s^2+x^2}}\del_t + \del_x\bigg)u\big(\sqrt{s^2+x^2},x\big).
$$
We remark that
\begin{equation}\label{eq 1 pr prop 1 K-S}
\big(s^2+x^2\big)^{1/2}v_s'(x) = Lu\big(\sqrt{s^2+x^2},x\big).
\end{equation}

Furthermore, we denote by
$$
w_{s,x}(y) := v_s\Big(y\sqrt{s^2+x^2}/2 + x\Big).
$$
For the convenience of discussion, we denote by
$$
z =  y\sqrt{s^2+x^2}/2 + x,
$$
then we remark that
$$
\aligned
w_{s,x}'(y)
=& \frac{1}{2}\big(s^2+x^2\big)^{1/2}v_s\Big(y\sqrt{s^2+x^2}/2 + x\Big)
=\frac{1}{2}\sqrt{\frac{s^2+x^2}{z^2+s^2}}
\cdot \sqrt{z^2+s^2}\cdot v_s(z)
\\
=&\frac{1}{2}\sqrt{\frac{s^2+x^2}{z^2+s^2}}\cdot Lu\Big(\sqrt{z^2+s^2},z\Big).
\endaligned
$$
We analyse the coefficient. Remark that $x\leq 0$ and $y\in[0,1]$, then
$$
\aligned
\frac{z^2+s^2}{x^2+s^2}=&\frac{y^2(s^2+x^2)/4 + (s^2+x^2) - |x|y(s^2+x^2)^{1/2}}{s^2+x^2}
\\
=&\big(y/2 - |x|(s^2+x^2)^{-1/2}\big)^2 + s^2(s^2+x^2)^{-1}.
\endaligned
$$

When $|x|\leq 3s$,
$$
 s^2(s^2+x^2)^{-1}\geq 1/10 >0.
$$
When $|x|> 3s$,
$$
|x|(s^2+x^2)^{-1/2} -y/2 \geq |x|(s^2+x^2)^{-1/2} -1/2 \geq 3/\sqrt{10} -1/2 >0.
$$
So there is a universal positive constant $C$ such that for $y\in[0,1]$ and $-(s^2-1)/2x\leq 0$,
$$
\sqrt{\frac{s^2+x^2}{z^2+s^2}}\leq C
$$
which leads to
\begin{equation}\label{eq 2 pr prop K-S}
|w_{s,x}'(y)|\leq Lu\big(\sqrt{z^2+x^2},z\big).
\end{equation}

Then we remark the following relation: the function
$$
x + \frac{\sqrt{s^2+x^2}}{2}
$$
is increasing with respect to $x$, so
$$
-(s^2-1)/2\leq x\leq x+ \frac{\sqrt{s^2+x^2}}{2}\leq s/2\leq (s^2-1)/2
$$

Then we remark the following two relations (in the following calculation, $z = (s^2+x^2)^{1/2}y/2 + x$):
\begin{equation}\label{eq 3 pr prop K-S}
\aligned
\int_0^1|w_{s,x}(y)|^2\,dy =& \int_0^1\big|v_s\big((s^2+x^2)^{1/2}y/2 + x\big)\big|^2\,dy
\\
=&\frac{2}{(s^2+x^2)^{1/2}}\int_x^{x+(x^2+s^2)^{1/2}/2}\big|v_s(z)\big|^2\,dz
\\
=& \frac{2}{(s^2+x^2)^{1/2}}\int_x^{x+(x^2+s^2)^{1/2}/2}\big|u\big(\sqrt{z^2+s^2},z\big)\big|^2\, dz
\\
\leq& \frac{2}{(s^2+x^2)^{1/2}}\int_{-(s^2-1)/2}^{(s^2-1)/2}\big|u\big(\sqrt{z^2+s^2},z\big)\big|^2\, dz
\\
\leq& 2(s^2+x^2)^{-1/2}\|u\|_{L^2(\Hcal_s)}^2.
\endaligned
\end{equation}
And (by \eqref{eq 2 pr prop K-S})
\begin{equation}\label{eq 4 pr prop K-S}
\aligned
\int_0^1|w_{s,x}'(y)|^2\,dy\leq &C\int_0^1\big|Lu\big(\sqrt{z^2+x^2},z\big)\big|^2\,dy
\\
=&\frac{2C}{(s^2+x^2)^{1/2}}\int_x^{x+(x^2+s^2)/2}\big|Lu\big(\sqrt{z^2+x^2},z\big)\big|^2\,dz
\\
\leq& C(s^2+x^2)^{-1/2}\|Lu\|_{L^2(\Hcal_s)}.
\endaligned
\end{equation}

Then, apply lemma \ref{lem 1 sobolev} on $w_{s,x}$, the desired result is established.
\end{proof}

On the transition and exterior region, we have the following Sobolev inequalities:
\begin{lemma}\label{lem 1 sobolev-trans}
Let $u$ be a function defined in $\Fcal_{[s_0,s_1]}$, sufficiently regular. Then the following inequality holds:
\begin{equation}\label{eq 1 sobolev-trans}
|u\big(T(s,x),x\big)|^2\leq C\big(\|\delb_1u\|_{L^2(\Tcal_s)}^2 + \|u\|_{L^2(\Tcal_s)}^2\big),\quad x\in\Tcal_s,
\end{equation}
\end{lemma}
\begin{proof}
We only consider the region where $x\geq 0$. When $x\leq 0$, we consider the transform $\tilde u\big(T(s,x),x\big) := u\big(T(s,-x),-x\big)$.

When $x\in \Tcal_s$ and $(s^2-1)/2\leq x\leq s^2/2$. Let
$$
v_{s,x}(y) := u\big(T(s,x + y/2), x + y/2\big),\quad y\in[0,1].
$$
Observe that when $y\in[0,1]$,
$$
(s^2-1)/2\leq y/2+x\leq (s^2+1)/2
$$
and
$$
v_{s,x}(0) = u\big(T(s,x),x\big).
$$
Furthermore, denote by $z = x + y/2$,
$$
\aligned
v_{s,x}'(y) =& \frac{1}{2}\big(\del_xT\cdot \del_tu + \del_x u\big)(T_s(z), z)
\\
=& \frac{1}{2}\delb_1u(T(s,z),z).
\endaligned
$$

Then we remark the following relation:
$$
\aligned
\int_0^1|v_{s,x}(y)|^2dy =& \int_0^1 \big|u\big(T(s,x + y/2),y/2+x\big)\big|^2dy
\\
=&2\int_x^{x+1/2}\big|u\big(T(s,z),z\big)\big|^2dz
\\
\leq& C\|u\|_{L^2(\Tcal_s)}^2
\endaligned
$$
and
$$
\aligned
\int_0^1|v_{s,x}'(y)|^2dy =& \frac{1}{4}\int_x^{x+1/2}|\delb_1 u(T(s,z),z)|^2\,dy = \frac{1}{2}\int_x^{x+1/2}|\delb_1 u(T(s,z),z)|^2\,dz
\\
\leq& C\|\delb_1 u\|_{L^2(\Tcal_s)}^2.
\endaligned
$$

Now apply \eqref{eq 1 lem 1 sobolev} on $v_{s,x}$,
$$
|u^2\big(T(s,x),x\big)| = |v_{s,x}^2(0)|\leq C\big(\|v_{s,x}'\|_{L^2([0,1])}^2 + \|v_{s,x}\|_{L^2([0,1])}^2\big)
\leq C\big(\|u\|_{L^2(\Tcal_s)}^2 + \|\delb_1 u\|_{L^2(\Tcal_s)}^2\big).
$$

When $s^2/2\leq x\leq (s^2+1)/2$. We introduce
$$
v_{s,x}(y) = u\big(T(x-y/2),x-y/2\big).
$$
Then we apply \eqref{eq 1 lem 1 sobolev} on $v_{s,x}$. The discussion is similar to the above case where $(s^2-1)/2\leq x\leq s^2/2$, we omit the detail.
\end{proof}

\begin{proposition}\label{prop 2 sobolev}
Let $u$ be a function defined in $\Fcal_{[s_0,s_1]}$, sufficiently regular. Then for $(t,x)\in\Pcal_s$, the following inequalities hold:
\begin{equation}\label{eq 1 sobolev-ext}
|(1+w_{\gamma})u(t,x)|^2\leq C\big(\|(1+w_{\gamma}(T(s),\cdot))\del_x u\|_{L^2(\Pcal_s)}^2 + \|(1+w_{\gamma}(T(s),\cdot))u\|_{L^2(\Pcal_s)}^2\big).
\end{equation}
\end{proposition}
\begin{proof}
$(t,s)\in\Pcal_s$ leads to the following facts:
$$
|x|\geq (s^2+1)/2,\quad t = T(s).
$$
We only consider the case $x\geq (s^2+1)/2$. For the case $x\leq -(s^2+1)/2$ we take the transformation $\tilde{u}(t,x) = u(t,-x)$.
Now we consider the function
$$
v_{t,x}(y) = (1 + w_{\gamma}(t,x+y))u(t,x+y) = \big(1 + \chi(x+y-t)(1+x+y-t)^{\gamma}\big)u(t,x+y).
$$
Then
$$
\aligned
v_{t,x}'(y) =& \del_xw_{\gamma}(t,x+y)u(t,x+y) + \big(1 + w_{\gamma}(t,x+y)\big)\del_xu(t,x+y)
\\
=&\frac{\del_xw_{\gamma}(t,x+y)}{w_{\gamma}(t,x+y)}\cdot w_{\gamma}(t,x+y)u(t,x+y) + \big(1+ w_{\gamma}(t,x+y)\big)\del_xu(t,x+y).
\endaligned
$$
We remark that
$$
0\leq \frac{\del_xw_{\gamma}(t,x+y)}{w_{\gamma}(t,x+y)}\leq \frac{\bar{w}_{\gamma}(t,x+y)}{w_{\gamma}(t,x+y)} + \frac{\gamma}{(1+x+y-t)}\leq C . 
$$

Then
$$
\int_0^1|v_{t,x}'(y)|^2\,dy \leq \int_x^{x+1}|w_{\gamma}(t,y)u(t,y)|^2\, dy + \int_x^{x+1}\big|1+w_{\gamma}(t,y)\del_xu(t,y)\big|^2\, dy
$$
and
$$
\int_0^1|v_{t,x}(y)|^2\,dy = \int_x^{x+1}|\big(1+w_{\gamma}(t,y)\big)u(t,y)|^2\,dy.
$$
Then apply \ref{eq 1 lem 1 sobolev} the desired result is established.
\end{proof}

\section{Basic calculus in transition and exterior region}\label{sec basic-ext}
Equipped with the energy estimates and global Sobolev inequalities, we are nearly ready to get a parallel framework as in the hyperboloidal foliation context. However, before regarding a concrete example in Part II, we need to establish several decompositions and estimates on commutators just as we have done for hyperboloidal foliation framework. Here we concentrate on these work in the transition and exterior region. The parallel calculus in interior region are nearly the same as what we have done in previous work (e.g. \cite{LM1}) and we will only give a sketch in appendix.

\underline{ All calculus in this section is made in the region $\Hb_{[s_0,s_1]}$ unless otherwise specified.}

\subsection{Families of vector fields}
In the exterior and transition region, we introduce the {\sl null derivative} $\delt_1 := (x/r)\del_t + \del_x$. We denote by
$$
\Zext := \Zscr \cup \{\delt_1\}.
$$
where $\Zscr = \{\del_t,\del_x,L\}$, see appendix \ref{sec basic} for detailed deiscussion. A high-order operator defied in $\Hb_{[s_0,\infty)}$ is said to be of type $(i,j,l)$, if it contains $i$ partial derivatives, $j$ Lorentzian boosts and $l$ null derivatives.
\subsection{Homogeneous functions}
As in the interior region, we introduce the following notion of homogeneous functions in transition and exterior region:
\begin{definition}
A $C^{\infty}$ function $u$ defined in $\{r\neq 0\}$ is said to be homogeneous of degree $k\in \mathbb{Z}$ in exterior region, if $\forall \lambda>0$,
$$
u(\lambda t,\lambda x) = \lambda^k u(t,x)
$$
and
$$
|\del^I u(t,x)|\leq C(I), \quad \forall |x|=1, 0\leq t\leq 2
$$
where $C(I)$ is a constant determined by $I$ and $u$.
\end{definition}
For simplicity, in this section when we say ``homogeneous'', we always mean by ``homogeneous in transition and exterior region''.
The constants are homogeneous of degree zero, and so is the function $t/r$.

The following properties are immediate:
\begin{lemma}
Let $u,v$ be homogeneous of degree $k,l$ respectively in exterior region. Then
\\
$-$ when $k=l$, $\alpha u + \beta v$ is homogeneous of degree $k$,
\\
$-$ $uv$ is homogeneous of degree $k+l$.
\\
$-$ $\del^IL^ju$ is homogeneous of degree $k-|I|$,
\\
$-$ $|\del^IL^j u|\leq C(I,j)(1+r)^{k-|I|}$.
\end{lemma}
\begin{proof}
Only the third deserve a proof. We derive the following equation with respect to $t$ and $x$,
$$
\lambda^k u(t',x') = u(\lambda t',\lambda x')
$$
and obtain
$$
\lambda^{k-1}\del_t u(t',x') = \del_t(\lambda t',\lambda x'),\quad \lambda^{k-1}\del_xu(t',x')  = \del_xu(\lambda t',\lambda x').
$$
Then,
$$
\aligned
Lu(\lambda t',\lambda x') =& \lambda t'\del_xu(\lambda t',\lambda x') + \lambda x' \del_t u(\lambda t',\lambda x')
\\
=& \lambda^k t'\del_xu(t',x') + x'\del_tu(t',x') = \lambda^k Lu(t',x').
\endaligned
$$
That is, when derived with respect $\del_{\alpha}$, the degree of homogeneity is reduced by one while derived with respect to $L$, the degree does not change. Then by recurrence, the desired result is established.
\end{proof}

Here is some examples for homogeneous functions. Let $a,b$ be non-negative integers,
\\
$\bullet$ $x^ar^{-b}$ is homogeneous of degree $(a-b)$,
\\
$\bullet$ $t^ar^{-b}$ is homogeneous of degree $(a-b)$, because $t^ar^{-b} = (t/r)^a r^{a-b}$ and $(t/r)$ is homogeneous of degree zero.

\subsection{Decomposition of commutators}
In this subsection, we investigate the commutation relation among the vector fields $\del_{\alpha}$, $L$ and $\delt_1$. First, in the region $\Hb_{[s_0,s_1]}$
\begin{equation}\label{eq 1 comm-ext}
[\del_\alpha,\delt_1] = 0, \quad
[L,\delt_1] = -(x/r)\delt_1
= \left\{
\aligned
&-1,\quad \Hb^+_{[s_0,s_1]},
\\
&1,\quad \Hb^-_{[s_0,s_1]}.
\endaligned
\right.
\end{equation}
Then we establish the following decomposition:
\begin{lemma}\label{lem 1 decompo-comm-ext}
Let $u$ be a function defined in $\Fcal_{[s_0,s_1]}$, sufficiently regular. Then in the interior of $\Hb_{[s_0,\infty)}$,
\begin{equation}\label{eq 1 lem 1 decompo-comm-ext}
[\del^IL^j,\delt_1]u =
\sum_{k<j}\Lambda^{Ij}_{k}\delt_1\del^IL^ku
\end{equation}
where $\Lambda^{Ij}_{k}$ are locally constant, i.e., they are constant on each component of connection of $\Hb_{[s_0,s_1]}$.
\end{lemma}
\begin{proof}
We first establish the following decomposition:
\begin{equation}\label{eq 1 proof lem 1 decompo-comm-ext}
[L^j,\delt_1]u = \sum_{k<j}\Lambda^j_k\delt_1L^k u
\end{equation}
where $\Lambda_k^j$ are locally constant. This is by induction. \eqref{eq 1 comm-ext} shows the case of $j=1$. The we remark the following calculation:
$$
\aligned
\,[LL^j,\delt_1]u =& [L,\delt_1]L^ju + L\big([L^j,\delt_1]u\big) = -(x/r)\delt_1L^ju + \sum_{k<j}L(\Lambda^j_k\delt_1L^ku)
\\
=&-(x/r)\delt_1L^ju  + \sum_{k<j}\Lambda^j_k\delt_1LL^k + \sum_{k<j}\Lambda^j_k[L,\delt_1]L^ku
\\
=&-(x/r)\delt_1L^ju  + \sum_{k<j}\Lambda^j_k\delt_1LL^k  - (x/r)\sum_{k<j}\delt_1L^ku
\endaligned
$$
where we have applied the fact that by induction, $L\Lambda^j_k = 0$. Then by induction \eqref{eq 1 proof lem 1 decompo-comm-ext} is concluded.


Then we consider $[\del^IL^j,\delt_1]$. Recall \eqref{eq 1 comm-ext}, $[\del^I,\delt_1] = 0$. Then
$$
\aligned
\,[\del^IL^j,\delt_1]u =& \del^I\big([L^j,\delt_1]u\big) = \sum_{k<j}\del^I\big(\Lambda^j_k\delt_1L^ku\big)
=\sum_{k<j}\Lambda^j_k\delt_1\del^I L^ku.
\endaligned
$$
Here we remark that $\del^{I_1}\big(\Lambda^j_k\big) =0$. This concludes \eqref{eq 1 lem 1 decompo-comm-ext}.
\end{proof}

Then, we establish the following results:
\begin{lemma}\label{lem 1 high-order-ext}
For $Z^J$ a $N-$order operator of type $(i,j,l)$ with $l\geq 1$, the following bounds hold:
\begin{equation}\label{eq 1 lem 1 high-order-ext}
Z^Ju = \sum_{k\leq j,|I|=i}\Lambda^J_{Ik}\delt_1^l\del^IL^k u
\end{equation}
where $\Lambda^J_{Ik}$ are locally constant.
\end{lemma}
\begin{proof}
This is by an induction on $l$.  We remark that by \eqref{eq 1 lem 1 high-order}, $Z^Ju$ can be written as a finite linear combination of the following terms with constant coefficients:
$$
\del^{I_1}L^{j_1}\delt_1Z^{J'}u
$$
where $Z^{J'}$ is of type $(i',j',l')$ with $i = i'+|I_1|$, $j = j'+j_1$, $l = l' + 1$.

When $l=1$ it is guaranteed by \eqref{eq 1 lem 1 decompo-comm-ext} and \eqref{eq 1 lem 1 high-order} (applied on $Z^{J'}$). When $l>1$, we have $l'\geq 1$ and
$$
\aligned
\del^{I_1}L^{j_1}\delt_1Z^{J'}u =& \delt_1\big(\del^{I_1}L^{j_1}Z^{J'}u\big) + [\del^{I_1}L^{j_1},\delt_1]Z^{J'}u
\\
=& \delt_1\big(\del^{I_1}L^{j_1}Z^{J'}u\big) + \sum_{ k<j_1 }\Lambda^{I_1j_1}_{k} \delt_1\del^{I_1}L^{k}Z^{J'}u.
\endaligned
$$
We denote by $Z^{J''} := \del^{I_1}L^{j_1}Z^{J'}$ which is of type $(i,j,l')$, by $Z^{J_k'} := \del^{I_1}L^kZ^{J'}$ which is of type $(i,j_k',l')$ with $l' = l-1$ and $j_k' = k + j'<j$. Then by the assumption of induction,
$$
\aligned
\del^{I_1}L^{j_1}\delt_1Z^{J'}u
=&\delt_1Z^{J''}u + \sum_{ k<j_1 }\Lambda^{I_1j_1}_{k} \delt_1Z^{J_k'}u
\\
=&\sum_{k\leq j,|I|=i}\Lambda^{J''}_{Ik} \delt_1^l\del^IL^k u
+ \sum_{k<j_1 }\!\!\!\sum_{|I|=i\atop k_2\leq k+j'}\Lambda^{I_1j_1}_{k} \Lambda^{J_k'}_{I k_2}\delt_1\delt_1^{l'}\del^{I}L^{k_2}u
\endaligned
$$
which concludes by induction the desired result.

\end{proof}

\subsection{Basic $L^2$ bounds}
\begin{lemma}\label{lem 1 L2-basic-ext}
Let $u$ be a function defined in $\Hb_{[s_0,s_1]}$, sufficiently regular. Then the following bounds hold:
\begin{equation}\label{eq 1 lem 1 L2-basic-ext}
\|(1+r)^{-1}(1+|r-t|)^{\gamma} L^{j+1} u\|_{L^2(\Hb_s)}\leq C\EE_{N,\gamma}(s,u)^{1/2},  \quad j\leq N,
\end{equation}
\\
$-$ When $K$ being of type $(i,j,0)$, $i\geq 1$ and $|K|\leq N+1$,
\begin{equation}\label{eq 2 lem 1 L2-basic-ext}
\|(1+|r-t|)^{\gamma}\zeta Z^Ku\|_{L^2(\Hb_s)}\leq C\EE_{N,\gamma}(s,u)^{1/2},
\end{equation}
\\
$-$ When $K$ being of type $(i,j,l)$, $l\geq 1$ and $|K|\leq N+1$,
\begin{equation}\label{eq 3 lem 1 L2-basic-ext}
\|(1+|r-t|)^{\gamma} Z^K u\|_{L^2(\Hb_s)}\leq C\EE_{N,\gamma}(s,u)^{1/2},
\end{equation}
\\
$-$ When $K$ being of type $(i,j,l)$ and $|K|\leq N$,
\begin{equation}\label{eq 4 lem 1 L2-basic-ext}
\|c(1+|r-t|)^{\gamma}Z^K u\|_{L^2(\Hb_s)}\leq C\EE_{N,\gamma,c}(s,v)^{1/2}.
\end{equation}
\end{lemma}
\begin{proof}
For \eqref{eq 4 lem 1 L2-basic-ext}, we only need to recall the structure of the energy \eqref{eq flat-energy-flat}, \eqref{eq trans-energy-flat} and lemma \ref{lem 1 trans-energy} and the decomposition \eqref{eq 1 lem 1 high-order-ext}. In the same manner, by \eqref{eq 1 lem 1 decompo-comm-ext}, \eqref{eq 1 lem 1 L2-basic-ext} and \eqref{eq 2 lem 1 L2-basic-ext} are direct.

For \eqref{eq 3 lem 1 L2-basic-ext}, we need to remark that
$$
\delt_1 = (x/r)\del_t + \del_x
$$
where the coefficients are locally constant. So $\delt_1^l\del^i$ is a finite linear combination of
$$
\Lambda_I\delt_1\del^I, \quad |I| = i+l-1
$$
with $\Lambda_I$ locally constant, thus uniformly bounded in $\Hb_{[s_0,\infty)}$. Thus \eqref{eq 3 lem 1 L2-basic-ext} is established.
\end{proof}

\subsection{Structure of Hessian forms}
In this subsection we will analysis the terms $\del_{\alpha}\del_{\beta}u$, which are components of Hessian form of $u$. For the convenience of discussion, we introduce the following notation:
$$
\Fcal_{[s_0,s_1]}^L := \Hb_{[s_0,s_1]}\cap \{r-t\leq 1\}, \quad \Fcal_s^L := \Hb_s\cap \{r-t\leq 1\}
$$
and
$$
\Fcal_{[s_0,s_1]}^* := \Hb_{[s_0,s_1]}\cap \{r-t\geq 1\}, \quad \Fcal_s^* := \Hb_s\cap \{r-t\geq 1\}.
$$

Furthermore, we introduce the region
$$
\aligned
&\Fcal^E_{[s_0,s_1]} := \Hb_{[s_0,s_1]}\cap \{r\geq 2t \}, \quad \Fcal^E_s := \Hb_s\cap \{r\geq 2t\},
\\
&\Fcal^I_{[s_0,s_1]} := \Hb_{[s_0,s_1]}\cap \{r\leq 2t \}, \quad \Fcal^I_s := \Hb_s\cap \{r\leq 2t\}.
\endaligned
$$
 To get start we make some preparations.
\begin{lemma}\label{lem 1 (r-t)}
In the region $\Fcal^*_{[s_0,\infty)}$, the following bound holds:
\begin{equation}\label{eq 1 lem 1 (r-t)}
|\del^IL^j\big((r-t)^{-1}\big)|\leq C(1+|r-t|)^{-1-|I|}
\end{equation}
\end{lemma}
\begin{proof}
We first establish the following relation:
\begin{equation}\label{eq 1 pr lem 1 (r-t)}
L^j\big((r-t)^{-1}\big) = \lambda^j(r-t)^{-1}
\end{equation}
with $\lambda^j$ a locally constant function. This is by induction. We just remark that
$$
L(r-t) = (x\del_t+t\del_x)\big((r-t)^{-1} \big)= (x/r)(r-t)^{-1}.
$$
Then, suppose that \eqref{eq 1 pr lem 1 (r-t)} holds for $j$. Then
$$
L^{j+1}(r-t) = L^j\big((x/r)(r-t)^{-1}\big) = (x/r)L^j(r-t)^{-1} = -(x/r)\lambda^j(r-t)^{-1}
$$
where we have applied the fact that $x/r$ is locally constant.

Then we establish the following relation:
\begin{equation}\label{eq 2 pr lem 1 (r-t)}
\del^I \big((r-t)^{-1}\big) = C^I(r-t)^{-1-|I|}
\end{equation}
where $C^I$ is a locally constant function determined by $I$. This is also by induction. We only need to remark that
$$
\del_t\big((r-t)^{-1}) = (r-t)^{-2},\quad \del_x \big((r-t)^{-1}\big) = -(x/r)(r-t)^{-2}
$$
and (by the assumption of induction)
$$
\aligned
\del_{\alpha}\del^I\big((r-t)^{-1}\big) =& \del_{\alpha}\big(C^I(r-t)^{-1-|I|}\big) = C^I\del_{\alpha}\big((r-t)^{-1-|I|}\big)
\\
=& C^I (-1-|I|)C_{\alpha}(r-t)^{-1-(|I|+1)}
\endaligned
$$
where $C_{\alpha} = 1$ when $\alpha=0$ and $C_{\alpha} = -(x/r)$ when $\alpha = 1$. By induction this concludes  \eqref{eq 2 pr lem 1 (r-t)}

Now combine \eqref{eq 1 pr lem 1 (r-t)} and \eqref{eq 2 pr lem 1 (r-t)}, the following relation holds:
\begin{equation}\label{eq 3 pr lem 1 (r-t)}
\del^IL^j \big((r-t)^{-1}\big) = C^{Ij}(r-t)^{-1-|I|}
\end{equation}
with $C^{Ij}$ locally constant. Then \eqref{eq 1 lem 1 (r-t)} follows directly.
\end{proof}

We consider the region $\Fcal_{[s_0,s_1]}^I$. We remark the following identity:
\begin{equation}\label{eq 1' decompo-box-semi}
\Box = \big(1-(r/t)^2\big)\del_t\del_t + t^{-1}\left((x/t)\del_tL - \del_xL - (x/t)\del_x + \del_t\right)
\end{equation}

and this leads to
\begin{equation}\label{eq 2 decompo-box-semi}
(r-t)\del_t\del_t u = \frac{t^2}{r+t}\Box u + \frac{x}{r+t}(\del_tL -\del_x)u + \frac{t}{r+t}(\del_t-\del_xL)u.
\end{equation}

Then we obtain the following estimate:
\begin{lemma}\label{lem 1 hessian-ext}
Let $u$ be a function defined in $\Hb_{[s_0,s_1]}$, sufficiently regular. Then the following bounds hold:
\begin{equation}\label{eq 2 lem 1 hessian-ext}
|(r-t)\del_t\del_t\del^IL^j u(t,x)|\leq |t^2r^{-1}\Box \del^IL^j u(t,x)|
+ C\!\!\!\!\sum_{\alpha,j'\leq j+1\atop |I'|\leq|I|}|\del_{\alpha}\del^{I'}L^{j'}u(t,x)|.
\end{equation}

Furthermore, when $(t,x)\in \Fcal^E_s$:
\begin{equation}\label{eq 2.5 lem 1 hessian-ext}
|r\del_x\del_x \del^IL^j u(t,x)|\leq Cr|\Box \del^IL^j u(t,x)|
 + C\!\!\!\!\sum_{\alpha,j'\leq j+1\atop |I'|\leq|I|}|\del_{\alpha}\del^{I'}L^{j'}u(t,x)|,
\end{equation}

\begin{equation}\label{eq 3.3 lem 1 hessian-ext}
|t\del_t\del_x \del^IL^ju(t,x)|\leq |t^2r^{-1}\Box \del^IL^j u(t,x)|
+ C\!\!\!\!\sum_{\alpha,j'\leq j+1\atop |I'|\leq|I|}|\del_{\alpha}\del^{I'}L^{j'}u(t,x)|,
\end{equation}

When $(t,x)\in\Fcal^I_s$,
\begin{equation}\label{eq 3 lem 1 hessian-ext}
|t\delt_1\del_t \del^IL^j u(t,x)|\leq |t\Box \del^IL^j u(t,x)|
+ C\!\!\!\!\sum_{\alpha,j'\leq j+1\atop |I'|\leq|I|}|\del_{\alpha}\del^{I'}L^{j'}u(t,x)|,
\end{equation}

\begin{equation}\label{eq 4 lem 1 hessian-ext}
|t\delt_1\del_x \del^IL^j u(t,x)|\leq |t\Box \del^IL^j u(t,x)|
+ C\!\!\!\!\sum_{\alpha,j'\leq j+1\atop |I'|\leq|I|}|\del_{\alpha}\del^{I'}L^{j'}u(t,x)|,
\end{equation}
and
\begin{equation}\label{eq 5 lem 1 hessian-ext}
\aligned
|t\delt_1\delt_1 \del^IL^j u(t,x)|\leq& \big|(r-t)\Box \del^IL^j u(t,x)\big|
\\
&+ C\!\!\!\!\sum_{j'\leq j+1\atop |I'|\leq|I|}|\delt_1\del^{I'}L^{j'}u(t,x)|
+ \sum_{j'\leq j+1\atop |I'|\leq|I|}|(r-t)t^{-1}\del_{\alpha}\del^{I'}L^{j'}u(t,x)|.
\endaligned
\end{equation}

\end{lemma}

\begin{proof}
We first remark that \eqref{eq 2 lem 1 hessian-ext} is a direct application of \eqref{eq 2 decompo-box-semi} on $\del^IL^j u$.

For \eqref{eq 2.5 lem 1 hessian-ext}, we remark the following identity:
\begin{equation}\label{eq 2.5 pr lem 1 hessian-ext}
\del_x\del_x u = \del_t\del_tu -\Box u
\end{equation}
then we apply \eqref{eq 2 lem 1 hessian-ext}.

For \eqref{eq 3.3 lem 1 hessian-ext}, we remark the following identity:
\begin{equation}\label{eq 3' pr lem 1 hessian-ext}
\del_t\del_x u = t^{-1}\del_tLu - (x/t)\del_t\del_t - t^{-1}\del_x
\end{equation}

For \eqref{eq 3 lem 1 hessian-ext}, we need to remark the following identity:
\begin{equation}\label{eq 1 proof prop 1 L2-hessian-ext}
\delt_1\del_t u = \del_t\delt_1 u = t^{-1}\del_t L u - t^{-1}\del_x u -  t^{-1}(x/r)(r-t)\del_t\del_tu
\end{equation}
and then combine it with \eqref{eq 2 lem 1 hessian-ext}.

For \eqref{eq 4 lem 1 hessian-ext} we combine
\begin{equation}\label{eq 1' proof prop 1 L2-hessian-ext}
\delt_1\del_xu =  \del_x\delt_1 u = t^{-1}\delt_1L u - (x/r)t^{-1}\delt_1 u - (x/t)\delt_1\del_tu.
\end{equation}
with \eqref{eq 2 lem 1 hessian-ext}.

For \eqref{eq 5 lem 1 hessian-ext}, we recall the following identity:
\begin{equation}\label{eq 3 pr lem 1 hessian-ext}
\delt_1\delt_1 u = t^{-1}\delt_1Lu - t^{-1}(x/r)\delt_1u - (x/r)(r-t)t^{-1}\delt_1\del_t u
\end{equation}

\end{proof}

Then we establish the following bounds:
\begin{lemma}\label{lem 1.5 hessian-ext}
Let $u$ be a function defined in $\Hb_{[s_0,s_1]}$, sufficiently regular. Then the following bounds hold:
\\
When $(t,x)\in \Fcal^I_s$,
\begin{equation}\label{eq 1 lem 1 hessian-ext}
\aligned
|(1+|r-t|)\del^IL^j \del_t\del_t u(t,x)|
\leq& \sum_{|I'|\leq |I|\atop j'\leq j}|r\Box \del^{I'}L^{j'}u(t,x)|
\\
&+ |\del^IL^j \del_t\del_t u(t,x)|
+ C\!\!\!\!\sum_{\alpha,j'\leq j+1\atop |I'|\leq|I|}|\del_{\alpha}\del^{I'}L^{j'}u(t,x)|,
\endaligned
\end{equation}
\begin{equation}\label{eq 2 lem 1.5 hessian-ext}
|t\del^IL^j \delt_1\del_t u(t,x)|\leq \sum_{|I'|=|I|\atop j'\leq j}|t\Box \del^{I'}L^{j'} u(t,x)|
+ C\!\!\!\!\sum_{\alpha,j'\leq j+1\atop |I'|\leq|I|}|\del_{\alpha}\del^{I'}L^{j'}u(t,x)|,
\end{equation}
\begin{equation}\label{eq 3 lem 1.5 hessian-ext}
\aligned
|t\del^IL^j \delt_1\delt_1 u(t,x)|\leq& \sum_{|I'|=|I|\atop j'\leq j}|(r-t)\Box \del^{I'}L^{j'} u(t,x)|
\\
&+ C\!\!\!\!\sum_{j'\leq j+1\atop |I'|\leq|I|}|\delt_1\del^{I'}L^{j'}u(t,x)|
+ \sum_{j'\leq j+1\atop |I'|\leq|I|}|(r-t)t^{-1}\del_{\alpha}\del^{I'}L^{j'}u(t,x)|.
\endaligned
\end{equation}
\\
When $(t,x)\in \Fcal^E_s$,
\begin{equation}\label{eq 4 lem 1.5 hessian-ext}
\aligned
|t\del^IL^j \del_{\alpha}\del_{\beta} u(t,x)|\leq& C\sum_{|I'|\leq|I|\atop j'\leq j}|t\Box \del^IL^j u(t,x)|
+ C\!\!\!\!\sum_{\alpha,j'\leq j+1\atop |I'|\leq|I|}|\del_{\alpha}\del^{I'}L^{j'}u(t,x)|.
\endaligned
\end{equation}
\end{lemma}
\begin{proof}
We firstly concentrate on \eqref{eq 1 lem 1 hessian-ext}. In $\Fcal^*_{[s_0,s_1]}$ we write \eqref{eq 2 decompo-box-semi} into the following form:
\begin{equation}\label{eq 0 pr lem 1 hessian-ext}
\del_t\del_t u = \frac{t^2}{r^2-t^2}\Box u + \frac{x}{r^2-t^2}(\del_tL -\del_x)u + \frac{t}{r^2-t^2}(\del_t-\del_xL)u.
\end{equation}

We derive this identity with respect to $\del^IL^j$ and apply \eqref{eq 1 lem 1 (r-t)}. To do so, we firstly remark
$$
|\del^IL^j(t^2)(r+t)^{-1}|\leq Cr.
$$
This is due to the homogeneity, and in $\Fcal^*_{[s_0,s_1]}$ by applying \eqref{eq 3 pr lem 1 (r-t)},
\begin{equation}\label{eq 1 pr 1 lem hessian-ext}
\big|\del^IL^j(t^2(r^2-t^2)^{-1})\big| = \big|\del^IL^j\big(t^2(r+t)^{-1}(r-t)^{-1}\big)\big| \leq \frac{Cr}{1+|r-t|}.
\end{equation}
In the same manner,
\begin{equation}\label{eq 2 pr lem 1 hessian-ext}
\big|\del^IL^j\big(x(r^2-t^2)^{-1}\big)\big| + \big|\del^IL^j\big(t(r^2-t^2)^{-1}\big)\big|\leq C(1+|r-t|)^{-1}.
\end{equation}

Then, we remark that for the first term in RHD of \eqref{eq 2 decompo-box-semi},
$$
\del^IL^j\big(t^2(r^2-t^2)^{-1}\Box u\big) =
\sum_{I_1+I_2=I\atop j_1+j_2=j}\del^{I_1}L^{j_1}(t^2(r^2-t^2)^{-1})\cdot \del^{I_2}L^{j_2}\Box u.
$$
Then we apply \eqref{eq 1 pr 1 lem hessian-ext}, this term is bounded by the first term in RHD of \eqref{eq 1 lem 1 hessian-ext}. The rest terms in RHS of \eqref{eq 0 pr lem 1 hessian-ext} are bounded in the same manner, we omit the detail. Then \eqref{eq 1 lem 1 hessian-ext} is established in the region $\Fcal^I_s\cap \Fcal^*_s$.
In $\Fcal^L_s$, \eqref{eq 1 lem 1 hessian-ext} is trivial.

\eqref{eq 2 lem 1.5 hessian-ext} is based on \eqref{eq 1 lem 1 high-order-ext}. We remark that $\del^IL^j \delt_1\del_{\alpha}$ is a finite linear combination of the following terms with constant coefficients:
$$
\delt_1\del_{\alpha}\del^{I'}L^{j'} u, \quad |I'| = |I|,\quad j'\leq j.
$$
Then we apply \eqref{eq 3 lem 1 hessian-ext} and \eqref{eq 4 lem 1 hessian-ext}.

In the same manner, \eqref{eq 3 lem 1.5 hessian-ext} is guaranteed by \eqref{eq 1 lem 1 high-order-ext} combined with \eqref{eq 5 lem 1 hessian-ext}.

Finally, \eqref{eq 4 lem 1.5 hessian-ext} is deduced from \eqref{eq 2 lem 1 hessian-ext}, \eqref{eq 2.5 lem 1 hessian-ext} and \eqref{eq 3.3 lem 1 hessian-ext} combined with \eqref{eq 1 lem 1 high-order-ext}.
\end{proof}


\subsection{$L^2$ bounds on Hessian form}
Based on the above subsection, we establish the following $L^2$ bounds.
\begin{proposition}[$L^2$ bounds on Hessian form of wave component]\label{prop 1 L2-hessian-ext}
Let $u$ be a function defined in $\Fcal_{[s_0,s_1]}$, sufficiently regular. Then the following estimate holds for $|I|+j\leq N-1$,
\begin{equation}\label{eq 1 prop 1 L2-hessian-ext}
\aligned
\|tr^{\gamma}\del^IL^j \del_{\alpha}\del_{\beta} u\|_{L^2(\Fcal^E_s)}\leq& C\sum_{|I'|=|I|\atop |j'|\leq|j|}\|r^{1+\gamma}\Box \del^{I'}L^{j'} u\|_{L^2(\Fcal^E_s)}+ C\EE_{N,\gamma}(s,u)^{1/2},
\endaligned
\end{equation}
\\
\begin{equation}\label{eq 2 prop 1 L2-hessian-ext}
\aligned
\|(1+|r-t|)^{1+\gamma}\zeta \del^IL^j\del_t\del_t u \|_{L^2(\Fcal^I_s)}\leq& C\sum_{|I'|=|I|\atop |j'|\leq|j|}\|t(1+|r-t|)^{\gamma}\zeta \Box \del^{I'}L^{j'} u\|_{L^2(\Fcal^I_s)}
\\
&+ C\EE_{N,\gamma}(s,u)^{1/2},
\endaligned
\end{equation}
\begin{equation}\label{eq 3 prop 1 L2-hessian-ext}
\aligned
\|t(1+|r-t|)^{\gamma}\zeta \del^IL^j\delt_1\del_t u \|_{L^2(\Fcal^I_s)}\leq& C\sum_{|I'|=|I|\atop |j'|\leq|j|}\|t(1+|r-t|)^{\gamma}\zeta \Box \del^{I'}L^{j'} u\|_{L^2(\Fcal^I_s)}
\\
&+ C\EE_{N,\gamma}(s,u)^{1/2},
\endaligned
\end{equation}
\begin{equation}\label{eq 4 prop 1 L2-hessian-ext}
\aligned
\|t(1+|r-t|)^{\gamma}\del^IL^j\delt_1\delt_1 u \|_{L^2(\Fcal^I_s)}\leq& C\sum_{|I'|=|I|\atop |j'|\leq|j|}\|(1+|r-t|)^{1+\gamma} \Box \del^{I'}L^{j'} u\|_{L^2(\Fcal^I_s)}
\\
&+ C\EE_{N,\gamma}(s,u)^{1/2},
\endaligned
\end{equation}

\end{proposition}
\begin{proof}
This is a direct application of lemma \ref{lem 1.5 hessian-ext}. \eqref{eq 1 prop 1 L2-hessian-ext} is due to \eqref{eq 4 lem 1.5 hessian-ext} and the fact that on $\Fcal^E_s$, $\zeta = 1$.

For the bounds on $\Fcal^I_s$, we only write the proof of \eqref{eq 4 prop 1 L2-hessian-ext}. Remark that in $\Fcal^L_{[s_0,s_1]}$,
$$
0\leq \frac{|t-r|}{t}\leq \zeta
$$
This is because that in $\Fcal^L_s$, by \eqref{eq xi-zeta},
$$
\frac{|t-r|}{t}\leq 1/t\leq 1/s\leq \zeta.
$$
\end{proof}


\section{Decay bounds in transition and exterior region}\label{sec decay-ext}
\subsection{Basic decay bounds}
First, by \eqref{eq 1 sobolev-ext} and \eqref{eq flat-energy-flat}, in the region when $(t,x)\in \Fcale_s$, the following bounds are direct:
\begin{equation}\label{eq 1 decay ext-far}
|\del_{\alpha} u(t,x)| + |\delt_{\alpha} u(t,x)|\leq C(1+r)^{-\gamma}\EE_{1,\gamma}(s,u)^{1/2},
\end{equation}
\begin{equation}\label{eq 1' decay ext-far}
|Lu(t,x)|\leq C(1+r)^{1-\gamma}\EE_{1,\gamma}(s,u)^{1/2}
\end{equation}
and when $c>0$,
\begin{equation}\label{eq 2 decay ext-far}
|cu(t,x)|\leq C(1+r)^{-\gamma}\EE_{0,\gamma}(s,u)^{1/2}.
\end{equation}
We denote by $\del\Kcal^2 = \{t = 2r\}$ and
$$
\del \Kcal^2_s := \del \Kcal^2\cap \Fcal_s = \{(T(s),2T(s)),(T(s),-2T(s))\}.
$$
Remark that the above bounds \eqref{eq 1 decay ext-far} and \eqref{eq 2 decay ext-far} also hold on $\Kcal^2_s$. Recall  \eqref{eq 1 prop 1 feuille}, thus
$$
T(s) \sim s^2.
$$

Then we concentrate on the region $\Fcali_{[s_0,s_1]}$. To do so we need the following lemma:
\begin{lemma}\label{lem 1 19-11-2017}
There exits a positive constant $C$ such that
\begin{equation}\label{eq 1 20-11-2017}
|\xi_s'(r) |\leq C \big(1-\xi_s(r)\big)^{1/2}, \quad \frac{s^2-1}{2}\leq r\leq \frac{s^2+1}{2}.
\end{equation}
\end{lemma}
\begin{proof}
We will prove that
\begin{equation}\label{eq 1 proof lem 1 19-11-2017}
0\leq \frac{\rho^2(x)}{\int_{-\infty}^x\rho(y)dy}\leq C,\quad  x\in(-1/2,1/2]
\end{equation}
with $C$ a universal constant. Remark that for $x>-1/2$, $\int_{-\infty}^x\rho(y)dy>0$. For the above bound we only need to prove that
$$
0\leq \lim_{x\rightarrow -1/2^+}\frac{\rho^2(x)}{\int_{-\infty}^x\rho(y)dy}<+\infty.
$$
Recall that when $x\rightarrow -(1/2)^+$ both $\rho^2(x)$ and $\int_{-\infty}^x\rho(y)dy$ tend to zero. So by l'H\^opital's rule,
$$
\lim_{x\rightarrow -1/2^+}\frac{\rho^2(x)}{\int_{-\infty}^x\rho(y)dy} = \lim_{x\rightarrow -1/2^+}\frac{\rho(x)\rho'(x)}{\rho(x)} = 0
$$
Thus \eqref{eq 1 proof lem 1 19-11-2017} is guaranteed. Then by the definition of $\chi$ and $\xi_s$, we see that \eqref{eq 1 20-11-2017} is established.
\end{proof}

Then we establish the following bound:
\begin{lemma}\label{lem 1 decompo-trans}
Let $u$ be a function defined in $\Fcal_{[s_0,s_1]}$, sufficiently regular. For $s_0\leq s\leq s_1$, the following bound holds:
\begin{equation}\label{eq 1 decompo-trans}
|\delb_x\delb_xu|\leq C\sum_{\alpha}|\delb_x\del_{\alpha}u| + C\zeta(s,x)|\del_tu|.
\end{equation}
\end{lemma}
\begin{proof}
This is by direct calculation. We remark that
$$
\delb_x\delb_xu = \delb_x\del_x + \frac{\xi_s(r)x}{\sqrt{s^2+x^2}}\delb_x\del_tu + \delb_x\Big(\frac{\xi_s(r)x}{\sqrt{s^2+x^2}}\Big)\del_tu
$$
Recall that $\Big|\frac{\xi_s(r)x}{\sqrt{s^2+x^2}}\Big|\leq 1$, we only need to treat the lats term. Direct calculation shows that (remark that $s$ is constant along the direction of $\delb_x$)
$$
\delb_x\Big(\frac{\xi_s(r)x}{\sqrt{s^2+x^2}}\Big) = \frac{s^2\xi_s(r)}{(s^2+x^2)^{3/2}} + \frac{\xi_s'(r)r}{\sqrt{s^2+x^2}}.
$$

In the right-hand-side of the above identity, the first term is bounded by $\zeta(s,x)$ (by \eqref{eq xi-zeta}). The second term can be controlled as following
$$
\frac{|\xi_s'(r)|r}{\sqrt{s^2+x^2}} \leq |\xi_s'(r)|\leq C(1-\xi_s(r))^{1/2}\leq C\zeta(s,r)
$$
where we have applied \eqref{eq 1 20-11-2017} and \eqref{eq xi-zeta} and this complete the proof.
\end{proof}

Based on the above estimates, we are ready to prove the following result:
\begin{lemma}\label{lem 1 decay-trans}
Let $u$ be a function defined in $\Fcal_{[s_0,s_1]}$, sufficiently regular. For $s_0\leq s\leq s_1$, the following bound holds for $(t,x)\in\Fcali_s$:
\begin{equation}\label{eq 1 lem 1 decay-trans}
|u(t,x)|\leq C\EE_{1,\gamma}(s,u)^{1/2}(1+t)^{1-\gamma} + \sup_{\del\Kcal_s^2}|u|,\quad 0 < \gamma < 1,
\end{equation}
\begin{equation}\label{eq 1' lem 1 decay-trans}
|u(t,x)|\leq C\EE_{1,\gamma}(s,u)^{1/2}(1+|r-t|)^{1-\gamma} + \sup_{\del\Kcal_s^2}|u|,\quad \gamma>1.
\end{equation}
\end{lemma}
\begin{proof}
First, by \eqref{eq 1 decompo-trans} combined  the expression of energy \eqref{eq trans-energy-flat},
\begin{equation}\label{eq 1 proof lem 1 decay-trans}
\|\delb_x\delb_x u\|_{L^2(\Tcal_s)}\leq C\EE_{1,\gamma}(s,u)^{1/2}.
\end{equation}
Also by \eqref{eq flat-energy-flat}
\begin{equation}\label{eq 2 proof lem 1 decay-trans}
\|(1+w_{\gamma})\del_x\del_x u\|_{L^2(\Pcal_s)}\leq C\EE_{1,\gamma}(s,u)^{1/2}.
\end{equation}
By \eqref{eq 1 sobolev-trans} and \eqref{eq 1 sobolev-ext}, for $(t,x)\in\Tcal_s$,
\begin{equation}\label{eq 3 proof lem 1 decay-trans}
|\delb_x u(t,x)|\leq C\EE_{1,\gamma}(s,u)^{1/2},
\end{equation}
for $(t,x)\in \Pcal_s$
\begin{equation}\label{eq 4 proof lem 1 decay-trans}
|(1+w_\gamma)\del_xu(t,x)|\leq C\EE_{1,\gamma}(s,u)^{1/2}.
\end{equation}
For the convenience of discussion, we denote by $\delb_x$ the tangent derivative of $\Fcal_s$ on both $\Tcal_s$ and $\Pcal_s$. Remark that on $\Tcal_s$ it coincides the original definition in tangent frame and on $\Pcal_s$ it equals to $\del_x$. Thus \eqref{eq 3 proof lem 1 decay-trans} and \eqref{eq 4 proof lem 1 decay-trans} leads to
\begin{equation}\label{eq 5 proof lem 1 decay-trans}
|(1+w_\gamma)\delb_xu(t,x)|\leq C\EE_{1,\gamma}(s,u)^{1/2}.
\end{equation}

Then we integrate along $\Fcal_s$. In the following we suppose that $(s^2-1)/2\leq x$. For the case $x\leq -(s^2-1)/2$ we only need to take $\tilde{u}(t,x) = u(t,-x)$. Recall that $t = T(s,r)$, and write
$$
u_s(x) = u(T(s,r),x)
$$
the restriction of $u$ on $\Hb_s$. Remark that
$$
u_s'(x) = \delb_xu(T(s,r),x).
$$
Then
$$
u_s(2T(s)) - u_s(x) = \int_x^{2T(s)}u_s'(y) dy
$$
which leads to (for $0<\gamma<1$)
$$
\aligned
|u_s(x)|\leq& \big|u_s\big(2T(s)\big)| + \int_x^{2T(s)}|\delb_x(T(s,|y|),y)|dy
\\
\leq & \sup_{\del\Kcal_s^2}|u| + C\Ecal^1(s,u)^{1/2}\int_x^{2T(s)}|1+w_\gamma(s,y)|^{-1}dy
\\
\leq & \sup_{\del\Kcal_s^2}|u| + C\Ecal^1(s,u)^{1/2}\big((1+t)^{1-\gamma} - (1+|r-t|)^{1-\gamma}\big)
\endaligned
$$
which leads to \eqref{eq 1 lem 1 decay-trans}.

When $\gamma>1$,
$$
\aligned
|u_s(x)|\leq& \big|u_s\big(2T(s)\big)| + \int_x^{2T(s)}|\delb_x(T(s,|y|),y)|dy
\\
\leq & \sup_{\del\Kcal_s^2}|u| + C\Ecal^1(s,u)^{1/2}\int_x^{2T(s)}(1+|y-t|)^{-\gamma}dy
\\
\leq & \sup_{\del\Kcal_s^2}|u| + C\Ecal^1(s,u)^{1/2}(1+|r-t|)^{1-\gamma}
\endaligned
$$
and this leads to \eqref{eq 1' lem 1 decay-trans}
\end{proof}

Now we are ready to establish the following decay estimate on wave component. That is, the following estimates make sense when $c=0$.
\begin{lemma}[Basic decay for wave component]\label{lem 1 decay-near-w}
Let $u$ be a function defined in $\Fcal_{[s_0,s_1]}$, sufficiently regular. The for $s_0\leq s\leq s_1$, the following bounds hold for $(t,x)\in\Hb_s$:
\begin{equation}\label{eq 1 lem 1 decay-near}
|\del_xu(t,x)| + |\del_tu(t,x)|\leq C\EE_{2,\gamma}(s,u)^{1/2}(1+|r-t|)^{-\gamma}, \quad 0<\gamma,
\end{equation}
and
\begin{equation}\label{eq 2 lem 1 decay-near}
|\delt_1 u(t,x)|\leq C\EE_{2,\gamma}(s,u)^{1/2}(1+r)^{-\gamma} ,\quad 0<\gamma<1,
\end{equation}
\begin{equation}\label{eq 2' lem 1 decay-near}
|\delt_1 u(t,x)|\leq C\EE_{2,\gamma}(s,u)^{1/2}(1+r)^{-1}(1+|r-t|)^{1-\gamma} ,\quad \gamma>1.
\end{equation}
\end{lemma}

\begin{proof}
For \eqref{eq 1 lem 1 decay-near}, in the region $\Pcal_s$, we apply directly \eqref{eq 1 sobolev-ext} and \eqref{eq 1 sobolev-trans}. In the transition region $\Tcal_s$, we just remark that by \eqref{eq 3 proof lem 1 decay-trans},
$$
|\delb_x\del_x u(t,x)| + |\delb_x\del_t u(t,x)|\leq C\Ecal^E_{2,\gamma}(s,u)^{1/2}.
$$
Then integrate from the frontier of $\Tcal_s\backslash\Pcal_s$ to $(t,x)$, and remark that the width of $\Tcal_s$ is limited by $1$ (on $\Tcal_s$, $(s^2-1)/2\leq r\leq (s^2+1)/2$).

For \eqref{eq 2 lem 1 decay-near}, recall \eqref{eq 1 decay ext-far}, we only need to prove in the region $\Fcali_s$.   Remark the following relation
\begin{equation}\label{eq 1 null-frame}
\delt_1 u = \frac{x}{r}\del_tu + \del_x u= t^{-1}Lu + \frac{x}{r}\frac{t-r}{t}\del_tu,
\end{equation}
Then we apply \eqref{eq 1 lem 1 decay-trans} and \eqref{eq 1' lem 1 decay-trans} on $Lu$ (combined with \eqref{eq 1' decay ext-far}) and obtain:
\begin{equation}\label{eq L null-frame}
|Lu(t,x)|\leq \left\{
\aligned
CC_1\vep& (1+r)^{1-\gamma}\Ecal^E_{2,\gamma}(s,u)^{1/2},\quad 0<\gamma<1,
\\
CC_1\vep& (1+|r-t|)^{1-\gamma}\Ecal^E_{2,\gamma}(s,u)^{1/2},\quad \gamma>1,
\endaligned
\right.
\end{equation}
then substitute  these bounds together with \eqref{eq 1 lem 1 decay-near} into \eqref{eq 1 null-frame} and obtain the desired bound.


%

\end{proof}

Then we are about to establish the following decay bounds on high-order derivatives for wave component
\begin{proposition}\label{prop 1 decay-high-order-ext}
Let $u$ be a function defined in $\Fcal_{[s_0,s_1]}$, sufficiently regular. Let $Z^M$ be an $(N-1)$ order operator of type $(i,j,l)$ and $N\geq 1$. Then
\\
$-$ When $l=0$, $i\geq 1$,
\begin{equation}\label{eq 2 decay-high-order-ext}
|Z^Mu(t,x)| \leq C(1+|r-t|)^{-\gamma}\EE_{N,\gamma}(s,u)^{1/2},
\end{equation}
\\
$-$ When $l\geq 1$,
\begin{equation}\label{eq 3 decay-high-order-ext}
|Z^Mu(t,x)|\leq C(1+r)^{-\gamma}\EE_{N,\gamma}(s,u)^{1/2}, \quad 0<\gamma<1
\end{equation}
\begin{equation}\label{eq 3' decay-high-order-ext}
|Z^Mu(t,x)|\leq C(1+r)^{-1}(1+|r-t|)^{-\gamma+1}\EE_{N,\gamma}(s,u)^{1/2}, \gamma>1.
\end{equation}
\end{proposition}
\begin{proof}
These are due to \eqref{eq 1 lem 1 high-order-ext} combined with \eqref{eq 1 lem 1 decay-near}, \eqref{eq 2 lem 1 decay-near} and \eqref{eq 2' lem 1 decay-near}.
\end{proof}

\subsection{Decay bounds for Klein-Gordon component}
This subsection is devoted to a refined decay bound on Klein-Gordon component (i.e. the bounds in this section only make sense when $c>0$).
We write \eqref{eq 1' decompo-box-semi} into the following form:
\begin{equation}\label{eq 1 decompo-SHF-Box}
\Box v + c^2v - \big((1-(r/t)^2\big)\del_t\del_tv + t^{-1}R[v] = c^2v
\end{equation}
where
$$
R[v] = (x/t)\del_tLv - \del_xLv - (x/t)\del_xv + \del_tv.
$$

Then we establish the following bounds on Klein-Gordon equation. Remark that this bounds is not a ``basic bound'' (i.e. it concerns the structure of equation).
\begin{proposition}\label{prop 1 decay-KG}
Let $v$ be a solution to the following equation:
\begin{equation}\label{eq 1 prop 1 decay-KG}
\Box v + c^2v = f
\end{equation}
with $v$ and $f$ sufficiently regular, defined in $\Fcal_{[s_0,s_1]}$. Then the following estimate holds for $(t,x)\in\Hb_s$:
\begin{equation}\label{eq 2 prop 1 decay-KG}
|v(t,x)|\leq
\left\{
\aligned
&C(1+|r-t|)^{1-\gamma}t^{-1}\EE_{3,\gamma,c}(s,v)^{1/2}
+ C|f(t,x)|,
&&\quad (t,x)\in \Fcal^I_s
\\
&C(1+r)^{-\gamma}c^{-1}\EE_{0,\gamma,c}(s,v)^{1/2},&&\quad (t,x)\in\Fcal^E_s.
\endaligned
\right.
\end{equation}
Furthermore, suppose that $Z^M$ is a $(N-3)$ order operator of type $(i,j,0)$, then
\begin{equation}\label{eq 3 prop 1 decay-KG}
|Z^M v(t,x)|\leq
\left\{
\aligned
&C(1+|r-t|)^{1-\gamma}t^{-1}\EE_{N,\gamma,c}(s,v)^{1/2} + C\!\!\!\sum_{|I|=i,k\leq j}\!\!\!|\del^IL^k f(t,x)|,
\quad &&(t,x)\in \Fcal^I_s
\\
&C(1+r)^{-\gamma}\EE_{N-3,\gamma,c}(s,v)^{1/2},\quad &&(t,x)\in\Fcal^E_s
\endaligned
\right.
\end{equation}
\end{proposition}
\begin{proof}
When $(t,x)\in \Fcal^E_s$, we apply \eqref{eq 2 decay ext-far}.

For the region $\Fcal^I_s$, thanks to \eqref{eq 1 decompo-SHF-Box}, \eqref{eq 1 prop 1 decay-KG} is written as
\begin{equation}\label{eq 1 pr prop 1 KG}
\aligned
c^3v =& -\big(1-(x/t)^2\big)c\del_t\del_tv
- t^{-1}cR[v] + cf.
\endaligned
\end{equation}
Then we apply \eqref{eq 1 lem 1 decay-near} (remark that $t-1\leq |x|\leq 2t$)
$$
|cR[v](t,x)| + |c\del_t\del_tv(t,x)|\leq C(1+|r-t|)^{-\gamma}\Ecal^3_{\gamma,c}(s,v)^{1/2}.
$$
Substitute the above bounds into \eqref{eq 1 pr prop 1 KG}, \eqref{eq 2 prop 1 decay-KG} is established.

For \eqref{eq 3 prop 1 decay-KG}, we derive \eqref{eq 1 prop 1 decay-KG} with respect to $Z^M$:
$$
\Box Z^M v + c^2Z^M v = Z^M f.
$$
Then we apply \eqref{eq 2 prop 1 decay-KG} on $Z^M v$. Recall that $Z^Mv$ can be written as a finite linear combination of $\del^IL^k v$ with $k\leq j, |I|=i$. So the desired result is established.
\end{proof}


\subsection{Decay bounds on Hessian form}
Lemma \ref{lem 1.5 hessian-ext} combined with the Sobolev's inequalities in transition and exterior region, we have the following decay bounds on Hessian forms:
\begin{lemma}\label{lem 1 decay-hessian-ext}
Let $u$ be a function defined in $\Fcal_{[s_0,s_1]}$, sufficiently regular. Then the following estimate holds for $(t,x)\in \Fcal^E_s$ and $|I|+j\leq N-3$
\begin{equation}\label{eq 1 lem 1 decay-hessian-ext}
\aligned
|tr^{\gamma}\del^IL^j \del_{\alpha}\del_{\beta} u(t,x)|\leq C\sum_{|I'|\leq|I|\atop j'\leq j}|r^{1+\gamma}\Box \del^IL^j u(t,x)|
+ C\EE_{N,\gamma}(s,u)^{1/2}.
\endaligned
\end{equation}

For $(t,x)\in\Fcal^I_s$ and $|I|+j\leq N-3$,
\begin{equation}\label{eq 00 lem 1 decay-hessian-ext}
|(1+|r-t|)^{1+\gamma}\del^IL^j \del_t\del_t u(t,x)|\leq \sum_{|I'|=|I|\atop j'\leq j}|t(1+|r-t|)^{\gamma} \Box \del^{I'}L^{j'} u(t,x)|
+ C\EE_{N,\gamma}(s,u)^{1/2},
\end{equation}

\begin{equation}\label{eq 01 lem 1 decay-hessian-ext}
|t(1+|r-t|)^{\gamma} \del^IL^j \delt_1\del_t u(t,x)|\leq \sum_{|I'|=|I|\atop j'\leq j}|t(1+|r-t|)^{\gamma} \Box \del^{I'}L^{j'} u(t,x)|
+ C\EE_{N,\gamma}(s,u)^{1/2},
\end{equation}
\begin{equation}\label{eq 11 lem 1 decay-hessian-ext}
|t^{1+\gamma}\del^IL^j \delt_1\delt_1 u(t,x)|\leq \sum_{|I'|=|I|\atop j'\leq j}|t^{\gamma}(1+|r-t|)\Box \del^{I'}L^{j'} u(t,x)|
+ C\EE_{N,\gamma}(s,u)^{1/2},\quad 0<\gamma<1
\end{equation}
and
\begin{equation}\label{eq 11' lem 1 decay-hessian-ext}
|t^2(1+|r-t|)^{\gamma-1}\del^IL^j \delt_1\delt_1 u(t,x)|\leq \sum_{|I'|=|I|\atop j'\leq j}|t(1+|r-t|)^{\gamma}\Box \del^{I'}L^{j'} u(t,x)|
+ C\EE_{N,\gamma}(s,u)^{1/2},\quad \gamma>1.
\end{equation}

\end{lemma}
\begin{proof}
This is a direct application of lemma \ref{lem 1.5 hessian-ext} and proposition \ref{prop 1 decay-high-order-ext}. We remark that by proposition \ref{prop 1 decay-high-order-ext}, for $|I|+j\leq N-2$, and $\gamma>0$, the following quantities:
\begin{equation}\label{eq 1 pr lem 1 decay-hessian-ext}
|(1+|r-t|)^{\gamma}\del_{\alpha}\del^IL^ju(t,x)|
\end{equation}
is bounded by $C\EE_{N,\gamma}(s,u)$. Furthermore, the following quantities
\begin{equation}
\aligned
&|(1+r)^{\gamma}\delt_1\del^IL^j u(t,x)|, \text{ for }0<\gamma<1,\quad
\\
&|(1+r)(1+|r-t|)^{\gamma-1}\delt_1\del^IL^j u(t,x)|,\text{ for } \gamma>1
\endaligned
\end{equation}
are also bounded by $C\EE_{N,\gamma}(s,u)$.
Then substitute these bounds into the bounds in lemma \ref{lem 1.5 hessian-ext}, the desired results are direct.
\end{proof}

\part{Model problem}
\section{Initialization of bootstrap argument}\label{sec bootstrap}
\subsection{Construction of initial data on $\Fcal_2$}
For the convenience of discussion the smallness conditions on the initial data will be made on the initial slice $\Fcal_2$. This seems non-standard, however, one can make the following observation. Suppose that the initial data are given on the slice $\{t=2\}$. By local theory, when initial data satisfies some smallness condition, the associated local solution extends to the region $\{2\leq t\leq 4\}\supset\Fcal_2$. We take the restriction of the local solution on $\Fcal_2$ as our initial data. Then, by energy estimate applied on the region $\{1\leq t\leq \Fcal_2\}$ (with multiplier $\del_t u$), we will obtain that

\begin{equation}\label{eq initial-smallness}
\Ecal_{\gamma}^{N+2}(2,u)^{1/2} + \Ecal_{\gamma,c}^{N+1}(2,v)^{1/2}\leq C_0\vep.
\end{equation}
where $C_0$ is a constant determined by $N$ and the system itself. Here recall that $\gamma>1$.

\subsection{Bootstrap argument}
We state the standard bootstrap argument. First, we write the following bounds on a time interval $[s_0,s_1]$:
\\
In exterior region:
\begin{subequations}\label{eq 1 bootstrap}
\begin{equation}\label{eq 1h bootstrap}
\EE_{N+1,\gamma}(s,Lu)^{1/2} + \EE_{N+1,\gamma,c}(s,v)^{1/2} \leq C_1\vep s^{1+\delta},
\end{equation}
\begin{equation}\label{eq 1m bootstrap}
\aligned
 \EE_{N,\gamma,c}(s,v)^{1/2} &\leq C_1\vep s^{\delta},
\endaligned
\end{equation}
\begin{equation}\label{eq 1l bootstrap}
\EE_{N+1,\gamma}(s,u)^{1/2} + \EE_{N+1,\gamma}(s,\del_{\alpha} u)^{1/2} + \EE_{N-1,\gamma,c}(s,v)^{1/2}\leq C_1\vep.
\end{equation}
\end{subequations}
and in interior region:
\begin{subequations}\label{eq 1 bootstrap-int}
\begin{equation}\label{eq 1h bootstrap-int}
\EH_N(s,Lu)^{1/2} \leq C_1\vep s^{1+\delta},
\end{equation}
\begin{equation}\label{eq 1m bootstrap-int}
\aligned
\EH_N(s,u)^{1/2} + \EH_N(s,\del_{\alpha} u)^{1/2} + \EH_{N,c}(s,v)^{1/2} &\leq C_1\vep s^{\delta},
\endaligned
\end{equation}
\begin{equation}\label{eq 1l bootstrap-int}
 \EH_{N-1,c}(s,v)^{1/2}\leq C_1\vep.
\end{equation}
\end{subequations}
Here we chose $0<\delta<\frac{1}{200}$, and $\delta\ll\gamma-1\ll1 $.

Then, let
$$
s^* = \sup\{2\leq s_1 < T^*|\text{\eqref{eq 1 bootstrap} and \eqref{eq 1 bootstrap-int} hold on }[2,s_1]\}
$$
where $T^*>2$ is the life-span time of the local solution.

Remark that when $C_1>C_0$, by continuity, $s^*>2$. Then if we can show that on any time interval $[2,s_1]$, with well chosen $(C_1,\vep)$, \eqref{eq 1 bootstrap} and \eqref{eq 1 bootstrap-int} lead to the following {\bf refined energy bounds}:
\\
In exterior region:
\begin{subequations}\label{eq 2 bootstrap}
\begin{equation}\label{eq 2h bootstrap}
\EE_{N+1,\gamma}(s,Lu)^{1/2} + \EE_{N+1,\gamma,c}(s,v)^{1/2} \leq \frac{1}{2}C_1\vep s^{1+\delta},
\end{equation}
\begin{equation}\label{eq 2m bootstrap}
\aligned
\EE_{N,\gamma,c}(s,v)^{1/2} &\leq \frac{1}{2}C_1\vep s^{\delta},
\endaligned
\end{equation}
\begin{equation}\label{eq 2l bootstrap}
\EE_{N+1,\gamma}(s,u)^{1/2} + \EE_{N+1,\gamma}(s,\del_{\alpha} u)^{1/2} +  \EE_{N-1,\gamma,c}(s,v)^{1/2}\leq \frac{1}{2}C_1\vep.
\end{equation}
\end{subequations}
and in interior region:
\begin{subequations}\label{eq 2 bootstrap-int}
\begin{equation}\label{eq 1h bootstrap-int}
\EH_N(s,Lu)^{1/2} \leq \frac{1}{2}C_1\vep s^{1+\delta},
\end{equation}
\begin{equation}\label{eq 2m bootstrap-int}
\aligned
\EH_N(s,u)^{1/2} + \EH_N(s,\del_{\alpha} u)^{1/2} + \EH_{N,c}(s,v)^{1/2} &\leq \frac{1}{2}C_1\vep s^{\delta},
\endaligned
\end{equation}
\begin{equation}\label{eq 2l bootstrap-int}
 \EH_{N-1,c}(s,v)^{1/2}\leq \frac{1}{2}C_1\vep.
\end{equation}
\end{subequations}
Then by continuity, $s^*=T^*$. That is, the following equalities holds for $s<T^*$:
$$
\lim_{s\rightarrow T^*}\big(\Ecal_{\gamma}^{N+1}(s,u)^{1/2}+\Ecal_{\gamma,c}^{N}(s,v)^{1/2}\big)\leq C_1\vep s^{1+ \delta}<+\infty
$$
and this (thanks to local theory) contradicts $T^*<+\infty$.

In the following, we concentrate on the refined energy bound \eqref{eq 2 bootstrap} \eqref{eq 2 bootstrap-int} based on \eqref{eq 1 bootstrap} \eqref{eq 1 bootstrap-int}. We remark especially that the first bound in \eqref{eq 1h bootstrap} leads to
\begin{equation}\label{eq 1' bootstrap}
\EE_{N+2,\gamma}(s,u)^{1/2} + \EH_{N+1}(s,u)^{1/2} \leq CC_1\vep s^{1+\delta}.
\end{equation}

\section{Analysis in  exterior region}\label{sec analysis-ext}
Recall that $\Pcal_s\cup\Tcal_s = \Hb_s = \Hb_s^+\cup \Hb_s^-$. In this section we establish the estimates valid in the transition and exterior region.
All calculations are made \underline{\sl in the transition and exterior region} unless otherwise specified.
\subsection{Basic $L^2$ bounds}

For $|I|+j\leq N+1$, by \eqref{eq 1l bootstrap} ,
\begin{equation}\label{eq 2 L2-w-basic-ext}
\aligned
\|(1+|r-t|)^{\gamma}\zeta \del^IL^j\del_{\alpha} u\|_{L^2(\Hb_s)} + \|(1+|r-t|)^{\gamma}\del^IL^j \delt_1 u\|_{L^2(\Hb_s)}\leq& CC_1\vep ,
\\
\|(1+|r-t|)^{\gamma}\zeta\del^IL^j\del_{\alpha}\del_{\beta} u\|_{L^2(\Hb_s)} + \|(1+|r-t|)^{\gamma}\del^IL^j\delt_1\del_\alpha u\|_{L^2(\Hb_s)}
\leq& CC_1\vep ,
\\
\|(1+|r-t|)^{\gamma}\del^IL^j\delt_1\delt_1 u\|_{L^2(\Hb_s)}
\leq& CC_1\vep .
\endaligned
\end{equation}


For the Klein-Gordon component, when $|I|+j = N+1$, thanks to \eqref{eq 1h bootstrap}
\begin{equation}\label{eq 1 L2-KG-basic-ext}
\|(1+|r-t|)^{\gamma}\zeta \del^IL^j\del_\alpha v\|_{L^2(\Hb_s)} + \|(1+|r-t|)^{\gamma}\del^IL^j \delt_1 v\|_{L^2(\Hb_s)}\leq CC_1\vep s^{1+\delta},
\end{equation}
\begin{equation}\label{eq 1' L2-KG-basic-ext}
\|c(1+|r-t|)^{\gamma} \del^IL^j v\|_{L^2(\Hb_s)}\leq CC_1\vep s^{1+\delta},
\end{equation}
\\
$-$ $|I|+j\leq N$, by \eqref{eq 1m bootstrap}
\begin{equation}\label{eq 2 L2-KG-basic-ext}
\|(1+|r-t|)^{\gamma}\zeta \del^IL^j\del_\alpha v\|_{L^2(\Hb_s)} + \|(1+|r-t|)^{\gamma}\del^IL^j \delt_1 v\|_{L^2(\Hb_s)}\leq CC_1\vep s^{\delta},
\end{equation}
\begin{equation}\label{eq 2' L2-KG-basic-ext}
\|c(1+|r-t|)^{\gamma} \del^IL^j v\|_{L^2(\Hb_s)}\leq CC_1\vep s^{\delta},
\end{equation}
\\
$-$ $|I|+|j|\leq N-1$, by \eqref{eq 1l bootstrap}
\begin{equation}\label{eq 3 L2-KG-basic-ext}
\|(1+|r-t|)^{\gamma}\zeta \del^IL^j\del_\alpha v\|_{L^2(\Hb_s)} + \|(1+|r-t|)^{\gamma}\del^IL^j \delt_1 v\|_{L^2(\Hb_s)}\leq CC_1\vep ,
\end{equation}
\begin{equation}\label{eq 3' L2-KG-basic-ext}
\|c(1+|r-t|)^{\gamma} \del^IL^j v\|_{L^2(\Hb_s)}\leq CC_1\vep.
\end{equation}

\subsection{Decay bounds}
Here we list out the decay bounds  which can be deduced directly from the bootstrap bounds \eqref{eq 1 bootstrap}.

From \eqref{eq 1' bootstrap} combined with lemma \ref{lem 1 decay-near-w} (remark that $\gamma>1$), for $|I|+j\leq N$,
\begin{equation}\label{eq 1h decay-w-basic-ext}
|r(1+|r-t|)^{\gamma-1}\del^IL^j\delt_1 u(t,x)| + |(1+|r-t|)^{\gamma}\del^IL^j\del_\alpha u(t,x)|\leq CC_1\vep s^{1+\delta}.
\end{equation}
For $|I|+j\leq N-1$, we apply  \eqref{eq 1l bootstrap} combined with lemma \ref{lem 1 decay-near-w}
\begin{equation}\label{eq 1l decay-w-basic-ext}
|r(1+|r-t|)^{\gamma-1}\del^IL^j\delt_1 u(t,x)| + |(1+|r-t|)^{\gamma}\del^IL^j\del_\alpha u(t,x)|\leq CC_1\vep.
\end{equation}

Similarly, for $|I|+j\leq N$, by \eqref{eq 1 sobolev-trans} and \eqref{eq 1 sobolev-ext} combined with \eqref{eq 1m bootstrap},
\begin{equation}\label{eq 1h decay-KG-basic-ext}
|c(1+|r-t|)^{\gamma}\del^IL^jv(t,x)|\leq CC_1\vep s^{\delta}
\end{equation}
and for $|I|+j\leq N-1$, by \eqref{eq 1 sobolev-trans} and \eqref{eq 1 sobolev-ext} combined with \eqref{eq 1l bootstrap},
\begin{equation}\label{eq 1 decay-KG-basic-ext}
|c(1+|r-t|)^{\gamma}\del^IL^jv(t,x)|\leq CC_1\vep.
\end{equation}

Based on the above bounds, we establish the following bound:
\begin{lemma}\label{lem 1 decay-KG-source}
Let $(u,v)$ be a regular solution of \eqref{eq main} defined in $\Fcal_{[s_0,s_1]}$ satisfying \eqref{eq 1 bootstrap}. Let $N^{\alpha\beta}$ be a null quadratic form (with $N^{\alpha\beta}$ constants in canonical frame).
For $|I|+j\leq N-1$,
\begin{equation}\label{eq 1h lem 1 decacy-KG-source-ext}
\big|\del^IL^j\big(N^{\alpha\beta}\del_{\alpha}u\del_{\beta}u\big)\big|\leq C(C_1\vep)^2 (1+|r-t|)^{1-2\gamma}(1+r)^{-1}.
\end{equation}

Furthermore, for $|I|+j\leq N-4$,
\begin{equation}\label{eq 1 decay-KG-applied-ext}
|c \del^IL^jv(t,x)|\leq CC_1\vep (1+r)^{-1}(1+|r-t|)^{1-\gamma},
\end{equation}
For $|I|+j\leq N-3$,
\begin{equation}\label{eq 1h decay-KG-applied-ext}
|c \del^IL^jv(t,x)| \leq CC_1\vep (1+r)^{-1}(1+|r-t|)^{1-\gamma}s^{\delta},
\end{equation}
and for $|I|+j\leq N-2$,
\begin{equation}\label{eq 1hh decay-KG-applied-ext}
|c \del^IL^jv(t,x)| \leq CC_1\vep (1+r)^{-1}(1+|r-t|)^{1-\gamma}s^{1+\delta}.
\end{equation}

\end{lemma}
\begin{proof}
For \eqref{eq 1h lem 1 decacy-KG-source-ext}, we rewrite this term in null frame (remark that $\Nt^{\alpha\beta}$ are locally constant):
$$
\del^IL^j\big(\Nt^{\alpha\beta}\del_{\alpha}u\del_{\beta}u\big)
= \sum_{I_1+I_2=I\atop j_1+j_2=j}\Nt^{\alpha\beta}\del^{I_1}L^{j_1}\delt_{\alpha}u\cdot \del^{I_2}L^{j_2}\delt_{\beta}u.
$$
The null condition leads to the fact that $\Nt^{00} = 0$. Then we apply \eqref{eq 1l decay-w-basic-ext}, the desired result is established.

Then we substitute, \eqref{eq 1h lem 1 decacy-KG-source-ext} and \eqref{eq 1l bootstrap} into \eqref{eq 3 prop 1 decay-KG}, \eqref{eq 1 decay-KG-applied-ext} ,\eqref{eq 1h decay-KG-applied-ext} and \eqref{eq 1hh decay-KG-applied-ext} are direct.
\end{proof}

\subsection{Energy and decay bounds on Hessian form}
The objective of this subsection is applying proposition \ref{prop 1 L2-hessian-ext} and lemma \ref{lem 1 decay-hessian-ext} together with \eqref{eq 1 bootstrap}. To this purpose we first remark that for $|I|+j = N+1$,
$$
\del^IL^jLv  = \del^I\big(r\delt_1L^jv\big) - \del^I\big((r-t)\del_xL^jv\big).
$$
Remark that on $\Hb_{[s_0,\infty)}$, due to the homogeneity
$$
|\del^I(r-t)|\leq
\left\{
\aligned
&(r-t),\quad |I|=0
\\
&r^{-|I|+1},\quad |I|\geq 1
\endaligned
\right.,
\quad
|\del^Ir|\leq Cr^{1-|I|}.
$$
Then for $|I|+j\leq N+1$ (remark that $(r-t)/r\leq \zeta$),
\begin{equation}\label{eq 0 pr lem 1 decay-KG-source}
\|cr^{-1}(1+|r-t|)^{\gamma}\del^IL^jLv\|_{L^2(\Hb_s)}\leq \EE_{N+1,\gamma,c}(s,v)^{1/2}.
\end{equation}

Furthermore, recall the $L^2$ bounds \eqref{eq 1 L2-KG-basic-ext}, for $|I|+j\leq N+2$
\begin{equation}\label{eq 1 pr lem 1 decay-KG-source}
\|cr^{-1}(1+|r-t|)^{\gamma}\del^IL^j v\|_{L^2(\Hb_s)}\leq \EE_{N+1,\gamma,c}(s,v)^{1/2}\leq C_1\vep s^{1+\delta}.
\end{equation}

Then we establish the following bounds:
\begin{lemma}\label{lem 1 hessian-ext-app}
Under the bootstrap assumption and suppose that $1/2\leq c\leq 2$:
\\
For $|I|+j\leq N+2$,
\begin{equation}\label{eq 1 lem 1 hessian-ext-app}
\|r^{1+\gamma} \del^IL^j (v^3)\|_{L^2(\Fcal^E_s)}   + \|t(1+|r-t|)^{\gamma}\del^IL^j(v^3)\|_{L^2(\Fcal^I_s)} \leq C (C_1\vep)^3 s^{1+\delta}.
\end{equation}
For $|I|+j\leq N+1$,
\begin{equation}\label{eq 1' lem 1 hessian-ext-app}
 \|r^{1+\gamma}  \del^IL^j (v^3)\|_{L^2(\Fcal^E_s)} + \|t(1+|r-t|)^{\gamma}\del^IL^j(v^3)\|_{L^2(\Fcal^I_s)} \leq C (C_1\vep)^3.
\end{equation}


For $|I|+j\leq N-1$,
\begin{equation}\label{eq 2hh lem 1 hessian-ext-app}
|\del^IL^j(v^3)(t,x)|\leq
\left\{
\aligned
C(C_1\vep)^3 t^{-2}(1+|r-t|)^{2-3\gamma},\quad (t,x)\in \Fcal^I_s,
\\
C(C_1\vep)^3 r^{-2-\gamma},\quad (t,x)\in \Fcal^E_s.
\endaligned
\right.
\end{equation}

For $|I|+j\leq N-3$,
\begin{equation}\label{eq 2h lem 1 hessian-ext-app}
|\del^IL^j(v^3)(t,x)|\leq
\left\{
\aligned
C(C_1\vep)^3 t^{-3}(1+|r-t|)^{3(1-\gamma)}s^{\delta},\quad (t,x)\in \Fcal^I_s,
\\
C(C_1\vep)^3 r^{-3\gamma}s^{\delta},\quad (t,x)\in \Fcal^E_s.
\endaligned
\right.
\end{equation}
For $|I|+j\leq N-4$,
\begin{equation}\label{eq 2 lem 1 hessian-ext-app}
|\del^IL^j(v^3)(t,x)|\leq
\left\{
\aligned
C(C_1\vep)^3 t^{-3}(1+|r-t|)^{3(1-\gamma)},\quad (t,x)\in \Fcal^I_s,
\\
C(C_1\vep)^3 r^{-3\gamma},\quad (t,x)\in \Fcal^E_s.
\endaligned
\right.
\end{equation}

\end{lemma}
\begin{proof}
It is clear that
$$
\del^IL^j (v^3) = \sum_{I_1+I_2+I_3=I\atop J_1+J_2+J_3=J}\del^{I_1}L^{j_1}v\cdot \del^{I_2}L^{j_2}v\cdot \del^{I_3}L^{j_3}v.
$$
That is, the operator $\del^IL^J$ is distributed on three factors.


For \eqref{eq 1 lem 1 hessian-ext-app}, observe that among these three factors there are at least two have derivatives of order $\leq [(N+2)/2]\leq N-4$ (for $N\geq 9$). Without loss of generality, suppose that $|I_2|+j_2\leq [(N+2)/2]$, $|I_3|+j_3\leq [(N+2)/2]$. Then we apply \eqref{eq 1 pr lem 1 decay-KG-source} on $\del^{I_1}L^{J_1}v$ and \eqref{eq 1 decay-KG-applied-ext} on the rest factors. On $\Fcal^E_s$, this leads to
$$
\aligned
&\|r^{1+\gamma}\del^{I_1}L^{j_1}v\cdot \del^{I_2}L^{j_2}v\cdot \del^{I_3}L^{j_3}v\|_{L^2(\Fcal^E_s)}
\\
\leq& C(C_1\vep)^2\|r^{-1}(1+|r-t|)^{\gamma}\del^{I_1}L^{j_1}v\|_{L^2(\Fcal^E_s)}
\cdot\|r(1+|r-t|)^{-\gamma}(1+r)^{1+\gamma}(1+r)^{-2\gamma}\|_{L^{\infty}(\Fcal^E_s)}
\\
\leq& C(C_1\vep)^3 s^{5-4\gamma+\delta}\leq C(C_1\vep)^3s^{1+\delta}.
\endaligned
$$
On $\Fcal^I_s$, this leads to
$$
\aligned
&\|t(1+|r-t|)^{\gamma}\del^{I_1}L^{j_1}v\cdot \del^{I_2}L^{j_2}v\cdot \del^{I_3}L^{j_3}v\|_{L^2(\Fcal^I_s)}
\\
\leq& C(C_1\vep)^2\|r^{-1}(1+|r-t|)^{\gamma}\del^{I_1}L^{j_1}v\|_{L^2(\Fcal^I_s)}
\cdot\|rt\cdot t^{-2}(1+|r-t|)^{2-2\gamma}\|_{L^{\infty}(\Fcal^I_s)}
\\
\leq& C(C_1\vep)^3s^{1+\delta}.
\endaligned
$$

For \eqref{eq 1' lem 1 hessian-ext-app}, we proceed exactly in the same manner except that we take \eqref{eq 1' L2-KG-basic-ext} instead of \eqref{eq 1 pr lem 1 decay-KG-source}.

For \eqref{eq 2hh lem 1 hessian-ext-app}, we apply \eqref{eq 1 decay-KG-basic-ext} on the factor with highest order and \eqref{eq 1 decay-KG-applied-ext} on the rest factors.



For \eqref{eq 2h lem 1 hessian-ext-app} and\eqref{eq 2 lem 1 hessian-ext-app}, we simply substitute \eqref{eq 1 decay-KG-applied-ext} and \eqref{eq 1h decay-KG-applied-ext} into the expression.
\end{proof}

Substitute the above bounds in lemma \ref{lem 1 hessian-ext-app} into the bounds in lemma \ref{lem 1 decay-hessian-ext}  (recall $\Box\del^IL^j u = \del^IL^j(v^3)$), we obtain the following decay bounds:
\\
By \eqref{eq 11' lem 1 decay-hessian-ext} and \eqref{eq 2hh lem 1 hessian-ext-app} combined with \eqref{eq 1' bootstrap},
\begin{equation}\label{eq 1h decay-hessian-ext-far}
|t^2(1+|r-t|)^{\gamma-1}\del^IL^j \delt_1\delt_1 u(t,x)|\leq CC_1\vep s^{1+\delta},\quad (t,x)\in\Fcal^I_s,\quad |I|+j\leq N-1.
\end{equation}

By \eqref{eq 1 lem 1 decay-hessian-ext} and \eqref{eq 2h lem 1 hessian-ext-app} combined with \eqref{eq 1l bootstrap},
\begin{equation}\label{eq 1 decay-hessian-ext-far}
|\del^IL^j\del_{\alpha}\del_{\beta}u(t,x)|\leq CC_1\vep t^{-1}r^{-\gamma},\quad (t,x) \in\Fcal^E_s,\quad |I|+j\leq N-2.
\end{equation}

When $(t,x)\in \Fcal^I_s$ and $|I|+j\leq N-2$, the following terms
\begin{equation}\label{eq 1 decay-hessian-ext-near}
|(1+|r-t|)^{1+\gamma}\del^IL^j\del_t\del_tu(t,x)|,\quad |t(1+|r-t|)^{\gamma}\del^IL^j\delt_1\del_tu(t,x)|,\quad  |t^2(1+|r-t|)^{\gamma-1}\del^IL^j\delt_1\delt_1u(t,x)|
\end{equation}
are bounded by $CC_1\vep$.

In the same manner, by Proposition \ref{prop 1 L2-hessian-ext}, we have the following bounds:
\\
For $|I|+j\leq N+1$, thanks to \eqref{eq 1' bootstrap} and \eqref{eq 1' lem 1 hessian-ext-app}, the following quantities are bounded by $CC_1\vep s^{1+\delta}$:
\begin{equation}\label{eq h L2-hessian}
\aligned
&\|tr^{\gamma}\del^IL^j\del_{\alpha}\del_{\beta}u\|_{L^2(\Fcal^E_s)},\quad \|(1+|r-t|)^{1+\gamma}\zeta \del^IL^j\del_t\del_t u\|_{L^2(\Fcal^I_s)},
\\
&\|t(1+|r-t|)^{\gamma}\zeta\del^IL^j \delt_1\del_t u\|_{L^2(\Fcal^I_s)},\quad \|t(1+|r-t|)^{\gamma}\del^IL^j\delt_1\delt_1 u\|_{L^2(\Fcal^I_s)}.
\endaligned
\end{equation}
For $|I|+j\leq N$, the following quantities are bounded by $CC_1\vep $:
\begin{equation}\label{eq m L2-hessian}
\aligned
&\|tr^{\gamma}\del^IL^j\del_{\alpha}\del_{\beta}u\|_{L^2(\Fcal^E_s)},\quad \|(1+|r-t|)^{1+\gamma}\zeta \del^IL^j\del_t\del_t u\|_{L^2(\Fcal^I_s)}
\\
&\|t(1+|r-t|)^{\gamma}\zeta\del^IL^j \delt_1\del_t u\|_{L^2(\Fcal^I_s)},\quad \|t(1+|r-t|)^{\gamma}\del^IL^j\delt_1\delt_1 u\|_{L^2(\Fcal^I_s)}.
\endaligned
\end{equation}

\subsection{Estimates on source terms}
Based on the bounds established in previous subsections, we will first establish the following bounds:
\begin{lemma}[Source bounds for wave equation in exterior region]\label{lem 1 source-w-ext}
Under the bootstrap assumption,  for $|I|+j\leq N+1$
\begin{equation}\label{eq 1 source-Lw-ext}
\|(\xi_s + (1-\xi_s)^{1/2}s)(1+\omega_{\gamma})\del^IL^{j+1} (v^3)\|_{L^2(\Hb_s)}\leq C(C_1\vep)^2s^{\delta}.
\end{equation}
And the following terms
\begin{equation}\label{eq 1 source-w-ext}
\|(\xi_s + (1-\xi_s)^{1/2}s)(1+\omega_{\gamma})\del^IL^j (v^3)\|_{L^2(\Hb_s)},\quad  \|(\xi_s+(1-\xi_s)^{1/2}s)(1+\omega_{\gamma})\del^IL^j\del_{\alpha}(v^3)\|_{L^2(\Hb_s)}
\end{equation}
are bounded by
$$
\aligned
C(C_1\vep)^2 s^{-2+\delta} \quad\text{ for }|I|+j\leq N+1.
\endaligned
$$
\end{lemma}

\begin{proof}
Thanks to \eqref{eq xi-zeta}
\begin{equation}\label{eq 1 weight-proof}
|\xi_s + (1-\xi_s)^{1/2}s|\leq C\zeta s
\end{equation}
and
\begin{equation}\label{eq 2 weight-proof}
|1+\omega_{\gamma}|\leq C(1+|r-t|)^{\gamma}.
\end{equation}

For \eqref{eq 1 source-Lw-ext}, we remark that
$$
\del^IL^{j+1}\big(v^3\big) = 3\sum_{I_1+I_2+I_3=I\atop j_1+j_2+j_3=j}\del^{I_1}L^{j_1}Lv\cdot \del^{I_2}L^{j_2}v\cdot \del^{I_3}L^{j_3}v.
$$
Then we distinguish between different cases:
\\
$-$ when $|I_1|+j_1=N+1$, we apply \eqref{eq 0 pr lem 1 decay-KG-source} on the first factor and \eqref{eq 1 decay-KG-applied-ext} on the rest factors:
$$
\aligned
&\|\zeta s (1+|r-t|)^{\gamma} \del^IL^{j+1}v\cdot v^2\|_{L^2(\Hb_s)}
\\
\leq& C(C_1\vep)^2\|r^{-1}(1+|r-t|)^{\gamma}\zeta \del^IL^{j+1}v \|_{L^2(\Hb_s)}\cdot\|sr (1+r)^{-2}(1+|r-t|)^{2-2\gamma}\|_{L^{\infty}(\Hb_s)}
\\
\leq& C(C_1\vep)^3 s^{\delta}.
\endaligned
$$
\\
$-$ when $4\leq |I_1|+j_1\leq N-1$, we apply \eqref{eq 1' L2-KG-basic-ext} on the first factor and \eqref{eq 1 decay-KG-applied-ext} on the rest factors (recall that $|I_2|+j_2+|I_3|+j_3\leq N-4$). The $L^2$ norm is bounded by $C(C_1\vep)^3s^{-2+\delta}$, we omit the detail of calculation.
\\
$-$ when $|I_1|+j_1\leq 3\leq N-5$, either $|I_2|+j_2\leq N-4$ or $|I_3|+j_3\leq N-4$. Without loss of generality we suppose $|I_2|+j_2\leq N-4$. Then we apply \eqref{eq 1 decay-KG-applied-ext} on the first and the last factor and \eqref{eq 1' L2-KG-basic-ext} on the second factor. The $L^2$ norm is bounded by $C(C_1\vep)^3s^{-2+\delta}$. So \eqref{eq 1 source-Lw-ext} is concluded.

For \eqref{eq 1 source-w-ext}, the estimate on $\del^IL^j(v^3)$ is simpler and similar to that of $\del^IL^j\del_{\alpha}(v^3)$, we omit the detail and concentrate on latter one. Remark the following calculation:
\begin{equation}\label{eq 1 proof lem 1 hessian-ext-app}
\del^IL^j\del_\alpha (v^3) = 3\sum_{I_1+I_2+I_3=I\atop j_1+j_2+j_3=j}\del^{I_1}L^{j_1}\del_{\alpha} v\cdot \del^{I_2}L^{j_2}v\cdot \del^{I_3}L^{j_3}v,
\end{equation}
\\
\noindent $-$ When $4\leq |I_1|+j_1\leq N+1$, $|I_2|+j_2\leq N-4$, $|I_3|+j_3\leq N-4$. Then we apply \eqref{eq 1 decay-KG-applied-ext} on $\del^{I_2}L^{j_2} v$ and $\del^{I_3}L^{j_3}v$ and \eqref{eq 1 L2-KG-basic-ext} on $\del^{I_1}L^{J_1}\del_{\alpha}v$. Then
$$
\aligned
&\|(1+|r-t|)^{\gamma}s\zeta\del^{I_1}L^{j_1}\del_{\alpha}v\cdot  \del^{I_2}L^{j_2}v\cdot\del^{I_3}L^{j_3} v\|_{L^2(\Fcal_s^E)}
\\
\leq& \|(1+|r-t|)^{\gamma}\zeta\del^{I_1}L^{j_1}\del_{\alpha}v\|_{L^2(\Fcal_s^E)}\cdot
 \|s\del^{I_2}L^{j_2}v\cdot\del^{I_3}L^{j_3} v\|_{L^\infty(\Fcal_s^E)}
\leq  C(C_1\vep)^3 s^{-2+\delta}.
\endaligned
$$

\noindent $-$ When $|I_1|+j_1\leq 3\leq N-5$, $|I_2|+j_2+|I_3|+j_3\leq N$ and this leads to $|I_2|+j_2\leq [(N+1)/2]\leq N-4$ or $|I_3|+j_3\leq [(N+1)/2]\leq N-4$. Suppose without loss of generality that $|I_2|+j_2\leq N-4$. Then:
$$
\aligned
&\|(1+|r-t|)^{\gamma}\zeta s\del^{I_1}L^{j_1}\del_{\alpha}v\cdot \del^{I_2}L^{j_2}v\cdot \del^{I_3}L^{j_3}v\|_{L^2(\Hb_s)}
\\
\leq& C(C_1\vep)^2 \|s\del^{I_1}L^{j_1}\del_{\alpha}v\cdot \del^{I_2}L^{j_2}v\|_{L^{\infty}(\Hb_s)}\cdot \|(1+|r-t|)^{\gamma}\zeta \del^{I_3}L^{j_3} v\|_{L^2(\Hb_s)}
\\
\leq& C(C_1\vep)^3 s^{-2+\delta} .
\endaligned
$$
%
\end{proof}

For Klein-Gordon component,  we establish the following bounds:
\begin{lemma}[Source bounds for Klein-Gordon equation in exterior region]\label{lem 1 source-KG-ext}
Under the bootstrap assumption \eqref{eq 1 bootstrap}, the following bounds hold for $|I|+j\leq N+1$:
\begin{equation}\label{eq 1 lem 1 source-KG-ext}
\|(1+(1-\xi_s)^{1/2}s)(1+\omega_\gamma)\del^IL^j\big(m^{\mu\nu}N^{\alpha\beta}\del_{\mu}\del_{\alpha}u\cdot \del_{\nu}\del_{\beta}u\big)\|_{L^2(\Hb_s)}
\leq C(C_1\vep)^2 s^{\delta},
\end{equation}
\\
$-$ for $|I|+j\leq N$
\begin{equation}\label{eq 2 lem 1 source-KG-ext}
\|(1+(1-\xi_s)^{1/2}s)(1+\omega_\gamma)\del^IL^j\big(m^{\mu\nu}N^{\alpha\beta}\del_{\mu}\del_{\alpha}u\cdot \del_{\nu}\del_{\beta}u\big)\|_{L^2(\Hb_s)}
\leq C(C_1\vep)^2s^{-1+\delta},
\end{equation}
\\
$-$ for $|I|+j\leq N-1$
\begin{equation}\label{eq 3 lem 1 source-KG-ext}
\|(1+(1-\xi_s)^{1/2}s)(1+\omega_\gamma)\del^IL^j\big(m^{\mu\nu}N^{\alpha\beta}\del_{\mu}\del_{\alpha}u\cdot \del_{\nu}\del_{\beta}u\big)\|_{L^2(\Hb_s)}
\leq C(C_1\vep)^2 s^{-2+\delta}.
\end{equation}

Furthermore, for $|I|+j\leq N+1$
\begin{equation}\label{eq 1 lem 2 source-KG-ext}
\big\|(\xi_s+(1-\xi_s)^{1/2}s )(1+\omega_{\gamma}) \del^IL^j\big(\del_{\alpha}u\cdot \del_{\beta}(v^3)\big)\big\|_{L^2(\Hb_s)}\leq C(C_1\vep)^4 s^{-2+\delta}.
\end{equation}
\end{lemma}
\begin{proof}
We recall \eqref{eq 1 weight-proof} and \eqref{eq 2 weight-proof} and observe that the above bounds \eqref{eq 1 lem 1 source-KG-ext},\eqref{eq 2 lem 1 source-KG-ext} and \eqref{eq 3 lem 1 source-KG-ext} are trivial for the region $\Fcal^E_s$. In fact we can prove that for $|I|+j\leq N+1$,
\begin{equation}\label{eq 1 source-KG-far}
\|(1+|r-t|)^{\gamma}s\zeta\del^IL^j\big(\del_{\alpha}\del_{\beta}u\cdot \del_{\mu}\del_{\nu} u\big)\|_{L^2(\Fcal^E_s)}\leq C(C_1\vep)^2s^{-2+\delta}.
\end{equation}
This is by applying directly \eqref{eq h L2-hessian} (the first term) and \eqref{eq 1 decay-hessian-ext-far}.

In the region $\Fcal^I_s$, we need to evoke the double-null structure. That is, $\mt^{00}$ and $\Nt^{00}$ are zero. To do so, we write this term within the null frame $\{\delt_{\alpha}\}$ and see that it is a linear combination of the following terms with constant coefficients (remark that the elements of transition matrix $\Psit^{\beta}_{\alpha}$ are locally constant, so there is no terms with derivatives on these elements):
\begin{equation}\label{eq 1 proof lem 1 source-KG-ext}
\aligned
&\mt^{10}\Nt^{10}\delt_1\delt_1u\del_t\del_tu,\quad \mt^{10}\Nt^{01}\del_t\delt_1u\delt_1\del_tu,\quad \mt^{10}\Nt^{11}\delt_1\delt_1u\del_t\delt_1u
\\
&\mt^{11}\Nt^{10}\delt_1\delt_1u\delt_1\del_tu,\quad \mt^{11}\Nt^{01}\delt_1\del_tu\delt_1\delt_1u,\quad \mt^{11}\Nt^{11}\delt_1\delt_1u\delt_1\delt_1u
\endaligned
\end{equation}

Then the rest work is to verify that each term in the above list are correctly bounded as in \eqref{eq 1 lem 1 source-KG-ext}, \eqref{eq 2 lem 1 source-KG-ext} and \eqref{eq 3 lem 1 source-KG-ext}. We remark that the except the first term, other terms are trivial. In fact by \eqref{eq 1 decay-hessian-ext-near} and \eqref{eq h L2-hessian}, for $|I|+j\leq N+1$, the $L^2$ norms of these terms are bounded by $C(C_1\vep)^2s^{-2+\delta}$.

The only problematic term is $\delt_1\delt_1u\del_t\del_tu$. We make the following discussion. First,
$$
\del^IL^j\big(\del_t\del_tu\cdot\delt_1\delt_1u\big) = \sum_{I_1+I_2=I\atop j_1+j_2=J}\del^{I_1}L^{j_1}\del_t\del_tu\cdot \del^{I_2}L^{j_2}\delt_1\delt_1u.
$$
Now suppose that $|I|+j\leq N+1$.
\\
{\bf 1.} when $2\leq |I_1|+j_1 \leq N+1$, we apply the second bound in \eqref{eq 2 L2-w-basic-ext} on $\del^{I_1}L^{J_1}\del_t\del_tu$ and \eqref{eq 1 decay-hessian-ext-near} on $\del^{I_2}L^{j_2}\delt_1\delt_1u$ where $|I_2|+j_2\leq N-2$. In this case the $L^2$ norm is bounded by $C(C_1\vep)^2s^{-2+\delta}$.
\\
{\bf 2.} when $|I_1|+j_1=1\leq N-2$, in this case $|I_1|+j_1\leq N$.  We apply \eqref{eq 1 decay-hessian-ext-near} on $\del^{I_1}L^{J_1}\del_t\del_tu$ and \eqref{eq m L2-hessian} on $\del^{I_2}L^{j_2}\delt_1\delt_1u$. Remark that $s^2\sim t\leq r$. In this case the $L^2$ norm is bounded by $C(C_1\vep)^2s^{-1+\delta}$.
\\
{\bf 3.} when $|I_2|+j_2=N+1$. This is the most critical case. We apply \eqref{eq 1 decay-hessian-ext-near} on $\del_t\del_tu$ and \eqref{eq h L2-hessian} on $\del^IL^j\delt_1\delt_1u$. In this case the $L^2$ norm is bonded by $C(C_1\vep)^2s^{\delta}$.

For \eqref{eq 1 lem 1 source-KG-ext}, we apply the above discussion and the desired is direct.

For \eqref{eq 2 lem 1 source-KG-ext}, we remark that the above case $3.$ does not exist. So the desired result is established.

For \eqref{eq 3 lem 1 source-KG-ext}, we remark that the above case $2.$ and $3.$ do not exist.

%
%
%

 For \eqref{eq 1 lem 2 source-KG-ext}, we remark that this term is a linear combination of the following terms:
$$
\del^{I_1}L^{j_1}\del_{\alpha}u\cdot \del^{I_2}L^{j_2}\del_{\beta}v\cdot \del^{I_3}L^{j_3}v\cdot \del^{I_4}L^{j_4}v,
\quad \sum_{k=1}^4I_k=I,\quad \sum_{k=1}^4j_k = j.
$$
Now we proceed as above with the application of \eqref{eq 2 L2-w-basic-ext} (the first bound), \eqref{eq 1l decay-w-basic-ext}, \eqref{eq 1 L2-KG-basic-ext}, \eqref{eq 1 decay-KG-applied-ext}.

\end{proof}

\begin{lemma}[Source bounds for Klein-Gordon equation in exterior region II]\label{lem 2 source-KG-ext}
Under the bootstrap assumption \eqref{eq 1 bootstrap}, the following bounds hold:
\begin{equation}\label{eq 2 lem 2 source-KG-ext}
\big\|(1+\omega_\gamma)\del^IL^j\del_{\gamma}\big(N^{\alpha\beta}\del_{\alpha}u\del_{\beta}u\big)\big\|_{L^2(\Hb_s)}\leq
\left\{
\aligned
&C(C_1\vep)^2 s^{1+\delta},\quad &&|I|+j = N+1,
\\
&C(C_1\vep)^2 ,\quad &&|I|+j \leq N.
\endaligned
\right.
\end{equation}
\begin{equation}\label{eq 2' lem 2 source-KG-ext}
\big\|(1+\omega_\gamma)\del^IL^j\big(N^{\alpha\beta}\del_{\alpha}u\del_{\beta}u\big)\big\|_{L^2(\Hb_s)}\leq
\left\{
\aligned
&C(C_1\vep)^2 s^{1+\delta},\quad &&|I|+j = N+1,
\\
&C(C_1\vep)^2,\quad &&|I|+j \leq N.
\endaligned
\right.
\end{equation}
\end{lemma}
\begin{proof}
For \eqref{eq 2 lem 2 source-KG-ext}, we proceed as in the proof of lemma \ref{lem 1 source-KG-ext} and distinguish between  $\Fcal^I_s$ and $\Fcal^E_s$. For $\Fcal^E_s$, we apply \eqref{eq h L2-hessian} and \eqref{eq 1 decay-hessian-ext-far} on the term $\del_{\alpha}\del_{\beta}u$ and \eqref{eq 2 L2-w-basic-ext} and \eqref{eq 1l decay-w-basic-ext} on the factor $\del_{\alpha}u$.

For $\Fcal^I_s$, we need to evoke the null structure. That is (recall that $\Psit_{\alpha}^{\beta}$ is locally constant),
$$
\aligned
&\del^IL^j\del_{\gamma}\big(N^{\alpha\beta}\del_{\alpha}\del_{\alpha}u\del_{\beta}u\big)
\\
=& \del^IL^j\del_{\gamma}\big(\Nt^{\alpha\beta}\delt_{\alpha}u\cdot\delt_{\beta}u\big)
= 2\sum_{I_1+I_2 = I\atop j_1+j_2=j}\Nt^{\alpha\beta}\del^{I_1}L^{j_1}\del_{\gamma}\delt_{\alpha}u\cdot \del^{I_2}L^{j_2}\delt_{\beta}u
\\
=&2\sum_{I_1+I_2 = I\atop j_1+j_2=j}\Nt^{10}\del^{I_1}L^{j_1}\del_{\gamma}\delt_1u\cdot \del^{I_2}L^{j_2}\del_tu
 +2\sum_{I_1+I_2 = I\atop j_1+j_2=j}\Nt^{11}\del^{I_1}L^{j_1}\del_{\gamma}\delt_1u\cdot \del^{I_2}L^{j_2}\delt_1u
\\
&+2\sum_{I_1+I_2 = I\atop j_1+j_2=j}\Nt^{01}\del^{I_1}L^{j_1}\del_{\gamma}\del_tu\cdot\del^{I_2}L^{j_2}\delt_1u.
\endaligned
$$
Then for $|I|+j=N+1$ we apply \eqref{eq 1 decay-hessian-ext-near} and \eqref{eq h L2-hessian} on $\del_{\gamma}\del_tu$ and $\del_{\gamma}\delt_1u$, and \eqref{eq 2 L2-w-basic-ext} and \eqref{eq 1l decay-w-basic-ext} on $\del_t u$. For $|I|+j\leq N$, we replace \eqref{eq h L2-hessian} by \eqref{eq m L2-hessian}.


The bounds \eqref{eq 2' lem 2 source-KG-ext} is proved similarly, we evoke the null structure and apply \eqref{eq 2 L2-w-basic-ext} and \eqref{eq 1l decay-w-basic-ext}.
\end{proof}

\subsection{Refined energy bounds in exterior region}\label{subsec energy-refine-ext}
We will substitute the above $L^2$ bounds into the exterior energy estimate proposition \ref{prop 2 energy}. We will establish:
\begin{subequations}\label{eq 1 refined-ext}
\begin{equation}\label{eq 1h refined-ext}
\EE_{N+1,\gamma}(s,Lu)^{1/2} + \EE_{N+1,\gamma,c}(s,v)^{1/2} \leq C_0\vep  + C(C_1\vep)^2s^{1+\delta},
\end{equation}
\begin{equation}\label{eq 1m refined-ext}
\aligned
 \EE_{N,\gamma,c}(s,v)^{1/2} \leq C_0\vep  + C(C_1\vep)^2s^{\delta},
\endaligned
\end{equation}
\begin{equation}\label{eq 1l refined-ext}
\EE_{N+1,\gamma}(s,u)^{1/2} + \EE_{N+1,\gamma}(s,\del_{\alpha} u)^{1/2} + \EE_{N-1,\gamma,c}(s,v)^{1/2}\leq C_0\vep  + C(C_1\vep)^2.
\end{equation}
\end{subequations}

To establish the above refined bounds, we make the following transformation:
\begin{equation}\label{eq main-trans}
\left\{
\aligned
&\Box \del^IL^j u = \del^IL^j (v^3),
\\
&\Box \del^IL^j w + c^2\del^IL^j w = -2N^{\alpha\beta}m^{\mu\nu}\del^IL^j\big(\del_{\mu}\del_{\alpha}u\cdot \del_{\nu}\del_{\beta}u\big)
- 2N^{\alpha\beta}\del_{\alpha}(v^3) \del_{\beta}u.
\endaligned
\right.
\end{equation}
where $w := v - N^{\alpha\beta}\del_{\alpha}u\del_{\beta}u$. To see this we only need to remark that
$$
\Box \big(v - N^{\alpha\beta}\del_{\alpha}u\del_{\beta}u\big) + c^2\big(v - N^{\alpha\beta}\del_{\alpha}u\del_{\beta}u\big) = -2N^{\alpha\beta}\Box \del_{\alpha}u\cdot \del_{\beta}u - 2N^{\alpha\beta}m^{\mu\nu}\big(\del_{\mu}\del_{\alpha}u\cdot \del_{\nu}\del_{\beta}u\big)
$$
and then we take the fist equation of \eqref{eq main}. Then we apply proposition \ref{prop 2 energy} combined with lemma \ref{lem 1 source-w-ext} on the first equation of \eqref{eq main-trans} and lemma \ref{lem 1 source-KG-ext} on the second equation of \eqref{eq main-trans}. We obtain the bounds on $u$ and $w$.

To obtain the energy bounds of $v$, recall that by lemma \ref{lem 2 source-KG-ext}, we see that
$$
\big\|\zeta (1+\omega_{\gamma})\del_{\gamma}\del^IL^j\big(N^{\alpha\beta}\del_{\alpha}u\del_{\beta}u\big) \big\|_{L^2(\Hb_s)}, \quad
\big\|(1+\omega_{\gamma})\del^IL^j\big(N^{\alpha\beta}\del_{\alpha}u\del_{\beta}u\big)\big\|_{L^2(\Hb_s)}
$$
$$
\big\|(1+\omega_{\gamma})\big(\xi_s(r)x(s^2+x^2)^{-1/2}\del_t + \del_x\big)\del^IL^j\big(N^{\alpha\beta}\del_{\alpha}u\del_{\beta}u\big) \big\|_{L^2(\Hb_s)}
$$
are bounded by $C(C_1\vep)^2 s^{1+\delta}$ for $(N+1)-$order and by $C(C_1\vep)^2$ for $N-$order (to see the second, we remark that the coefficients are bounded by $1$ and apply commutator estimate). This leads to
\begin{equation}\label{eq energy w-correction-ext}
E^{\text{E}}_{\gamma,c}\big(s,\del^IL^j\big(N^{\alpha\beta}\del_{\alpha}u\del_{\beta}u\big)\big)^{1/2}\leq
\left\{
\aligned
&C(C_1\vep)^2 s^{1+\delta},\quad |I|+j =N+1,
\\
&C(C_1\vep)^2, \quad |I|+j\leq N.
\endaligned
\right.
\end{equation}
Then by triangle inequality (the energy itself is a norm), the bounds on $v$ is obtained.

The energy estimate proposition \ref{prop 2 energy} also leads to the following bounds which will applied in the refined energy estimate in interior region:
\begin{equation}\label{eq energy-cone}
\aligned
&E^K(s,\del^IL^ju;s_0)^{1/2} \leq C(C_1\vep)^2,\quad |I|+j\leq N+1
\\
&E_c^K(s,\del^IL^jv;s_0)^{1/2}\leq
\left\{
\aligned
&C(C_1\vep)^2s^{\delta},\quad && |I|+j\leq N,
\\
&C(C_1\vep)^2,\quad && |I|+j\leq N-1.
\endaligned
\right.
\endaligned
\end{equation}

\section{Estimates in interior region}\label{sec analysis-int}
\subsection{$L^2$ and decay bounds in interior region}
All calculations in this subsection are made in \underline{\sl$\Hcal^*_{[2,s_1]}$}, unless otherwise specified. Furthermore, all calculations in this section {\bf do not} depend on the choice of $\gamma$.

The estimate in interior region has been treated in our previous work (see for example \cite{LM1}). The $L^2$ and decay bounds are based on the relations listed in subsection \ref{subsec basic-interior} combined with the energy bounds \eqref{eq 1 bootstrap-int}.

For $|I|+j\leq N$, by \eqref{eq 1m bootstrap-int} and lemma \ref{lem 1 est-high-int}:
\begin{equation}\label{eq 1 L2-w-basic-int}
\|(s/t)\del^IL^j\del_{\alpha}u\|_{L^2(\Hcal^*_s)} + \|\del^IL^j\delu_1u\|_{L^2(\Hcal^*_s)}\leq CC_1\vep s^{\delta}.
\end{equation}
\begin{equation}\label{eq 2.5 L2-w-basic-int}
\|(s/t)\del^IL^j\del_{\alpha}\del_{\beta}u\|_{L^2(\Hcal_s^*)} + \|(s/t)\del_{\alpha}\del_{\beta}\del^IL^j u\|_{L^2(\Hcal_s^*)}\leq CC_1\vep s^{\delta}.
\end{equation}
\begin{equation}\label{eq 2 L2-w-basic-int}
\|\del^IL^j\del_{\alpha}\delu_1 u\|_{L^2(\Hcal^*_s)} + \|\del^IL^j\delu_1\del_{\alpha} u\|_{L^2(\Hcal^*_s)}\leq CC_1\vep s^{\delta}.
\end{equation}
By \eqref{eq 1 lem 1 est-high-int} and \eqref{eq 1' bootstrap}, for $|I|+j = N$,
\begin{equation}\label{eq 3 L2-w-basic-int}
\|(t/s)\del^IL^j\delu_1\delu_1 u\|_{L^2(\Hcal^*_s)}  \leq  CC_1\vep s^{\delta}.
\end{equation}

For $|I|+j\leq N-1$, by \eqref{eq 1m bootstrap-int} and lemma \ref{lem 1 est-high-int}
\begin{equation}\label{eq 2' L2-w-basic-int}
\|s\del^IL^j\del_{\alpha}\delu_1 u\|_{L^2(\Hcal^*_s)} + \|s\del^IL^j\delu_1\del_{\alpha} u\|_{L^2(\Hcal^*_s)}\leq CC_1\vep s^{\delta},
\end{equation}
\begin{equation}\label{eq 3' L2-w-basic-int}
\|t\del^IL^j\delu_1\delu_1 u\|_{L^2(\Hcal^*_s)}  \leq  CC_1\vep s^{\delta}.
\end{equation}

In the same manner, by lemma \ref{lem 2' esti-high-int} and \eqref{eq 1m bootstrap-int},
for $|I|+j\leq N-1$
\begin{equation}\label{eq 1' decay-w-basic-int}
|(s/t)\del_{\alpha}\del^IL^ju| + |\delu_1\del^IL^ju|\leq CC_1\vep t^{-1/2}s^{\delta},
\end{equation}
\begin{equation}\label{eq 1 decay-w-basic-int}
|(s/t)\del^IL^j\del_{\alpha}u| + |\del^IL^j\delu_1u|\leq CC_1\vep t^{-1/2}s^{\delta}.
\end{equation}
For $|I|+j\leq N-2$, by \eqref{eq 1 lem 2' esti-high-int} and \eqref{eq 1m bootstrap-int},
\begin{equation}\label{eq 3 decay-w-basic-int}
|\del^IL^j\del_{\alpha}\delu_1 u| + |\del^IL^j\delu_1\del_{\alpha} u|\leq CC_1\vep t^{-1/2}s^{-1+\delta}.
\end{equation}
By \eqref{eq 0 lem 2' esti-high-int} and \eqref{eq 1m bootstrap-int},
\begin{equation}\label{eq 4 decay-w-basic-int}
|\del^IL^j\delu_1\delu_1u| + |\delu_1\delu_1\del^IL^j u|\leq CC_1\vep t^{-3/2}s^{\delta}.
\end{equation}

We also establish the following bounds on Klein-Gordon component, listed as following:

For $|I|+j\leq N$,
\begin{equation}\label{eq 1m L2-KG-basic-int}
\|(s/t)\del^IL^j\del_{\alpha} v\|_{L^2(\Hcal^*_s)} + \|c\del^IL^j v\|_{L^2(\Hcal^*_s)}\leq CC_1\vep s^{\delta}.
\end{equation}
For $|I|+j\leq N-1$,
\begin{equation}\label{eq 1l L2-KG-basic-int}
\|(s/t)\del^IL^j\del_{\alpha} v\|_{L^2(\Hcal^*_s)} + \|c\del^IL^j v\|_{L^2(\Hcal^*_s)}\leq CC_1\vep .
\end{equation}

Also, for $|I| + j\leq N-1$,
\begin{equation}\label{eq 1m decay-KG-basic-int}
 |c\del^IL^j v|\leq CC_1\vep t^{-1/2}s^{\delta}
\end{equation}
and for $|I|+j\leq N-2$
\begin{equation}\label{eq 1l decay-KG-basic-int}
 |c\del^IL^j v|\leq CC_1\vep t^{-1/2}.
\end{equation}
Also, because of the relation $t^{-1}\delu_1 = L$, we have for $|I|+j\leq N-3$
\begin{equation}\label{eq 1ll decay-KG-basic-int}
 |c\del^IL^j \delu_1v|\leq CC_1\vep t^{-3/2}.
\end{equation}

%

\subsection{Sharp decay bounds on $\del^IL^j\nabla u$}
We start by establishing
\begin{equation}\label{eq 1.85' w-sharp}
|\del_t\del^IL^j u| \leq CC_1\vep (t-r)^{-1/4}, \quad |I|+j\leq N-2.
\end{equation}
This is based on the following decomposition of the wave operator with respect to the semi-hyperboloidal frame (which has been applied in two space dimension in \cite{M3}):
\begin{equation}\label{eq 1 decompo-Box-int}
\aligned
\Box \del^IL^j u =& (t-r)^{-\beta}\left((s/t)^2\del_t + (2x/t)\delu_1 \right)((t-r)^{\beta}\del_t \del^IL^j u)
+ \frac{t-r}{t^2}\left(\frac{t+r}{t}-\beta\right)\del_t \del^IL^j u
\\
&- \delu_1\delu_1 \del^IL^j u.
\endaligned
\end{equation}
and the fact that
\begin{equation}\label{eq 1.95 w-sharp}
|(t-r)^{\beta}\del_t\del^IL^ju|\leq CC_1\vep,\quad |I|+j\leq N-2
\end{equation}
on the cone $\{r = t-1\}$, which is direct form \eqref{eq 1l decay-w-basic-ext}.

More precisely, we write the \eqref{eq 1 decompo-Box-int} into the following form:
\begin{equation}\label{eq 1; decompo-Box-int}
\aligned
\frac{(t-r)^\beta t^2}{t^2+r^2}\big(\Box \del^IL^j u + \delu_1\delu_1\del^IL^ju\big) =& \mathcal{J} U_{\beta} + P U_{\beta}
\endaligned
\end{equation}
with
$$
\mathcal{J} = \del_t + \frac{2tx}{t^2+r^2}\del_x,\quad P = \frac{t-r}{t^2+r^2}\left(\frac{t+r}{t} - \beta\right), \quad U_{\beta} = (t-r)^{\beta}\del_t\del^IL^ju.
$$

When taking the wave equation in \eqref{eq main}, the above identity leads to
\begin{equation}\label{eq 1 proof lem 1 w-sharp}
\mathcal{J} U_{\beta} + P U_{\beta} = \frac{(t-r)^{\beta}t^2}{t^2+r^2}\big( \del^IL^j(v^3) + t^{-1}L\delu_1\del^IL^j u\big).
\end{equation}

Then for the integral curve of $\mathcal{J}$, we have the following description:
\begin{lemma}\label{lem 0 w-sharp}
Let $(t_2,x_2)$ be a interior point of $\Hcal^*_{[2,s_1]}$ and let $\gamma(t;t_2,x_2)$ be the integral curve of $\mathcal{J}$ with
$$
\gamma(t_2;t_2,x_2) = (t_2,x_2).
$$
Then there exists $2\leq t_0<t$ such that $\gamma(t_0;t_2,x_2)\in \Hcal_2^*\cup \del \Kcal_{[2,s_2]}$ and the arc $\{\gamma(t;t_2,x_2), t_0\leq t\leq t_2\}$ is contained in $\Hcal^*_{[2,s_2]}$ with $s^2_2 = t^2_2-x^2_2$. $\del\Kcal_{[2,s_2]}$ is the conical boundary of $\Hcal^*_{[2,s_2]}$ which is $\left\{(t,x)|t=r+1,5/2\leq t\leq (s_2^2+1)/2\right\}$.
\end{lemma}
\begin{proof}
$\mathcal{J}$ is defined in the region $\{t>0\}$. So all of its integral curve extend to the boundary of this region. We define
$$
t_0 = \inf\{t\in[2,t_2]\ |\ \forall t\leq \tau\leq t_2, \gamma(\tau;t_2,x_2)\in\Hcal^*_{[2,s_2]}\}.
$$
Because $(t_2,x_2)$ is an interior point of $\Hcal^*_{[2,s_2]}$, $t_0<t$. By continuity, $\gamma(t_0;t_2,x_2)$ locates on the boundary of $\Hcal^*_{[2,s_2]}$. We now prove that $\gamma(t_0;t_2,x_2)\notin \Hcal^*_{s_2}$.

To do so, we only need to remark that along $\gamma(\cdot;t_2,x_2)$, $s = \sqrt{t^2-r^2}$ is strictly increasing with respect to $t$. That is because
$$
\mathcal{J} (s^2) = \frac{2t(t^2-r^2)}{t^2+r^2}>0.
$$
\end{proof}

\begin{lemma}\label{lem 1 w-sharp}
Under the bootstrap assumption, \eqref{eq 1.85' w-sharp} holds.
\end{lemma}
\begin{proof}
For the bound on $\del_t u$, recall the notation in \eqref{eq 1 proof lem 1 w-sharp}, we denote by
$$
\aligned
&u_{t,x,\beta} (\tau) := U_{\beta}|_{\gamma(\tau,t,x)}, \quad p_{t,x}(\tau) := P|_{\gamma(\tau;t,x)},
\\
&R_{t,x}(\tau) := \frac{(t-r)^{\beta}t^2}{t^2+r^2}\big( \del^IL^j(v^3) + t^{-1}L\delu_1\del^IL^j u\big)\bigg|_{\gamma(\tau;t,x)}.
\endaligned
$$
and \eqref{eq 1 proof lem 1 w-sharp} is written as
\begin{equation}\label{eq 2 proof lem 1 w-sharp}
u_{t,x,\beta}'(\tau) + p_{t,x}(\tau) u_{t,x,\beta}(\tau) = R_{t,x}(\tau).
\end{equation}

We remark that for $|I|+j\leq N-2$,
\begin{equation}\label{eq 2.5 proof lem 1 w-sharp}
|\del^IL^j(v^3)|\leq \sum_{I_1+I_2+I_3=I\atop j_1+j_2+j_3=j}|\del^{I_1}L^{j_1}v\cdot \del^{I_2}L^{j_2}v\cdot \del^{I_3}L^{j_3}v|\leq C(C_1\vep)^3t^{-3/2}
\end{equation}
where we applied \eqref{eq 1l decay-KG-basic-int}.

Recall that $p_{t,x}\geq 0$, and by \eqref{eq 2.5 proof lem 1 w-sharp} and \eqref{eq 1 decay-w-basic-int}, for $|I|+j\leq N-2$,
\begin{equation}\label{eq 3 proof lem 1 w-sharp}
|R_{t,x}(\tau)|\leq CC_1\vep \tau^{-3/2+\beta+\delta} .
\end{equation}

Then integrate the above equation on $[t_0,t]$ where $t_0$ is determined by the above lemma \ref{lem 0 w-sharp} and taking into consideration of \eqref{eq 1.95 w-sharp}, we remark that $u_{t,x,\beta}$ is bounded
$CC_1\vep$. Then \eqref{eq 1.85' w-sharp} is established (when fix $\beta = 1/4$).
\end{proof}

Form the relation
$$
\del_x u = \delu_1 u - (x/t)\del_tu
$$
and the bound \eqref{eq 1' decay-w-basic-int} combined with \eqref{eq 1.85' w-sharp} and \eqref{eq 1 lem 2 high-order-int}, we obtain
\begin{equation}\label{eq 3 w-sharp}
|\del_{\alpha}\del^IL^j u| + |\del^IL^j\del_{\alpha}u|\leq CC_1\vep (t-r)^{-1/4},\quad |I|+j\leq N-2.
\end{equation}

Now we are going to establish the following bound:
\begin{equation}\label{eq 6 w-sharp}
|L\del^IL^ju(t,x)|\leq CC_1\vep (t-r)^{3/4}, \quad |I|+j\leq N-3.
\end{equation}
This is by the following observation. From \eqref{eq 3 w-sharp} and lemma \ref{lem 1 high-order-int},
$$
|\frac{x^a}{r}\del_aL\del^IL^ju(t,x)| = |\del_rL\del^IL^ju(t,x)|\leq CC_1\vep (t-r)^{-1/4},\quad |I|+j\leq N-3.
$$
By \eqref{eq L null-frame} applied on the frontier of $\Tcal_s\backslash \Hcal_s^*$, for $|I|+j\leq N-1$,
$$
|L\del^IL^ju(t,t-1)|\leq CC_1\vep.
$$
Then by integration along the radial direction, we obtain \eqref{eq 6 w-sharp}. Then recall the relation $L = t^{-1}\delu_1$, we obtain
\begin{equation}\label{eq 7 w-sharp}
|\delu_1\del^IL^ju(t,x)|\leq CC_1\vep t^{-1}(t-r)^{3/4},\quad |I|+j\leq N-3.
\end{equation}

%

\subsection{Sharp decay bounds on $ \del^IL^j\del_t\del_t u$}

We first establish the following bound:
\begin{equation}\label{eq 1 w-sharp}
(s/t)^2|\del_t\del_t \del^IL^j u|\leq CC_1\vep t^{-1}(t-r)^{-1/4}, \quad |I|+j\leq N-3.
\end{equation}

For this objective we recall the following decomposition:
$$
\Box \del^IL^j u = (s/t)^2\del_t\del_t \del^IL^j u + t^{-1}\big((2x/t)L\del_t\del^IL^j u - L\delu_1\del^IL^j u + (s/t)^2\del_t\del^IL^j u\big)
$$
which leads to
$$
(s/t)^2\del_t\del_t \del^IL^j u = \del^IL^j (v^3) - t^{-1}\big((2x/t)L\del_t\del^IL^j u - L\delu_1\del^IL^j u + (s/t)^2\del_t\del^IL^ju\big).
$$
Then by \eqref{eq 2.5 proof lem 1 w-sharp} and \eqref{eq 1 decay-w-basic-int},
$$
(s/t)^2|\del_t\del_t\del^IL^j u|\leq CC_1\vep t^{-3/2} + Ct^{-1}\big(|L\del_\alpha\del^IL^j u| + |\del_t\del^IL^j u|\big).
$$
So the question reduced to the following bounds:
$$
|\del_t\del^IL^j u| + |L\del_\alpha\del^IL^j u| \leq CC_1\vep (t-r)^{-1/4}, \quad |I|+j\leq N-3,
$$
which is guaranteed by \eqref{eq 3 w-sharp}. Thus \eqref{eq 1 w-sharp} is established (thanks to \eqref{eq 1 lem 2 high-order-int}).

We recall the following relations:
$$
\del_t\del_x u = \del_x\del_t u = t^{-1}L\del_tu - (x/t)\del_t\del_tu,\quad \del_x\del_x u = t^{-1}L\del_xu - (x/t)\del_t\del_xu.
$$
Then combined with \eqref{eq 3 w-sharp} and \eqref{eq 1 w-sharp},
\begin{equation}\label{eq 4 w-sharp}
|(s/t)^2\del_{\alpha}\del_{\beta}\del^IL^j u|\leq CC_1\vep t^{-1}(t-r)^{-1/4},\quad |I|+j\leq N-3.
\end{equation}
Then by \eqref{eq 1 lem 2 high-order-int},
\begin{equation}\label{eq 5 w-sharp}
|(s/t)^2\del^IL^j\del_{\alpha}\del_{\beta}u|\leq CC_1\vep t^{-1}(t-r)^{1/4},\quad |I|+j\leq N-3.
\end{equation}

\subsection{Estimates on source term in interior region}
In this section we are going to establish the following bounds:
\begin{lemma}[Source for wave equation in interior region]\label{lem 1 source-w-int}
Under the bootstrap bound \eqref{eq 1 bootstrap-int}, the following bound holds for $|I|+j\leq N$:
\begin{equation}\label{eq 1 lem 1 source-w-int}
\|\del^IL^j\del_\alpha(v^3)\|_{L^2(\Hcal^*_s)} +
\|\del^IL^j(v^3)\|_{L^2(\Hcal^*_s)}\leq C(C_1\vep)^3s^{-1+\delta}
\end{equation}
and
\begin{equation}\label{eq 2 lem 1 source-w-int}
\|\del^IL^{j+1}(v^3)\|_{L^2(\Hcal^*_s)}\leq C(C_1\vep)^3s^{\delta}.
\end{equation}
\end{lemma}
\begin{proof}
Remark that
$$
\del^IL^j(v^3) = \sum_{I_1+I_2+I_3=I\atop j_1+j_2+j_3=j}\del^{I_1}L^{j_1}v\cdot \del^{I_2}L^{j_2}v\cdot \del^{I_3}L^{j_3}v.
$$
Then it is quite similar to what we have done in the proof of lemma \ref{lem 1 source-w-ext}. We apply \eqref{eq 1l decay-KG-basic-int} for lower order factor and \eqref{eq 1m L2-KG-basic-int} for higher order factor when $|I|+j\leq N$.

For $|I|+j = N+1$, we remark that
$$
|\del^IL^j Lv| = |\del^I(t\delu_1L^jv)|\leq Ct\sum_{|I'|\leq|I|}|\del^{I'}\delu_1L^jv|.
$$
Then \eqref{eq 1m L2-KG-basic-int} leads to
\begin{equation}\label{eq 1h L2-KG-basic-int}
\|t^{-1}\del^IL^j Lv\|_{L^2(\Hcal_s)}\leq CC_1\vep s^{\delta}
\end{equation}
which is parallel to \eqref{eq 1 pr lem 1 decay-KG-source}.
Now for \eqref{eq 2 lem 1 source-w-int}, we apply \eqref{eq 1l decay-KG-basic-int} and \eqref{eq 1h L2-KG-basic-int}.
\end{proof}

Then we recall lemma \ref{lem 2 Hessian-int}. Combined with \eqref{eq 1m bootstrap-int} and \eqref{eq 1 lem 1 source-w-int},
\begin{equation}\label{eq 1m L2-hessian-int}
\|s(s/t)^2\del^IL^j\del_{\alpha}\del_{\beta}u\|_{L^2(\Hcal^*_s)}\leq CC_1\vep s^{\delta},\quad |I|+j\leq N-1.
\end{equation}
For $|I|+j = N$, we recall \eqref{eq 2.5 L2-w-basic-int}.

Furthermore for $|I|+j\leq N-2$,
\begin{equation}\label{eq 4 lem 1 hessian-w}
|(s/t)^2\del^IL^j\big(\del_{\alpha}\del_{\beta}u\big)|\leq C(C_1\vep)t^{-1/2}s^{-1+\delta}.
\end{equation}


Now we are ready to establish the following bounds:

\begin{lemma}[Source of Klein-Gordon in interior region I]\label{lem 1 source-KG-int}
Under the bootstrap argument, the following bound holds for $|I|+j\leq N$:
\begin{equation}\label{eq 1 lem 0 source-KG-int}
\|\del^IL^j\big(m^{\mu\nu}N^{\alpha\beta}\del_{\mu}\del_{\alpha}u\del_{\nu}\del_{\beta}u\big)\|_{L^2(\Hcal^*_s)}\leq C(C_1\vep)^2s^{-1+\delta}.
\end{equation}
For $|I|+j\leq N-1$
\begin{equation}\label{eq 1 lem 1 source-KG-int}
\|\del^IL^j\big(m^{\mu\nu}N^{\alpha\beta}\del_{\mu}\del_{\alpha}u\del_{\nu}\del_{\beta}u\big)\|_{L^2(\Hcal^*_s)}\leq C(C_1\vep)^2s^{-3/2+2\delta}.
\end{equation}
\end{lemma}
\begin{proof}
The proof of theses two estimates are quite similar. We need to evoke the null structure.
$$
\del_{\alpha}\del_{\beta}u = \Psiu_{\alpha'}^{\alpha}\delu_{\alpha'}\big(\Psiu_{\beta}^{\beta'}\delu_{\beta'}u\big) = \Psiu_{\alpha}^{\alpha'}\Psiu_{\beta}^{\beta'}\delu_{\alpha'}\delu_{\beta'}u + \del_{\alpha}\big(\Psiu_{\beta}^{\beta'}\big)\delu_{\beta'}u
$$
Then
$$
\aligned
&m^{\mu\nu}N^{\alpha\beta}\del_{\mu}\del_{\alpha}u\del_{\nu}\del_{\beta}u
\\
=& \miu^{\mu\nu}\Nu^{\alpha\beta}\delu_{\mu}\delu_{\alpha}u\delu_{\nu}\delu_{\beta}u
\\
&+ m^{\mu\nu}N^{\alpha\beta}\Psiu_{\alpha}^{\alpha'}\del_{\mu}\delu_{\alpha'}u\cdot \del_{\nu}(\Psiu_{\beta}^{\beta'})\delu_{\beta'}u
 + m^{\mu\nu}N^{\alpha\beta}\Psiu_{\beta}^{\beta'}\del_{\nu}\delu_{\beta'}u\cdot \del_{\mu}(\Psiu_{\alpha}^{\alpha'})\delu_{\alpha'}u
\\
&+ m^{\mu\nu}N^{\alpha\beta}\del_{\mu}(\Psiu_{\alpha}^{\alpha'})\del_{\nu}(\Psiu_{\beta}^{\beta'})\delu_{\alpha'}u\del_{\beta'}u
\\
=:&T_1+T_2+T_3+T_4.
\endaligned
$$

{\bf For the term $T_1$},
$$
\aligned
T_1
 =& \miu^{00}\Nu^{00}\del_t\del_tu\del_t\del_tu + 2\miu^{01}\Nu^{00}\del_t\del_tu \delu_1\del_tu + \miu^{11}\Nu^{00}\delu_1\del_tu\delu_1\del_tu
\\
&+2\miu^{00}\Nu^{01}\del_t\del_tu\del_t\delu_1u + 2\miu^{01}\Nu^{01}\del_t\del_tu\delu_1\delu_1u + 2\miu^{10}\Nu^{01}\delu_1\del_tu\del_t\delu_1u
\\
&+2\miu^{11}\Nu^{01}\delu_1\del_tu\delu_1\delu_1u
\\
&+\miu^{00}\Nu^{11}\del_t\delu_1u\del_t\delu_1u +  2\miu^{01}\Nu^{11}\del_t\delu_1u\delu_1\delu_1u + \miu^{11}\Nu^{11}\delu_1\delu_1u\delu_1\delu_1u.
\endaligned
$$

These ten terms can be classified into four groups (denoted by $T_{11}, T_{12}, T_{13}$ and $T_{14}$). The first term forms the first group. The second group is composed by the second, the third, the forth and the eighth term. The forth group only contains the fifth term.  The rest terms are left in the third group.

When $|I|+j\leq N$, we regard the bound on $T_{11}$.
$$
\del^IL^j\big(\miu^{00}\Nu^{00}\del_t\del_tu\del_t\del_tu\big) = \sum_{I_1+I_2+I_3=I\atop j_1+j_2+j_3=j}\del^{I_1}L^{j_1}\big(\miu^{00}\Nu^{00}\big)\del^{I_2}L^{j_2}\del_t\del_tu\cdot \del^{I_3}L^{j_3}\del_t\del_tu.
$$
In RHD of the above expression, recall the null condition on $m^{\alpha\beta}$ and $N^{\alpha\beta}$, thanks to \eqref{eq-SHF null}
$$
|\del^IL^j(\miu^{00}\Nu^{00})|\leq C(s/t)^4.
$$
Then, by \eqref{eq 2.5 L2-w-basic-int} and \eqref{eq 4 lem 1 hessian-w}, the $L^2$ norm of each term in RHD is controlled by $C(C_1\vep)^2s^{-3/2+2\delta}$.

The $L^2$ norm on $\Hcal^*_s$ of the terms in $T_{12}$ are bounded by $C(C_1\vep)^2s^{-3/2+2\delta}$. There are only one good coefficient $\miu^{00}$ or $\Nu^{00}$. We take the second term in the expression of $T_1$ as an example, the rest terms are bounded in the same manner:
$$
\del^IL^j(\miu^{01}\Nu^{00}\del_t\del_tu \delu_1\del_tu) = \sum_{I_1+I_2+I_3=I\atop j_1+j_2+j_3=j}\del^{I_1}L^{j_1}\big(\miu^{01}\Nu^{00}\big)\del^{I_2}L^{j_2}\del_t\del_tu \del^{I_3}L^{j_3}\delu_1\del_tu
$$
In the RHD of the above expression, the first factor is bounded by $C(s/t)^2$ (thanks to \eqref{eq-SHF null} and the fact that $\miu^{01}$ is homogeneous of degree zero). Then we apply \eqref{eq 2.5 L2-w-basic-int} or \eqref{eq 4 lem 1 hessian-w} on the factor of $\del_t\del_tu$, and \eqref{eq 2 L2-w-basic-int} or \eqref{eq 3 decay-w-basic-int} on the factor of $\delu_1\del_tu$, then the $L^2$ norm of each term in RHD of the above expression is controlled by $C(C_1\vep)^2s^{-3/2+2\delta}$.

The $L^2$ norm on $\Hcal^*_s$ of the terms in $T_{13}$ are bounded by $C(C_1\vep)^2s^{-3/2+2\delta}$. For these term, there are no supplementary decay supplied by the coefficients, however, the terms contain more good derivatives, which supply a sufficient bound. We take the sixth term as an example:
$$
\del^IL^j\big(\miu^{10}\Nu^{01}\delu_1\del_tu\del_t\delu_1u\big)
= \sum_{I_1+I_2+I_3=I\atop j_1+j_2+j_3=j}\del^{I_1}L^{j_1}(\miu^{10}\Nu^{01})\del^{I_2}L^{j_2}\delu_1\del_tu \del^{I_3}L^{j_3}\del_t\delu_1u
$$
Recall that $\miu^{10}\Nu^{01}$ is homogeneous of degree zero, so the first factor is bounded. We apply \eqref{eq 2 L2-w-basic-int} and \eqref{eq 3 decay-w-basic-int} on the last two factors.

$T_{14}$ is critical term. Again,
$$
\del^IL^j\big(\miu^{01}\Nu^{01}\del_t\del_tu\delu_1\delu_1u\big) =  \sum_{I_1+I_2+I_3=I\atop j_1+j_2+j_3 = j}\del^{I_1}L^{j_1}(\miu^{01}\Nu^{01})\del^{I_2}L^{j_2}\del_t\del_tu \del^{I_3}L^{j_3}\delu_1\delu_1u.
$$
In the RHD of the above expression, we remark that $\miu^{01}\Nu^{01}$ is homogeneous of degree zero, thus bounded. For the last two factors, we make the following discussion:
\\
$-$ when $|I_3| + j_3 = N$, $|I_2|+j_2 =0$. We apply \eqref{eq 1 w-sharp} on $\del_t\del_tu$ and \eqref{eq 3 L2-w-basic-int} on $\del^{I_3}L^{j_3}\delu_1\delu_1u$. In this case the $L^2$ norm is bounded by $C(C_1\vep)^2s^{-1+\delta}$.
\\
$-$ when $|I_3| + j_3 \leq N-1$ and $|I_2|+j_2\leq N-1$, we apply \eqref{eq 1m L2-hessian-int} or \eqref{eq 4 lem 1 hessian-w} on $\del_t\del_t u$ and \eqref{eq 3' L2-w-basic-int} or \eqref{eq 4 decay-w-basic-int} on $\delu_1\delu_1 u$. In this case the $L^2$ norm is bounded by $C(C_1\vep)^2s^{-3/2+2\delta}$.
\\
$-$ when $|I_2|+j_2=N$, $|I_3|+j_3=0$, we apply  \eqref{eq 2.5 L2-w-basic-int} on $\del_t\del_t u$ and \eqref{eq 4 decay-w-basic-int} on $\delu_1\delu_1 u$. In this case the $L^2$ norm is bounded by $C(C_1\vep)^2s^{-3/2+2\delta}$.

Now we see that when $|I|+j\leq N-1$, all terms are bounded by $C(C_1\vep)^2s^{-3/2+2\delta}$ while when $|I|+j=N$, there is one term ($\|\miu^{01}\Nu^{01}\del_t\del_tu\del^IL^j\delu_1\delu_1 u\|_{L^2(\Hcal_s)}$) bounded by $C(C_1\vep)^2s^{-1+\delta}$ and all other terms are bounded by $C(C_1\vep)^2s^{-3/2+2\delta}$. So the desired bounds are proved.

{\bf For the terms in $T_2$ or $T_3$}, there is one factor homogeneous of degree $-1$ but this is not sufficient. The key is the following observation:
\begin{equation}\label{eq 1 proof lem 1 source-KG-int}
\del_t\big(\Phiu_1^0\big) = \frac{x}{t^2},\quad \del_x\big(\Phiu_1^0\big) = -t^{-1}.
\end{equation}
The derivatives of other components are zero. Furthermore,
\begin{equation}\label{eq 2 proof lem 1 source-KG-int}
\delu_1\big(\Phiu_1^0\big) = -t^{-1}(s/t)^2.
\end{equation}
Then, we evoke the null structure of $T_2$ and $T_3$:
$$
\aligned
&m^{\mu\nu}N^{\alpha\beta}\Psiu_{\alpha}^{\alpha'}\del_{\mu}\delu_{\alpha'}u\cdot \del_{\nu}(\Psiu_{\beta}^{\beta'})\delu_{\beta'}u
\\
=&\miu^{\mu\nu}N^{\alpha\beta}\Psiu_{\alpha}^{\alpha'}\delu_{\mu}\delu_{\alpha'}u\cdot \delu_{\nu}(\Psiu_{\beta}^{\beta'})\delu_{\beta'}u
\\
=&\miu^{00}N^{\alpha\beta}\Psiu_{\alpha}^{\alpha'}\del_t\delu_{\alpha'}u\cdot \del_t(\Psiu_{\beta}^{\beta'})\delu_{\beta'}u
+\miu^{01}N^{\alpha\beta}\Psiu_{\alpha}^{\alpha'}\del_t\delu_{\alpha'}u\cdot \delu_1(\Psiu_{\beta}^{\beta'})\delu_{\beta'}u
\\
&+\miu^{10}N^{\alpha\beta}\Psiu_{\alpha}^{\alpha'}\delu_1\delu_{\alpha'}u\cdot \del_t(\Psiu_{\beta}^{\beta'})\delu_{\beta'}u
 +\miu^{11}N^{\alpha\beta}\Psiu_{\alpha}^{\alpha'}\delu_1\delu_{\alpha'}u\cdot \delu_1(\Psiu_{\beta}^{\beta'})\delu_{\beta'}u.
\endaligned
$$
Then by a the same argument applied on $T_1$, we see that when $|I|+j = N$ the $L^2$ norm is bounded by $C(C_1\vep)^2s^{-1+\delta}$ and when $|I|+j\leq N-1$, it can be controlled by $C(C_1\vep)^{-3/2+2\delta}$. This proves the desired result.

{\bf For the term $T_4$}, we only need to remark that there are two factors homogeneous of degree $-1$, which provides sufficient decay rate.

\end{proof}

\begin{lemma}\label{lem 2 source-KG-int}
Under the bootstrap assumption, the following bound holds:
\begin{equation}\label{eq 1 lem 2 source-KG-int}
\|\del^IL^j(N^{\alpha\beta}\del_{\alpha}(v^3)\del_{\beta}u)\|_{L^2(\Hcal^*_s)}\leq C(C_1\vep)^4s^{-3/2+3\delta}.
\end{equation}
\end{lemma}
\begin{proof}
We remark that
\begin{equation}\label{eq 1 proof lem 2 source-KG-int}
\aligned
\del^IL^j(N^{\alpha\beta}\del_{\alpha}(v^3)\del_{\beta}u)
=& 3\sum_{I_1+I_2=I\atop j_1+j_2=j}\del^{I_1}L^{j_1}(v^2)\cdot \del^{I_2}L^{j_2}\big(N^{\alpha\beta}\del_{\alpha}v\del_{\beta}u\big)
\endaligned
\end{equation}

We first prove that
\begin{equation}\label{eq 2 proof lem 2 source-KG-int}
|\del^IL^j (N^{\alpha\beta}\del_{\alpha}u\del_{\beta}u)|\leq C(C_1\vep)^2t^{-1}s^{\delta}, \quad |I|+j\leq N-3.
\end{equation}
This is because the following calculation:
\begin{equation}\label{eq 3 proof lem 2 source-KG-int}
\aligned
N^{\alpha\beta}\del_{\alpha}v\del_{\beta}u =& \Nu^{\alpha\beta}\delu_{\alpha}v\delu_{\beta}u
\\
=& \Nu^{00}\del_tv\del_tu + \Nu^{10}\delu_1v\del_tu + \Nu^{01}\del_tv\delu_1u + \Nu^{11}\delu_1v\delu_1u.
\endaligned
\end{equation}
Then we apply the null condition ($\del^IL^j(\Nu^{00})\sim(s/t)^2$), \eqref{eq 1 decay-w-basic-int} and \eqref{eq 1ll decay-KG-basic-int}.

Then we establish the following bound:
\begin{equation}\label{eq 4 proof lem 2 source-KG-int}
\|t^{1/2}\del^IL^j (N^{\alpha\beta}\del_{\alpha}v\del_{\beta}u)\|_{L^2(\Hcal^*_s)}\leq C(C_1\vep)^2s^{2\delta}, \quad |I|+j\leq N.
\end{equation}
This is also by \eqref{eq 3 proof lem 2 source-KG-int} and the application of \eqref{eq 1 L2-w-basic-int}, \eqref{eq 3 w-sharp} on $\del u$,  and \eqref{eq 1m L2-KG-basic-int}, \eqref{eq 1ll decay-KG-basic-int} on $\del v$.

We remark the following bounds:
\begin{equation}\label{eq 5 proof lem 2 source-KG-int}
\|t^{1/2}\del^IL^j (v^2)\|_{L^2(\Hcal_s^*)}\leq C(C_1\vep)^2 s^{\delta},\quad |I|+j\leq N
\end{equation}
and
\begin{equation}\label{eq 6 proof lem 2 source-KG-int}
|\del^IL^j(v^2)|\leq C(C_1\vep)^2 t^{-1},\quad |I|+j\leq N-2.
\end{equation}
These are by \eqref{eq 1m L2-KG-basic-int} and \eqref{eq 1l decay-KG-basic-int}.

Now we apply \eqref{eq 1 proof lem 2 source-KG-int}. When $|I_1|+j_1\geq [N/2]+1$, we apply \eqref{eq 5 proof lem 2 source-KG-int} in $v^2$ and \eqref{eq 2 proof lem 2 source-KG-int} on $N^{\alpha\beta}\del_{\alpha}u\del_{\beta}u$. When $|I_2|+j_2\geq[N/2]+1$, we apply \eqref{eq 4 proof lem 2 source-KG-int} on $N^{\alpha\beta}\del_{\alpha}u\del_{\beta}u$ and \eqref{eq 6 proof lem 2 source-KG-int} on $v^2$
\end{proof}
 
%
%

%
%
%
%
%
%
%
%
%
%

\subsection{Refined energy bound in interior region}
Exactly as what we have done in section \ref{subsec energy-refine-ext}, we will establish the following energy bounds:
\begin{subequations}\label{eq 1 refined-int}
\begin{equation}\label{eq 1h refined-int}
\EH_N(s,Lu)^{1/2} \leq C_0\vep +  C(C_1\vep)^2 s^{1+\delta},
\end{equation}
\begin{equation}\label{eq 1m refined-int}
\aligned
\EH_N(s,u)^{1/2} + \EH_N(s,\del_{\alpha} u)^{1/2} + \EH_{N,c}(s,v)^{1/2} &\leq C_0\vep + C(C_1\vep)^2 s^{\delta},
\endaligned
\end{equation}
\begin{equation}\label{eq 1l refined-int}
 \EH_{N-1,c}(s,v)^{1/2}\leq C_0\vep + C(C_1\vep)^2.
\end{equation}
\end{subequations}

For the energy bounds of $u,\ \del_{\alpha}u$ and $Lu$,  we substitute the bounds \eqref{eq 1 lem 1 source-w-int}, \eqref{eq 2 lem 1 source-w-int} into \eqref{eq 1 prop 1 energy} and recall \eqref{eq energy-cone}.

For the bounds of $v$, we recall \eqref{eq main-trans} and the definition of $w$. For the bounds on $w$, we substitute the bounds \eqref{eq 1 lem 0 source-KG-int}, \eqref{eq 1 lem 1 source-KG-int} and \eqref{eq 1 lem 2 source-KG-int} into \eqref{eq 1 prop 1 energy} and recall \eqref{eq energy-cone}. To obtain the bounds of $v$, we need to estimate the energy of the correction term $N^{\alpha\beta}\del_{\alpha}u\del_{\beta}u$. To do so, we establish the following bounds:
\begin{lemma}\label{lem 1 corrcetion-int}
Under the bootstrap assumption, the following bounds hold:
\begin{equation}\label{eq 1 correction-int}
\|\del_{\gamma}\del^IL^j(N^{\alpha\beta}\del_{\alpha}u\del_{\beta}u)\|_{L^2(\Hcal_s^*)} + \|\del^IL^j\big(N^{\alpha\beta}\del_{\alpha}u\del_{\beta}u\big)\|_{L^2(\Hcal^*_s)}\leq
\left\{
\aligned
& C(C_1\vep)^2 s^{\delta},\quad &&|I|+j = N,
\\
& C(C_1\vep)^2 s^{-1/2+2\delta},  &&|I|+j \leq N-1.
\endaligned
\right.
\end{equation}
\end{lemma}
\begin{proof}
For the first term, we make the following calculation:
\begin{equation}\label{eq 1 pr lem 1 correction-int}
\aligned
\del_{\gamma}\del^IL^j\big(N^{\alpha\beta}\del_{\alpha}u\del_{\beta}u\big)
=& \del_{\gamma}\del^IL^j\big(\Nu^{\alpha\beta}\delu_{\alpha}u\delu_{\beta}u\big)
\\
=& \del_{\gamma}\del^IL^j\big(\Nu^{00}\del_tu\del_tu\big)
+ \del_{\gamma}\del^IL^j\big(\Nu^{10}\delu_1\del_{\gamma}u\del_tu \big)
\\
&+ \del_{\gamma}\del^IL^j\big(\Nu^{01}\del_t\del_{\gamma}u\delu_1u\big) + \del_{\gamma}\del^IL^j\big(\Nu^{11}\delu_1\del_\gamma u\delu_1u\big).
\endaligned
\end{equation}
We only give a sketch. First, remark the null structure $\del^IL^jN^{00}\sim (s/t)^2$.

For $|I|+j = N$, we apply the $L^2$ bounds \eqref{eq 2.5 L2-w-basic-int}, \eqref{eq 2 L2-w-basic-int} and decay bounds \eqref{eq 3 decay-w-basic-int}, \eqref{eq 5 w-sharp} on $\del\del u$, and  the $L^2$ bounds \eqref{eq 1 L2-w-basic-int}, the decay bounds \eqref{eq 3 w-sharp} applied on $\del u$.

For $|I|+j\leq N-1$ we apply \eqref{eq 2' L2-w-basic-int} instead of \eqref{eq 2 L2-w-basic-int} and \eqref{eq 1m L2-hessian-int} instead of \eqref{eq 2.5 L2-w-basic-int}.

For the second term in LHS of \eqref{eq 1 correction-int}, we only need to apply the null structure and \eqref{eq 1 L2-w-basic-int}, \eqref{eq 3 w-sharp} and \eqref{eq 7 w-sharp}.
\end{proof}

Based on the above bounds, and thanks to \eqref{eq 1 lem 2 high-order-int}, we can turn the bounds on the norms $\|\del^IL^j\del_{\gamma}\big(N^{\alpha\beta}\del_{\alpha}u\del_{\beta}u\big)\|_{L^2(\Hcal_s^*)}$ into the bounds on the norms $\|\del_{\gamma}\del^IL^j\big(N^{\alpha\beta}\del_{\alpha}u\del_{\beta}u\big)\|_{L^2(\Hcal_s^*)}$. So we conclude that
\begin{equation}
E^H_c\big(s,\del^IL^j\big(N^{\alpha\beta}\del_{\alpha}u\del_{\beta}u\big)\big)^{1/2}\leq
\left\{
\aligned
& C(C_1\vep)^2 s^{\delta},\quad &&|I|+j = N,
\\
& C(C_1\vep)^2 s^{-1/2+2\delta},  &&|I|+j \leq N-1.
\endaligned
\right.
\end{equation}
The by triangle inequality, we obtain the energy bounds on $v$.

\section{Conclusion of the bootstrap argument}\label{sec conclusion}
Based on the bounds \eqref{eq 1 refined-ext} and \eqref{eq 1 refined-int}, we see that we only need to make a choice of $(C_1,\vep)$  with $0< C_1\vep <1$ and
$$
C_0\vep + C(C_1\vep)^2\leq \frac{1}{2}C_1\vep.
$$
To do so, we only need to fix $C_1>2C_0$ and take $\vep <\min\{\frac{C_1-2C_0}{CC_1^2},C^{-1}\}$.

\appendix

\section{Recall of basic calculation}\label{sec basic}
In this section we recall some basic calculation which are valid in $\Fcal_{[s_0,\infty)}$.
We introduce the following family of vector fields:
$$
\mathscr{Z} := \{\del_t,\del_x,L\}
$$
with $Z_0 = \del_t,Z_1=\del_x$ and $Z_2=L$. A {\sl high-order derivative} of $\mathscr{Z}$ is written as
$$
Z^I = Z_{i_1}Z_{i_2}\cdots Z_{i_N}
$$
where $I=(i_1,i_2,\cdots,i_N)$ is a multi-index with $i_j\in\{0,1,2\}$. The {\sl type} of  a high-order derivative $Z^I$ is a couple of integers $(i,j)$, such that $Z^I$ contains $i$ partial derivatives ($\del_t,\del_x$) and $j$ Lorentzian boosts ($L$). For an operator of type $(i,0)$, we write $Z^I = \del^I$.
Then we recall the following result:
\begin{lemma}\label{lem 1 high-order}
Let $u$ be a function defined in $\Fcal_{[s_0,\infty)}$, sufficiently regular. Let $Z^K$ be a $N-$order operator of type $(i,j)$. Then the following bound holds:
\begin{equation}\label{eq 1 lem 1 high-order}
Z^Ku = \sum_{|I|=i\atop j'\leq j}\Theta^K_{Ij'}\del^IL^{j'}u
\end{equation}
with $\Theta_{Ij'}^K$ constants determined by $K$ and $I,j'$.
\end{lemma}
\begin{proof}[Sketch of proof]
This is by induction on $(i,j)$. We remark that $[L,\del_{\alpha}] = \gamma_{\alpha}^{\beta}\del_{\beta}$ with $\gamma_{\alpha}^{\beta}$ constants. Then by induction we establish the decomposition of $[L^j,\del_{\alpha}]$ and $[L^j,\del^I]$:
\begin{equation}\label{eq 1 proof lem 1 high-order}
[L^j,\del^I] = \sum_{|I'|=|I|\atop j'<j}\Theta^{jI}_{I'j'}\del^{I'}L^{j'}.
\end{equation}

Then, we write $Z^K$ into the following form:
$$
Z^K = L^{j_1}\del^{I_1}L^{j_2}\del^{I_2}\cdots L^{j_k}\del^{I_k}.
$$
where $j_1$ and $|I_k|$ might be zero. We make induction on $k$. When $k=1$, it is guaranteed by \eqref{eq 1 proof lem 1 high-order}. When $k>1$,
$$
Z^K = \del^{I_1}Z^{K'} + [L^{j_1},Z^{I_1}]Z^{K''},
$$
with
$$
Z^{K'} := L^{j_1+j_2}\del^{I_2}\cdots L^{j_k}\del^{I_k},\quad K^{''} = L^{j_2}\del^{I_2}\cdots L^{j_k}\del^{I_k}.
$$
Then by assumption of induction applied on $Z^{K'}$ and $Z^{K''}$ combined with \eqref{eq 1 proof lem 1 high-order}, the desired result is established.
\end{proof}

\section{Recall of basic calculus in interior region}\label{sec basic-int}
In this section we list the basic $L^2$ and decay bounds in the interior region $\Hcal^*_{[s_0,\infty)}$. The basic decompositions and estimates of commutators and high-order derivatives hold in the same manner as in \cite{LM1}. For detailed proofs, see \cite{M-RC-V1}. We emphasise that {\bf  all discussions in this suction are in $\Hcal^*_{[s_0,\infty)}$} unless otherwise specified and the functions under discussion are sufficiently regular defined in $\Hcal^*_{[s_0,s_1]}\subset \Hcal^*_{[s_0,\infty)}$.
\subsection{Families of vector fields}
In $\Hcal^*_{[s_,\infty)}$ we introduce the following vector fields:
$$
(s/t)\del_{\alpha},\quad \delu_1 := t^{-1}L = (x/t)\del_t+\del_x.
$$
They are called {\sl adapted partial derivatives} and {\sl hyperbolic derivative}. We denote by
$$
\Zint := \Zscr\cup \{(s/t)\del_t,\ (s/t)\del_x,\ \delu_1\}.
$$
A high-order operator $Z^I$ defined in $\Hcal^*_{[s_0,\infty)}$ takes factors in $\Zint$. The type of $Z^I$ is a quadruple of integers $(i,j,k,l)$ such that $Z^I$ contains $i$ partial derivatives, $j$ Lorentzian boots, $k$ adapted partial derivatives and $l$ hyperbolic derivatives.

\subsection{Basic decomposition of the high-order derivatives}
First we recall the notion of homogeneous functions (in interior region) which is introduced in \cite{LM1}:
\begin{definition}
Let $u$ be a $C^{\infty}$ function defined in $\{t>|x|\}$, satisfying the following properties:
\begin{itemize}
\item For $k\in\RR$,  $u(\lambda t,\lambda x) = \lambda^ku(t,x),\quad \forall \lambda>0$.
\\
\item $\del^Iu(1,x)$ is bounded by a constant $C$ (determined by $|I|$ and $u$) for $|x|< 1$.
\end{itemize}
Then $u$ is said to be {\sl homogeneous of degree $k$}.
\end{definition}
The following properties are immediate (for proof, see for example \cite{M-RC-V1}):
\begin{proposition}\label{prop 1 homo}
Let $u,v$ be homogeneous of degree $k,l$ respectively. Then
\begin{itemize}
\item When $k=l$, $\alpha u + \beta v$ is homogeneous of degree $k$ where $\alpha$ and $\beta$ are constants.
\\
\item $uv$ is homogeneous of degree $k+l$.
\\
\item $\del^IL^j u$ is homogeneous of degree $k-|I|$.
\\
\item There is a positive constant determined by $I,J$ and $u$ such that the following inequality holds in $\Kcal$:
\begin{equation}\label{eq 1 homo}
|\del^IL^ju|\leq Ct^{k-|I|}.
\end{equation}
\end{itemize}
\end{proposition}

Then we are ready to state the following results:
\begin{lemma}\label{lem 1 high-order-int}
Let $u$ be a function defined in $\Hcal^*_{[s_0,s_1]}$, sufficiently regular. Let $Z^K$ be a $N-$order operator of type $(j,i,0,l)$ and $N\geq 1$. Then the following identity holds:
\begin{equation}\label{eq 1 lem 2 high-order-int}
Z^Ku = \sum_{|I|\leq i, j'\leq j+l\atop |I|+j'\geq 1}t^{-l-i+|I|}\Delta_{Ij'}^K\del^IL^{j'}u
\end{equation}
with $\Delta_{Ij}^K$ homogeneous functions of degree zero.
\end{lemma}

We also recall the following bounds which will be applied in many contexts:
\begin{lemma}\label{lem 3 s/t}
In the region $\Kcal$, the following bounds hold for $k,l\in\mathbb{Z}$:
\begin{equation}\label{eq 1 lem 3 s/t}
\big|\del^IL^j\big((s/t)^kt^l\big)\big|\leq
\left\{
\aligned
&C(s/t)^kt^l,\quad &&|I|=0,
\\
&C(s/t)^kt^l(t/s^2),\quad &&|I|\geq 1.
\endaligned
\right.
\end{equation}
\end{lemma}
This can be proved by induction, for detail, see for example \cite{M-RC-V1}.
\subsection{Null condition in semi-hyperboloidal frame}
\begin{proposition}[Null condition in interior region]\label{prop 1 null-int}
Let $T$, $Q$ be null forms of two and three contravariant type respectively. Suppose that in $\{|x|\leq t-1\}$,
$$
T^{\alpha\beta}, Q^{\alpha\beta\gamma},\quad \text{are constants}.
$$
Then
\begin{equation}\label{eq-SHF null}
|\del^IL^j\Qu^{000}| + |\del^IL^j \Tu^{00}|
\leq
\left\{
\aligned
&C(s/t)^2,\quad &&|I|=0,
\\
&Ct^{-|I|}\leq C(s/t)^2t^{-|I|+1},\quad&&|I|>0.
\endaligned
\right.
\end{equation}
\end{proposition}

\subsection{Basic $L^2$ and $L^{\infty}$ bounds}\label{subsec basic-interior}
Based on the above decompositions, we can prove the following bounds (for proof, see \cite{LM1} or \cite{M-RC-V1}).
\begin{lemma}\label{lem 1 est-high-int}
Let $u$ be a function defined in $\Kcal_{[s_0,s_1]}$, sufficiently regular. Let $Z^K$ be an operator of type $(j,i,0,l)$, and let $|K|=N+1\geq 1$. Then the following bounds hold:
\begin{equation}\label{eq 0 lem 1 est-high-int}
 \|t^{l-1}Z^K u\|_{L^2(\Hcal_s)}\leq C\Ecal^H_N(s,u)^{1/2},\quad i=0,
\end{equation}
\begin{equation}\label{eq 1 lem 1 est-high-int}
\|(s/t)t^l Z^Ku\|_{L^2(\Hcal_s)}\leq C\Ecal^H_N(s,u)^{1/2},\quad i\geq 1,
\end{equation}
When $c>0$, the following bound holds for $|K|\leq N$:
\begin{equation}\label{eq 3 lem 1 est-high-int}
\|ct^l Z^K u\|_{L^2(\Hcal_s)}\leq C\Ecal^H_{N,c}(s,u)^{1/2}.
\end{equation}
\end{lemma}
The following bounds are to be combined with proposition \ref{prop 1 K-S}. It can be seen as a special case of lemma 4.5 of \cite{M-RC-V1}, or can be proved directly by lemma \ref{lem 1 est-high-int} combined with lemma \ref{lem 1 high-order-int} together with \eqref{eq 1 lem 3 s/t}.
\begin{lemma}\label{lem 2 esti-high-int}
Let $u$ be a function defined in $\Kcal_{[s_0,s_1]}$, sufficiently regular. Then the following bounds hold for $Z^K$ of type $(j,i,0,l)$ with $|K|\leq N$:
\begin{equation}\label{eq 0 lem 2 esti-high-int}
\big\|L\big(t^{l-1}Z^Ku\big)\big\|_{L^2(\Hcal_s)}\leq C\Ecal^H_N(s,u)^{1/2},\quad i=0
\end{equation}
\begin{equation}\label{eq 1 lem 2 esti-high-int}
\big\|L\big(t^l(s/t)Z^Ku\big)\big\|_{L^2(\Hcal_s)}\leq C\Ecal^H_N(s,t)^{1/2},\quad i\geq 1.
\end{equation}
When $c>0$ and $|K|\leq N - 1$,
\begin{equation}\label{eq 3 lem 2 esti-high-int}
\big\|cL\big(t^lZ^Ku\big)\big\|_{L^2(\Hcal_s)}\leq C\Ecal^H_{N,c}(s,t)^{1/2}.
\end{equation}
\end{lemma}

Then we combine lemma \ref{lem 2 esti-high-int} and proposition \ref{prop 1 K-S} and obtain the following bounds:
\begin{lemma}\label{lem 2' esti-high-int}
Let $u$ be a function defined in $\Kcal_{[s_0,s_1]}$, sufficiently regular. Then the following bounds hold for $Z^K$ of type $(j,i,0,l)$ with $1\leq |K|\leq N$:
\begin{equation}\label{eq 0 lem 2' esti-high-int}
\big\|t^{l-1/2}Z^Ku\big\|_{L^\infty(\Hcal_s)}\leq C\Ecal^H_N(s,u)^{1/2},\quad i=0
\end{equation}
\begin{equation}\label{eq 1 lem 2' esti-high-int}
\big\|t^{l+1/2}(s/t)Z^Ku\big\|_{L^\infty(\Hcal_s)}\leq C\Ecal^H_N(s,t)^{1/2},\quad i\geq 1.
\end{equation}
When $c>0$ and $|K|\leq N - 1$,
\begin{equation}\label{eq 3 lem 2' esti-high-int}
\big\|ct^{l+1/2}Z^Ku\big\|_{L^\infty(\Hcal_s)}\leq C\Ecal_{\gamma,c}^N(s,t)^{1/2}.
\end{equation}
\end{lemma}

\subsection{Bounds on Hessian form}
In this subsection we recall the bounds on the following terms:
$$
\del_{\alpha}\del_{\beta}Z^Ku,\quad Z^K\del_{\alpha}\del_{\beta}u.
$$
The following result is firstly established in \cite{LM1}. The following version is from \cite{M-RC-V1} proposition 4.7 combined with proposition \ref{prop 1 K-S}:
\begin{lemma}\label{lem 2 Hessian-int}
Let $u$ be a function defined in $\Hcal^*_{[s_0,s_1]}$, sufficiently regular. Suppose that $Z^K$ is of type $(j,i,0,0)$. Then the following bounds hold for $|K|\leq N-1$:
\begin{equation}\label{eq 1 lem 2 Hessian-int}
\big\|t(s/t)^3\del_{\alpha}\del_{\beta}Z^Ku\big\|_{L^2(\Hcal^*_s)} + \big\|t(s/t)^3Z^K\del_{\alpha}\del_{\beta}u\big\|_{\Hcal^*_s}\leq C\Ecal^H_N(s,u)^{1/2} + C\sum_{|I|+j\leq |K|}\|s\del^IL^j\Box u\|_{L^2(\Hcal^*_s)}.
\end{equation}
Furthermore,  for all $Z^{K'}$ of type $(j,i,0,0)$ with $|K'|\leq N-2$,
\begin{equation}\label{eq 2 lem 2 Hessian-int}
\aligned
\big\|t^{3/2}(s/t)^3\del_{\alpha}\del_{\beta}Z^{K'}u\big\|_{L^\infty(\Hcal^*_s)}
+& \big\|t^{3/2}(s/t)^3Z^{K'}\del_{\alpha}\del_{\beta}u\big\|_{L^\infty(\Hcal^*_s)}
\\
\leq& C\Ecal^H_N(s,u)^{1/2} + C\sum_{|I|+j\leq |K'|+1}\|s\del^IL^j\Box u\|_{L^2(\Hcal^*_s)}.
\endaligned
\end{equation}
\end{lemma}
%

\bibliographystyle{siam}
\bibliography{WKGm-bibtex}
\end{document}